\documentclass[11pt]{amsart}

\usepackage{amssymb,amsfonts}
\usepackage{hyperref}
\usepackage{mdwlist}

\usepackage{amssymb,textcomp}
\usepackage{enumerate}

\usepackage[utf8]{inputenc}
\usepackage[T1]{fontenc}

\usepackage[mathscr]{eucal}

\usepackage{hyperref}
 \hypersetup{
     colorlinks=true,
     linktocpage=true,
     linkcolor=red,
     filecolor=blue,
     citecolor = blue,
     urlcolor=cyan,
     }

\usepackage{xcolor}

\usepackage[a4paper, twoside=false, vmargin={2cm,3cm}, includehead]{geometry}

\usepackage{comment}
\newtheorem{lemma}{Lemma}[section]
\newtheorem{theorem}[lemma]{Theorem}
\newtheorem*{theorem*}{Theorem}
\newtheorem{claim}[lemma]{Claim}
\newtheorem{corollary}[lemma]{Corollary}
\newtheorem{question}[lemma]{Question}
\newtheorem{proposition}[lemma]{Proposition}
\newtheorem*{proposition*}{Proposition}

\newtheorem{conjecture}[lemma]{Conjecture}
\newtheorem{problem}[lemma]{Problem}
\newtheorem*{problem*}{Problem}

\theoremstyle{definition}
\newtheorem*{claim*}{Claim}

\newtheorem{definition}[lemma]{Definition}

\newtheorem{example}[lemma]{Example}
\newtheorem{examples}[lemma]{Examples}

\DeclareMathOperator*{\oE}{\overline{\mathbb{E}}}
\DeclareMathOperator*{\E}{\mathbb{E}}

\newcommand{\C}{{\mathbb C}}

\newcommand{\N}{{\mathbb N}}

\newcommand{\Q}{{\mathbb Q}}
\newcommand{\R}{{\mathbb R}}

\newcommand{\T}{{\mathbb T}}
\newcommand{\Z}{{\mathbb Z}}

\newcommand{\CA}{{\mathcal A}}
\newcommand{\CB}{{\mathcal B}}
\newcommand{\CC}{{\mathcal C}}
\newcommand{\CD}{{\mathcal D}}

\newcommand{\CF}{{\mathcal F}}
\newcommand{\CG}{{\mathcal G}}
\newcommand{\CH}{{\mathcal H}}
\newcommand{\CI}{{\mathcal I}}
\newcommand{\CK}{{\mathcal K}}

\newcommand{\CQ}{{\mathcal Q}}

\newcommand{\CX}{{\mathcal X}}
\newcommand{\CY}{{\mathcal Y}}

\newcommand{\CZ}{{\mathcal Z}}

\newcommand{\FD}{{\mathfrak D}}

\newcommand{\ba}{\mathbf{a}}

\newcommand{\bm}{{\mathbf{m}}}
\newcommand{\bn}{{\mathbf{n}}}
\newcommand{\be}{{\mathbf{e}}}
\newcommand{\bg}{{\mathbf{g}}}
\newcommand{\bh}{{\mathbf{h}}}
\newcommand{\bk}{{\mathbf{k}}}

\newcommand{\br}{{\mathbf{r}}}

\newcommand{\bu}{{\mathbf{u}}}
\newcommand{\bv}{{\mathbf{v}}}

\newcommand{\bx}{{\mathbf{x}}}

\newcommand{\bz}{{\mathbf{z}}}

\newcommand{\bJ}{{\mathbf{J}}}

\newcommand{\balpha}{{\boldsymbol{\alpha}}}

\newcommand{\bbeta}{{\boldsymbol{\beta}}}

\newcommand{\uh}{{\underline{h}}}

\newcommand{\um}{{\underline{m}}}

\newcommand{\veps}{\varepsilon}

\newcommand{\e}{\varepsilon}

\newcommand{\eps}{\epsilon}
\newcommand{\ueps}{{\underline{\epsilon}}}

\newcommand{\norm}[1]{\left\Vert #1\right\Vert}
\newcommand{\nnorm}[1]{\lvert\!|\!| #1|\!|\!\rvert}
\newcommand{\Bignnorm}[1]{\Big|\!\Big|\!\Big| #1\Big|\!\Big|\!\Big|}

\newcommand{\inv}{^{-1}}

\DeclareMathOperator{\sgn}{sgn}

\DeclareMathOperator{\Span}{Span}

\DeclareMathOperator{\extspan}{ExtSpan}
\DeclareMathOperator{\extSpan}{ExtSpan}
\DeclareMathOperator{\Nil}{Nil}

\newcommand{\abs}[1]{\mathopen{}\left| #1\mathclose{}\right|}

\newcommand{\brac}[1]{\mathopen{}\left( #1 \mathclose{}\right)}

\newcommand{\Bigbrac}[1]{\Bigl( #1 \Bigr)}
\newcommand{\floor}[1]{{\left \lfloor #1 \right \rfloor}}
\newcommand{\sfloor}[1]{{\lfloor #1 \rfloor}}

\newcommand{\rem}[1]{\left \{ #1 \right \}}

\newcommand{\srem}[1]{ \{ #1 \}}

\title[]{Resolving the joint ergodicity problem\\ for Hardy sequences}

\author{Sebasti\'an Donoso, Andreas Koutsogiannis, Borys Kuca,\\ Wenbo Sun, and Konstantinos Tsinas}

\address[Sebasti\'an Donoso]{Departamento de Ingenier\'{\i}a
	Matem\'atica and Centro de Modelamiento Ma\-te\-m\'a\-ti\-co, Universidad de Chile and IRL-CNRS 2807, Beauchef 851, Santiago,
	Chile.} 
\email{sdonosof@uchile.cl}

\address[Andreas Koutsogiannis]{
Department of Mathematics, Aristotle University of Thessaloniki, Thessaloniki 54124, Greece}
\email{akoutsogiannis@math.auth.gr}

\address[Borys Kuca]{Faculty of Mathematics and Computer Science, Jagiellonian University, 30-348 Krak\'ow, Poland}
\email{borys.kuca@uj.edu.pl}

\address[Wenbo Sun]{Department of Mathematics, Virginia Tech, 225 Stanger Street, Blacksburg, VA, 24061, USA}
\email{swenbo@vt.edu}

\address[Konstantinos Tsinas]{Institute of Mathematics, Ecole Polytechnique F\'{e}d\'{e}rale de Lausanne (EPFL), Lausanne 1015,
Switzerland}
\email{konstantinos.tsinas@epfl.ch}

\thanks{The first author was partially funded by Centro de Modelamiento Matemático (CMM) FB210005, BASAL funds for centers of excellence from ANID-Chile and ANID/Fondecyt/1241346. For the second author, the research project is implemented in the framework of H.F.R.I call ``3rd Call for H.F.R.I.’s
Research Projects to Support Faculty Members \& Researchers'' (H.F.R.I. Project Number: 24979).  The third author was supported by the NCN Polonez Bis 3 grant No. 2022/47/P/ST1/00854 (H2020 MSCA GA No. 945339). The fourth author was partially supported by the NSF Grant DMS-2247331. The fifth author was supported by the Swiss National Science Foundation grant TMSGI2-211214. For the purpose of Open Access, the authors have applied a CC-BY public copyright licence to any Author Accepted Manuscript (AAM) version arising from this submission.}

\subjclass[2020]{Primary: 37A44; Secondary: 11B30, 28D05}

\keywords{Ergodic averages, joint ergodicity, Host-Kra seminorms, Hardy fields.}

\begin{document}

\begin{abstract}
    The joint ergodicity classification problem aims to characterize those sequences which are jointly ergodic along an arbitrary dynamical system if and only if they satisfy two natural, simpler-to-verify conditions on this system. These two conditions, dubbed the difference and product ergodicity conditions, naturally arise from Berend and Bergelson's pioneering work on joint ergodicity. Elaborating on our earlier work, we investigate this problem for Hardy sequences of polynomial growth, this time without making any independence assumptions on the sequences. Our main result establishes the ``difficult’’ direction of the problem: if a Hardy family satisfies the difference and product ergodicity conditions on a given system, then it is jointly ergodic for this system. We also find that, surprisingly, the converse fails for certain pathological families of Hardy sequences, even though it holds for all ``reasonable’’ Hardy families. We conclude by suggesting potential fixes to the statement of this problem. {New ideas of independent interest developed in this paper include the structure theory of a family of factors generalizing Host-Kra and box factors; a strengthening of Tao-Ziegler's concatenation results; and the most robust extension of a seminorm smoothing argument.}
\end{abstract}

\maketitle

\tableofcontents

\section{Introduction}

In 1977, Furstenberg gave his famous dynamical proof of Szemer\'edi's theorem on arithmetic progressions, initiating a highly successful program of applying ergodic tools to problems in combinatorics and number theory \cite{Fu77}. As part of his proof, he showed that whenever $(X, \CX, \mu, T)$ is a weakly mixing measure-preserving dynamical system, the averages
\begin{align*}
    \E_{n\in[N]}T^n f_1 \cdots T^{\ell n}f_\ell
\end{align*}
(where $\E_{n\in[N]}:=\frac{1}{N}\sum\limits_{n=1}^N$ {and $[N]:=\{1,\ldots,N\}$ for $N\in \N$}) converge in $L^2(\mu)$ to the product of integrals $\int f_1\; d\mu \cdots \int f_\ell\; d\mu$ for all functions $f_1, \ldots, f_\ell\in L^\infty(\mu)$. In doing so, he proved the first result on the \textit{joint ergodicity} of multiple ergodic averages {after the mean ergodic theorem}.

\begin{definition}[Joint ergodicity]\label{D: joint ergodicity}
    Let $a_1, \ldots, a_\ell:\N\to\R$ be sequences and $(X, \CX, \mu, T_1,$ $ \ldots, T_\ell)$ be a system.\footnote{By which we mean commuting invertible measure preserving transformations $T_1, \ldots, T_\ell$ acting on a Lebesgue probability space $(X, \CX, \mu)$.} We say that $a_1, \ldots, a_\ell$ are \textit{jointly ergodic} for $(X, \CX, \mu, T_1, \ldots, T_\ell)$ (alternatively, the tuple $(T_1^{\floor{a_1(n)}}, \ldots, T_\ell^{\floor{a_\ell(n)}})_n$ is \textit{jointly ergodic} on $(X, \CX, \mu)$) if
\begin{align}\label{E: joint ergodicity}
        \lim_{N\to\infty}\norm{\E_{n\in[N]}\prod_{j=1}^\ell T_j^{\floor{a_{j}(n)}}f_j - \prod_{j=1}^\ell \int f_j\, d\mu}_{L^2(\mu)} = 0
    \end{align}
    holds for all $f_1, \ldots, f_\ell\in L^\infty(\mu)$. {For $\ell=1,$ we say that $a_1$ is \emph{ergodic} for $(X, \CX, \mu, T_1)$ (and respectively we call $(T_1^{\floor{a_1(n)}})_n$ \emph{ergodic}).}
\end{definition}

In particular, $(T^n)_n$ is ergodic if and only if $T$ is ergodic.

A few years after Furstenberg's work, Bergelson and Berend inquired about the necessary and sufficient conditions for $(T_1^n, \ldots, T_\ell^n)_n$ to be jointly ergodic, proving the following result.
\begin{theorem}[{\cite{BB84}}]
    Let $(X, \CX, \mu, T_1, \ldots, T_\ell)$ be a system. Then $(T_1^n, \ldots, T_\ell^n)_n$ is jointly ergodic if and only if the following two conditions hold:
    \begin{enumerate}
    \item $T_i T_j^{-1}$ is ergodic on $(X, \CX, \mu)$ for all distinct $i, j\in[\ell]$;
    \item $T_1\times \cdots \times T_\ell$ is ergodic on the product system $(X^\ell, \CX^{\otimes \ell}, \mu^\ell)$.
    \end{enumerate}
\end{theorem}
It is interesting to understand the extent to which the Bergelson-Berend conditions for joint ergodicity can be extended to more general multiple ergodic averages, leading to the following definitions. 
\begin{definition}\label{D: good for joint ergodicity}
    Let $a_1, \ldots, a_\ell:\N\to\R$ be sequences and $(X, \CX, \mu, T_1, \ldots, T_\ell)$ be a system. We say that on the system, the sequences $a_1, \ldots, a_\ell$ satisfy
    \begin{enumerate}
        \item\label{i: difference} the \emph{difference ergodicity condition} if the sequence $(T_i^{\floor{a_i(n)}}T_j^{-\floor{a_j(n)}})_n$ is ergodic\footnote{A sequence of transformations $(T_n)_n$ on a Lebesgue probability space $(X, \CX, \mu)$ is called \textit{ergodic} if
\[  \lim_{N\to\infty}\norm{\E_{n\in[N]} T_n f - \int f\, d\mu}_{L^2(\mu)} = 0    \]     
    for all $f\in L^\infty(\mu)$.} on $(X, \CX, \mu)$ for all distinct $i, j\in[\ell]$; and
        \item\label{i: product} the \emph{product ergodicity condition} if $(T_1^{\floor{a_1(n)}}\times \cdots \times T_\ell^{\floor{a_\ell(n)}})_n$ is ergodic on the product system $(X^\ell, \CX^{\otimes \ell}, \mu^\ell)$.
    \end{enumerate}
\end{definition}

We then have the following natural problem whose variants were previously asked in \cite{BMR20, DKS22, DKS23}.
    \begin{problem}[Joint ergodicity classification problem]\label{Pr: joint ergodicity problem}
    Classify the sequences $a_1, \ldots, a_\ell:\N\to\R$ with the following property: the sequences are jointly ergodic for a system $(X, \CX, \mu, T_1, \ldots, T_\ell)$ if and only if they simultaneously satisfy the difference and product ergodicity conditions on the system. 
    \end{problem}

Joint ergodicity has been a subject of intense scrutiny in the last decades \cite{BB84, BB86, Ber87, BeHK09, BLS16, BMR20, BS22, BS23, BFM22, CFH11, DFKS22, DKKST24, DKS22, DKS23, DKS24,  Fr10, Fr12, Fr22, Fr21, FrKr05, FrKu22b, FrKu22c, FrKu22a, KK19, Kouts21,  KoSu23, KoTs23, Ts22}. Broadly speaking, the aforementioned works circulate around three themes:
\begin{enumerate}
    \item establishing joint ergodicity and its variants for broad classes of sequences and systems;
    \item deriving easier-to-verify criteria for joint ergodicity;
    \item investigating Problem \ref{Pr: joint ergodicity problem} for various classes of sequences.
\end{enumerate}

So far, Problem \ref{Pr: joint ergodicity problem} has been examined for integer polynomials, Hardy sequences, and generalized linear functions, and the equivalence postulated by Problem \ref{Pr: joint ergodicity problem} has been verified affirmatively for the following classes of sequences:
\begin{itemize}
    \item linear functions \cite{BB84};
    \item generalized linear functions \cite{BLS16};
    \item a single integer polynomial $a_1 = \cdots = a_\ell$ \cite{DKS22} (see also \cite{DFKS22} for a slight generalization);
    \item pairwise independent integer polynomials \cite{FrKu22a};
    \item arbitrary integer polynomials \cite{FrKu22b};
    \item a single Hardy sequence $a_1 = \cdots = a_\ell$ of polynomial growth that stays logarithmically away from $\R[x]$ \cite{DKS23};
    \item pairwise independent Hardy sequence of polynomial growth \cite{DKKST24}.
\end{itemize}
Of this rich variety of results, we particularly emphasize the resolution of Problem \ref{Pr: joint ergodicity problem} for integer polynomials \cite{FrKu22b} and pairwise independent Hardy sequences of polynomial growth \cite{DKKST24}. By identifying broad classes of sequences for which Problem \ref{Pr: joint ergodicity problem} holds, the aforementioned works have given substantial credibility to why one should think of the difference and product ergodicity conditions as a right set of conditions guaranteeing joint ergodicity. 

So far, the difficult part of Problem \ref{Pr: joint ergodicity problem} has typically been to establish the converse direction, i.e. the sufficiency of the difference and product ergodicity conditions for joint ergodicity. The main result of this paper shows that the difficult direction in Problem \ref{Pr: joint ergodicity problem} holds for all Hardy sequences of polynomial growth (whose collection is denoted below by $\CH$; see Section \ref{SS: Hardy} for the formal definition). 
\begin{theorem}\label{T: main theorem}
    Suppose that $a_{1},\dots,a_{\ell}\in\mathcal{H}$ satisfy the difference and product ergodicity conditions on a system $(X, \CX, \mu, T_1, \ldots, T_\ell)$.
Then they are jointly ergodic for the system.
     \end{theorem}
We also establish the following partial converse.

\begin{definition}
    For $a\in\mathcal{H}$, we say that $a$ is an \emph{almost rational 
    polynomial} if $a-p$ converges for some $p$ in $\Q[x]$. 
    We say that $a$ \emph{stays logarithmically away from $\Q[x]$} (\emph{resp. $\R[x]$}) if $a-p\succ \log$ for all $p$ in $\Q[x]$ (resp. $\R[x]$).
\end{definition}

\begin{theorem}\label{T: easy direction}
        Suppose that $a_{1},\dots,a_{\ell}\in\mathcal{H}$ are jointly ergodic for a system $(X, \CX, \mu, T_1, $ $ \ldots, T_\ell)$. Then the following holds:
        \begin{enumerate}
            \item The sequences satisfy the product ergodicity condition on the system.
            \item Suppose additionally that for every $i\neq j$, $c_i, c_j\in\R$, the function $c_i a_i - c_j a_j$ is either an almost rational polynomial  or stays logarithmically away from $\Q[x]$. Then the sequences satisfy the difference ergodicity condition on the system.
        \end{enumerate}
\end{theorem}

As a corollary, we provide an affirmative answer to Problem \ref{Pr: joint ergodicity problem} for a large class of Hardy sequences.
\begin{corollary}\label{C:np}
    Let $a_1, \ldots, a_\ell\in\CH$ be Hardy sequences with the property that for every $i\neq j$ and $c_i, c_j\in\R$ the function $c_i a_i - c_j a_j$ is either an almost rational polynomial or it stays logarithmically away from $\Q[x]$. Then the sequences are jointly ergodic for a system $(X, \CX, \mu, T_1, \ldots, T_\ell)$ if and only  if they satisfy the product and difference ergodicity conditions on the system. 
\end{corollary}

In particular, we resolve \cite[Conjecture 6.1]{BMR20}, i.e. Problem \ref{Pr: joint ergodicity problem} {for general real and fractional polynomials}.
\begin{corollary}
    Let $\CF$ be the class of functions of the form $$a(x) = \beta_d x^{c_d} + \cdots + \beta_1 x^{c_1}$$ for $\beta_i\in\R$ and $c_i\geq 0$. Then any sequences $a_1, \ldots, a_\ell\in \CF$ are jointly ergodic for a system $(X, \CX, \mu, T_1, \ldots, T_\ell)$ if and only  if they satisfy the product and difference ergodicity conditions on the system. 
\end{corollary}

We also show that, surprisingly, joint ergodicity does not always imply the difference ergodicity condition, yielding a negative answer to Problem \ref{Pr: joint ergodicity problem} in general. The failure of the (supposedly) easy direction in Problem \ref{Pr: joint ergodicity problem} is instantiated by the following example. 
\begin{theorem}\label{T: counterexample}
    Let $(X, \CX, \mu, T)$ be a system. Then $(T^n, T^{n+\floor{\log_2 n}})_n$ is jointly ergodic if and only if $T$ is (strongly) mixing. By contrast, $(T^{\floor{\log_2 n}})_n$ is ergodic if and only the system is conjugate to a one-point system, or equivalently $\mu$ is concentrated on a single point. Consequently, if $T$ is mixing but not conjugate to a one-point system, then $(T^n, T^{n+\floor{\log_2 n}})_n$ is jointly ergodic whereas $(T^{\floor{\log_2 n}})_n$ is not ergodic.
\end{theorem}

To conclude the discussion, we remark that in this work we essentially resolve Problem~\ref{Pr: joint ergodicity problem} for polynomial and Hardy sequences of polynomial growth, i.e. classes for which we have a reduction-of-complexity argument based on van der Corput's inequality. Up until now, we have had no results at all for general real polynomials and several different Hardy sequences (the case when all of them are the same has been partially dealt with in \cite{DKS22, DKS23}).

Lastly, in Section \ref{section: repaired conjecture}, we propose several ways of amending the statement of Problem~\ref{Pr: joint ergodicity problem} for Hardy sequences that would avoid the issues posed by the counterexample for Theorem \ref{T: counterexample}. While we obtain some preliminary results towards these modified  versions of Problem \ref{Pr: joint ergodicity problem}, more research is needed in order to see if they could serve as a good fix to the statement of Problem~\ref{Pr: joint ergodicity problem}.

\subsection{Outline of the argument}
In \cite[Theorem 2.5]{FrKu22a}, Frantzikinakis and the third author derived the following simpler-to-verify criteria for joint ergodicity. 
\begin{theorem}[Joint ergodicity criteria]\label{T: joint ergodicity criteria}
Let $a_1, \ldots, a_\ell:\N\to\R$ be sequences and $(X, \CX, \mu, T_1, \ldots, T_\ell)$ be a system with each $T_1, \ldots, T_\ell$ ergodic. Then $a_1, \ldots, a_\ell$ are jointly ergodic for  $(X, \CX, \mu, T_1, \ldots, T_\ell)$ if and only if the following two conditions hold:
    \begin{enumerate}
        \item $a_1, \ldots, a_\ell$ are \emph{good for seminorm control} for the system in that there exists $s\in\N$ such that for all 1-bounded $f_1, \ldots, f_\ell\in L^\infty(\mu)$, we have
    \begin{align*}
        \lim_{N\to\infty}\norm{\E_{n\in[N]} T_1^{\floor{a_1(n)}}f_1 \cdots T_\ell^{\floor{a_\ell(n)}}f_\ell}_{L^2(\mu)} = 0
    \end{align*}
    whenever $\nnorm{f_j}_{s, T_j} = 0$ for some $j\in[\ell]$;
        \item $a_1, \ldots, a_\ell$ are \emph{good for equidistribution} for the system
        if for all eigenvalues $\lambda_1, \ldots, \lambda_\ell$ of $T_1, \ldots, T_\ell$ respectively, not all of which are 0, we have
        \begin{align}\label{E: good for equidistribution}
            \lim_{N\to\infty} \E_{n\in[N]}e(\lambda_1 \floor{a_1(n)} + \cdots + \lambda_\ell \floor{a_\ell(n)}) = 0.
        \end{align}
    \end{enumerate}
\end{theorem}
While the equidistribution properties of Hardy sequences are well-understood thanks to Boshernitzan \cite{Boshernitzan-equidistribution}, establishing seminorm control over multiple ergodic averages of interest is a highly challenging problem. 
Up to several years ago, few seminorm estimates were known for multiple ergodic averages considered in this paper. Then came two breakthroughs which dramatically changed the situation. On one hand, the fifth author significantly refined existing arguments and obtained seminorm estimates for essentially all possible (unweighted) multiple ergodic averages along Hardy sequences of a single transformation \cite{Ts22}.\footnote{In another interesting work that came out around the same time, Bergelson, Moreira, and Richter showed that one can get seminorm estimates for another large class of averages if one is willing to introduce a suitable weight \cite{BMR20}.} On the other hand, Frantzikinakis and the third author introduced a seminorm smoothing technique \cite{FrKu22b, FrKu22a} that gave seminorm estimates for multiple ergodic averages of \textit{commuting} transformations along integer polynomials. As a starting point, they used earlier box seminorm estimates obtained by the first, second and fourth author jointly with Ferr\'e-Moragues \cite{DFKS22}, based on a concatenation result of Tao and Ziegler \cite{TZ16}. 

Generalizing all the aforementioned arguments, we have recently shown \cite{DKKST24} that averages of commuting transformations along Hardy sequences can be controlled by certain seminorms that generalize the now classical Host-Kra and box seminorms. We then used a variant of the seminorm smoothing argument to obtain (quantitative) Host-Kra seminorm estimates for averages of commuting transformations along any family of pairwise independent Hardy sequences; from this we derived a number of corollaries of joint ergodicity, recurrence, and $L^2(\mu)$ convergence. In this work, we drop the assumption of pairwise independence and show that under some necessary conditions which are good enough for the purpose of proving Theorem \ref{T: main theorem}, we can get Host-Kra seminorm control for averages along not necessarily pairwise independent Hardy sequences.
\begin{theorem}\label{T: seminorm control}
    Let $a_1, \ldots, a_\ell\in\CH$ and $(X, \CX, \mu, T_1, \ldots, T_\ell)$ be a system with $T_1, \ldots, T_\ell$ ergodic. If the sequences satisfy the difference ergodicity condition, then they are good for seminorm control for the system. 
\end{theorem}
The proof of Theorem \ref{T: seminorm control} occupies a bulk of this paper. Starting with the generalized box seminorm estimates from \cite{DKKST24}, we use a new variant of the seminorm smoothing argument to obtain another collection of seminorm estimates for the averages in question. For instance, the estimates from \cite{DKKST24} allow us to control the limit of an average 
\begin{align}\label{E: average example intro}
    \E_{n\in[N]} T_1^{\sfloor{n^{3/2}}}f_1\cdot T_2^{\sfloor{n^{3/2}+\log n}}f_2\cdot T_3^{\sfloor{n^{3/2}+n\log n}}f_3\cdot T_4^{\sfloor{n^{3/2}+n\log n}}f_4
\end{align}
by a box seminorm of $f_1$ involving many copies of $T_1, T_1T_2\inv, T_1T_3\inv, T_1T_4\inv$. By contrast, our smoothing argument gives us control over (the limit of) \eqref{E: average example intro} by a seminorm of $f_1$ along many copies of $T_1, T_1T_2\inv, T_3T_4\inv$.
A priori, there is no reason why the new seminorm 
would be easier to work with than the original one. It is at this point where the difference ergodicity condition kicks in. It forces $T_1T_2\inv, T_3T_4\inv$ to be ergodic, hence the intermediate seminorm involving $T_1, T_1T_2\inv, T_3T_4\inv$ can be replaced by a proper Host-Kra seminorm along $T_1$.

To carry out the argument above, we introduce a number of tools of independent interest. In Section \ref{SS: factors}, we introduce factors related to the generalized box seminorms in much the same manner in which Host-Kra factors are related to Host-Kra seminorms (Theorem \ref{T: relative concatenation}). We also obtain a weak structure theorem for these factors, showing that they can be generated via dual functions weighted by nilsequences of appropriate step. In Section \ref{S: concatenation}, we derive a Tao-Ziegler-type concatenation result for such factors. In fact, our concatenation strengthens Tao and Ziegler's results \cite{TZ16} in the following sense. While Tao and Ziegler show for example that 
\begin{align*}
    \mathcal{Z}_{H_{1},\dots,H_{d}}\cap \mathcal{Z}_{H'_{1},\dots,H'_{d'}}\subseteq \mathcal{Z}_{(H_{i}+H'_{i'})_{i\in[d], i'\in[d']}}
\end{align*}
for any finitely generated subgroups $H_1, \ldots, H_d, H'_1, \ldots, H'_{d'}\subseteq\Z^\ell$, we show that
\begin{align*}
    \mathcal{Z}_{G_{1},\dots,G_s,H_{1},\dots,H_{d}}\cap \mathcal{Z}_{G_{1},\dots,G_s,H'_{1},\dots,H'_{d'}}\subseteq \mathcal{Z}_{G_{1},\dots,G_{s},(H_{i}+H'_{i'})_{i\in[d], i'\in[d']}}
\end{align*}
for any finitely generated subgroups $G_1, \ldots, G_s, H_1, \ldots, H_d, H'_1, \ldots, H'_{d'}\subseteq\Z^\ell$. Because we are able to fix the same set of subgroups $G_1, \ldots, G_s$ in all the factors that we compare, we refer to our concatenation result as \textit{relative concatenation}.

Relative concatenation plays a key role in the seminorm smoothing argument carried out in Section \ref{S: smoothing}. The version of the smoothing argument developed in this paper (Theorem \ref{T: new estimates general}) is more robust than previous variations from \cite{DKKST24, FrKu22b, FrKu22a}. While performing smoothing for integer polynomials \cite{FrKu22b}, Frantzikinakis and the third author had to make some a priori ergodicity assumptions on the system, and then they had to deal with intermediate averages in which functions are invariant under various transformations. Not only did keeping track of all this information throughout an inductive argument require an unwieldy formalism; even worse, the argument could not be translated to the language of generalized box seminorms needed in the setting of Hardy sequences. Hence the need for a more flexible smoothing argument that we develop in this paper. Our argument is conceptually simpler than the one from \cite{FrKu22b} in that we make no a priori assumptions on the ergodicity of the transformations, and we have to track much less information throughout the proof than in \cite{FrKu22b}. The key innovation that makes it possible is the use of relative concatenation in the smoothing argument. On top of \textit{ping} and \textit{pong}, the two traditional steps of the smoothing argument, we add a third step, where we concatenate the original seminorm controlling our average (along $T_1, T_1T_2\inv, T_1T_3\inv, T_1T_4\inv$ in \eqref{E: average example intro}) with an auxiliary one obtained in the \textit{pong} step to get an upgraded seminorm (along $T_1, T_1T_2\inv, T_3T_4\inv$ in \eqref{E: average example intro}) controlling our average. The details of this innovation are explained with an explicit example in Section \ref{SS: concatenation example}.

Sections \ref{S: seminorm comparison I} and \ref{S: seminorm comparison II} are then dedicated to extracting information from the difference ergodicity assumption. The ergodicity of $(T_i^{\floor{a_i(n)}}T_j^{-\floor{a_j(n)}})_n$ imposes strict conditions on the spectral measures of functions in $L^2(\mu)$ since by the Herglotz-Bochner theorem, a spectral measure may not have a point mass on any $(t,s)$ for which the exponential sum
\begin{align}\label{E: exp sums intro}
    \E_{n\in[N]}e(t\floor{a_i(n)}-s\floor{a_j(n)})
\end{align}
does not vanish as $N\to\infty$. The bulk of these two sections is then dedicated to a case-by-case study of exponential sums \eqref{E: exp sums intro}. While the details of the computations performed in these sections are not very illuminating, we manage to extract enough information to upgrade seminorm control obtained in smoothing to a Host-Kra seminorm control over our average. Section \ref{S: proofs of main theorems} then pieces everything together and derives the theorems stated in the introduction.

\subsection{Basic notation}\label{SS:notation}
We start with presenting basic notation used throughout the paper.

The symbols $\C, \R, \R_+, \Q, \Z, \N_0, \N$ respectively stand for the sets of complex numbers, reals, positive reals, rationals, integers, nonnegative integers, and positive integers. With $\T$, we denote the one dimensional torus, and we often identify it with $\R/\Z$ or  with $[0,1)$. For a ring $R$, we denote the collection of polynomials over $R$ in variables $x_1, \ldots, x_k$ via $R[x_1, \ldots, x_k]$. If $E\subseteq G$ is a subset of a group $G$, then $\langle E \rangle$ is the subgroup generated by $E$.

We frequently use the interval notation to mean either a real interval or its restriction to $\Z$; the meaning would typically be clear from the context. If we restrict to $\Z$, then we write $[M,N] = \{M, \ldots, N\}$ and $[N]=\{1, \ldots, |N|\}$ for integers $M\leq N$.

For a finite set $E,$ we define the average of a sequence $a:E\to \C$ as 
\[\E_{n\in E}a(n) :=\frac{1}{|E|}\sum_{n\in E}a(n).\] We will mainly use finite averages along subsets of integers of the form $[N]^s$ or $[\pm N]^s$, also known as \textit{Ces\`aro averages}.

We let $\CC z := \overline{z}$ be the complex conjugate of $z\in \C$.

For a set $E$, we define its indicator function by $1_E$. Whenever convenient, we also use the notation $1(E)$ for an event $E$.

Typically, we use underlined symbols $\uh$ to denote elements of $\R^s$ (which we think as tuples of parameters), and we employ the bold notation $\bx$ to denote elements of $\R^\ell$ (which we think as points in a space); sometimes we also use bold notation to denote $s$-tuples of points in $\R^\ell$, e.g., $\bm=(\bm_1, \ldots, \bm_s)\in(\R^s)^\ell$. We denote the integer and fractional parts of $\bx$ via
\begin{align*}
    \floor{\bx} = (\floor{x_1}, \ldots, \floor{x_\ell})\quad \textrm{and}\quad \rem{\bx} =(\rem{x_1}, \ldots, \rem{x_\ell}).
\end{align*}

We often write $\ueps\in\{0,1\}^s$ for a vector of 0s and 1s of length $s$. For $\ueps\in\{0,1\}^s$ and $\uh, \uh'\in\R^s$, we set $\ueps\cdot \uh:=\eps_1 h_1+\cdots+ \eps_s h_s$, $\abs{\uh} := |h_1|+\cdots+|h_s|$ and
\begin{gather}\label{E: h^eps}
    h_j^{\eps_j} =\begin{cases}
        h_j,\; &\eps_j =0\\
        h_j',\; &\eps_j = 1
    \end{cases}
    \quad \textrm{for}\quad j\in[s].
\end{gather}

Given a system $(X, \CX, \mu, T_1, \ldots, T_\ell)$ and $\bm=(m_{1},\dots,m_{\ell})\in\R^\ell$, we define $$T^{\floor{\bm}}:=T_1^{\floor{m_1}}\cdots T_\ell^{\floor{m_\ell}}.$$
For $j\in[\ell]$, we set $\be_j$ to be the unit vector in $\R^\ell$ in the $j$-th direction, and we let $\be_0 = \mathbf{0}$, so that $T^{\be_j} = T_j$ for $j\in[\ell]$ and $T^{\be_0}$ is the identity transformation. 

For $M>0$, we say that a function $f\in L^\infty(\mu)$ is \textit{$M$-bounded} if $\norm{f}_{L^\infty(\mu)}\leq M$.

For a $\sigma$-algebra $\CA\subseteq\CX$ and $f\in L^1(\mu)$, we let $\E(f|\CA)$ be the conditional expectation of $f$ with respect to $\CA$. By abuse of notation, we can identify the $\sigma$-algebra $\CA$ with the algebra of functions $L^2(\CA)$. One algebra that we often work with is the \textit{invariant factor}
$$\CI(T) = \{A\in \CX:\; T\inv A = A\}$$
of a transformation $T$.

Given two functions $a,b:(x_0,\infty)\to\mathbb{R}$ for some $x_0\in\R$ we write
\begin{enumerate}
    \item $a\ll b$ or $a = O(b)$ if there exists $C>0$ and $x_1\geq x_0$ such that $|a(x)| \leq C |b(x)|$ for all $x\geq x_1$; 
    \item $a \asymp b$ if $a\ll b$ and $b\ll a$;
    \item $a \prec b$ or $a = o(b)$  if $\lim\limits_{x\to\infty}\frac{a(x)}{b(x)} = 0$.
\end{enumerate}
If the absolute constant in (i) depends on some parameters $x_1, \ldots, x_n$, then we write $a\ll_{x_1, \ldots, x_n} b$, $a = O_{x_1, \ldots, x_n}(b)$, etc. If in (ii) we want to emphasize what parameter we are taking to infinity, we can also write $a = o_{x\to\infty; x_1, \ldots, x_n} (b)$ to emphasize that $x_1, \ldots, x_n$ are kept fixed while $x$ goes to infinity.

For a finitely generated abelian group $G$ and a bounded function $\psi:G\to\C$, we define
\begin{align*}
    \E_{\bm\in G}\psi(\bm) = \lim_{N\to\infty}\E_{\bm\in I_N}\psi(\bm)
\end{align*}
to be the limit along any F\o lner sequence $(I_N)_N$ on $G$ (if such a limit exists), and if $\psi$ is nonnegative, then we define
\begin{align*}
    \oE_{\bm\in G}\psi(\bm) = \sup_{(I_N)_N}\limsup_{N\to\infty}\E_{\bm\in I_N} \psi(\bm),
\end{align*}
where the supremum is taken along any F\o lner sequence on $G$. Likewise, if $(X,\CX,\mu)$ is a probability space and $f_\bm\in L^\infty(\mu)$ for all $\bm\in G$, then we define $\E_{\bm\in G}f_\bm$ to be the $L^2(\mu)$ limit (whenever it exists).

\subsection{Basic properties of Hardy sequences}\label{SS: Hardy}

Let $B$ be the collection of germs at $\infty$ of real valued functions defined on some halfline $(x_0,\infty)$ for $x_0\geq 0$ (thus functions that eventually agree are considered to be the same).  A \emph{Hardy field} is a subfield of the ring $(B, +, \cdot)$ that is closed under differentiation.
 
 For various technical reasons, we will consider Hardy fields $\mathcal{H}_0$ which are closed under composition and compositional inversion of functions, when defined. We say that $a\in\CH$ has \emph{polynomial growth} if it satisfies $a(x)\ll x^d$ for some $d>0$, and we then let
\begin{align*}
    \CH = \{a\in \CH_0:\; a(x) \ll x^d\;\; \textrm{for\; some}\;\; d>0\},
\end{align*}
which is still a Hardy field closed under compositions, but not closed under inverses.

A prime example is the Hardy field $\mathcal{LE}$ of \textit{logarithmico-exponential functions}, which consists of functions defined on a half-line $(x_0,+\infty)$ by a finite combination of symbols $+, -, \times, \div, \sqrt[n]{\cdot}, \exp, \log$ acting on the real variable $x$ and on real constants. While $\mathcal{LE}$ is not closed under compositional inverses, it is contained in the Hardy field of Pfaffian functions that is. See \cite{Fr10, Fr12, Hardy12} for more on Hardy fields in general and $\mathcal{LE}$ in particular.

\begin{theorem}[{\cite[Theorem 1.3]{Boshernitzan-equidistribution}}]\label{T: Boshernitzan}
    Let $a\in\mathcal{H}$. The sequence $(a(n))_n$ is equidistributed $\mod~\!1$ (meaning that $$\lim_{N\to\infty}\E_{n\in [N]}F(a(n)\!\!\! \mod 1)=\int_{\T} F(t)\,dt$$ for any Riemann-integrable function $F\colon\T\to\C$) if and only if for every $p\in \Q[x]$ we have \[\lim_{x\to \infty}\frac{|a(x)-p(x)|}{\log x}=\infty.\]
\end{theorem}

\begin{definition}[Independence properties of Hardy sequences]\label{D: independence}
Let $a_1, \ldots, a_\ell\in\CH$. We call them:
\begin{enumerate}
\item \textit{independent} if for any $c_1, \ldots, c_\ell\in\R$ not all 0, we have
\begin{align*}
    \lim_{x\to\infty}\frac{|c_1 a_1(x) + \cdots + c_\ell a_\ell(x)|}{\log x}=\infty,
\end{align*}
and \textit{dependent} otherwise;
\item \textit{pairwise independent} if $a_i, a_j$ are independent for any $1\leq i<j\leq \ell$, and \textit{pairwise dependent} otherwise.
 \end{enumerate}
\end{definition}

\begin{definition}[Strongly nonpolynomial Hardy sequences]
    We call $a\in\CH$ \textit{strongly nonpolynomial} if     there exists $d\in\N_0$ such that $x^d\prec a(x)\prec x^{d+1}$.
\end{definition}

 \subsection*{Acknowledgments}
 We would like to thank Nikos Frantzikinakis for helpful discussions and sharing ideas that turned into the proof of part (ii) of Lemma \ref{L: slowly growing}.

\section{Generalized box seminorms and factors}

This section discusses generalized box seminorms along finitely generated subgroups $G_1, \ldots, G_s\subseteq\R^\ell$ and associated factors, all of which extend the classical Host-Kra and box seminorms and factors defined for subgroups of $\Z^\ell$ (see e.g. \cite{DFKS22, H09, HK05a, HK18, TZ16}). The generalized box seminorms have been defined in our earlier paper \cite{DKKST24}, in which we proved their main properties; the notion of factors is new. For the duration of the entire section, we fix a system $(X, \CX, \mu, T_1, \ldots, T_\ell)$.

\subsection{Generalized box seminorms}

Let $G_1, \ldots, G_s\subseteq\R^\ell$ be finitely generated subgroups. For $f\in L^\infty(\mu)$ and $\bm,\bm'\in\R^\ell$, let $\Delta_{(\bm,\bm')}f = T^{\sfloor{\bm}}f\cdot T^{\sfloor{\bm'}}\overline{f}$. Recall from \cite{DKKST24} that we can define two seminorms of $f$ along $G_1, \ldots, G_s$. The first one is given by
\begin{align*}
    \nnorm{f}_{\substack{G_1, \ldots, G_s}} &= \brac{\E_{\bm_1, \bm_1'\in G_1}\cdots \E_{\bm_s, \bm_s'\in G_s}\int \Delta_{({\bm_1}, {\bm'_1}), \ldots, ({\bm_s}, {\bm'_s})}f\, d\mu}^{1/2^{s}}\\
    &= \brac{\E_{\bm_1, \bm_1'\in G_1}\cdots \E_{\bm_s, \bm_s'\in G_s}\int \prod_{\ueps\in\{0,1\}^s}\CC^{|\ueps|}T^{\sfloor{\bm_1^{\eps_1}}+\cdots + \sfloor{\bm_s^{\eps_s}}} f\, d\mu}^{1/2^{s}},
\end{align*}
while the second one is 
\begin{align*}
    \nnorm{f}^+_{\substack{G_1, \ldots, G_s}} &:=  \brac{\E_{\bm_1, \bm_1'\in G_1}\cdots \E_{\bm_s, \bm_s'\in G_s}\abs{\int \Delta_{({\bm_1}, {\bm_1'}), \ldots, ({\bm_s}, {\bm_s'})}f\, d\mu}^2}^{1/2^{s+1}}\\
      &= \brac{\E_{\bm_1, \bm_1'\in G_1}\cdots \E_{\bm_s, \bm_s'\in G_s}\int \Delta_{({\bm_1}, {\bm_1'}), \ldots, ({\bm_s}, {\bm_s'})}f\otimes\overline{f}\, d(\mu\times\mu)}^{1/2^{s+1}}\\
      &=\nnorm{f\otimes\overline{f}}_{G_1, \ldots, G_s}^{1/2}.
\end{align*}
See \eqref{E: h^eps} for the convention regarding expressions like $T^{\sfloor{\bm_1^{\eps_1}}}$.

If $G_i = \langle \bv_i\rangle$, we also write 
\begin{align*}
    \nnorm{f}_{\substack{G_1, \ldots, G_s}} = \nnorm{f}_{\substack{\bv_1, \ldots, \bv_s}}\quad \textrm{and}\quad \nnorm{f}^+_{\substack{G_1, \ldots, G_s}} = \nnorm{f}^+_{\substack{\bv_1, \ldots, \bv_s}},
\end{align*}
and if $G_1 = \cdots = G_s = G$, then
\begin{align*}
    \nnorm{f}_{\substack{G_1, \ldots, G_s}} = \nnorm{f}_{G^{\times s}}\quad \textrm{and}\quad \nnorm{f}^+_{\substack{G_1, \ldots, G_s}} = \nnorm{f}^+_{G^{\times s}}.
\end{align*}

For instance, if $\ell=s=2$ and $G_i = \langle \alpha_i \be_i\rangle$, then 
\begin{align*}
    \nnorm{f}_{\substack{G_1, G_2}}^4= \E\limits_{m_{1},m'_{1}\in\Z}\E\limits_{m_{2},m'_{2}\in\Z} &\int T_1^{\sfloor{\alpha_1 m_1}}T_2^{\sfloor{\alpha_2 m_2}} f \cdot T_1^{\sfloor{\alpha_1 m_1}}T_2^{\sfloor{\alpha_2 m'_2}}\overline{f}\\
        &\qquad \cdot T_1^{\sfloor{\alpha_1 m'_1}}T_2^{\sfloor{\alpha_2 m_2}}\overline{f}\cdot T_1^{\sfloor{\alpha_1 m_1'}}T_2^{\sfloor{\alpha_2 m_2'}}f\, d\mu
\end{align*}
and
\begin{align*}
    \Bigbrac{\nnorm{f}^+_{\substack{G_1, G_2}}}^8= \E\limits_{m_{1},m'_{1}\in\Z}\E\limits_{m_{2},m'_{2}\in\Z} &\left|\int T_1^{\sfloor{\alpha_1 m_1}}T_2^{\sfloor{\alpha_2 m_2}} f \cdot T_1^{\sfloor{\alpha_1 m_1}}T_2^{\sfloor{\alpha_2 m'_2}}\overline{f}\right.\\
        &\qquad \cdot\left. T_1^{\sfloor{\alpha_1 m'_1}}T_2^{\sfloor{\alpha_2 m_2}}\overline{f}\cdot T_1^{\sfloor{\alpha_1 m_1'}}T_2^{\sfloor{\alpha_2 m_2'}}f\, d\mu\right|^2.
\end{align*}

The following properties of generalized box seminorms have all been established in \cite[Section 3]{DKKST24}.
\begin{lemma}[Basic properties of seminorms]\label{L: basic properties of seminorms}
    Let $G_1, \ldots, G_s\subseteq\R^\ell$ be finitely generated subgroups and $f\in L^\infty(\mu)$. Then: 
    \begin{enumerate}
        \item (Symmetry) For any permutation $\sigma:[s]\to[s]$, we have
        \begin{align*}
            \nnorm{f}_{G_1, \ldots, G_s} = \nnorm{f}_{G_{\sigma(1)}, \ldots, G_{\sigma(s)}}\quad \textrm{and}\quad \nnorm{f}_{G_1, \ldots, G_s}^+ = \nnorm{f}_{G_{\sigma(1)}, \ldots, G_{\sigma(s)}}^+.
        \end{align*}
        \item (Monotonicity) For any finitely generated $G_{s+1}\subseteq\R^\ell$, we have
        $$\nnorm{f}_{G_1, \ldots, G_s}\leq\nnorm{f}_{G_1, \ldots, G_s}^+\leq\nnorm{f}_{G_1, \ldots, G_{s+1}}\leq\nnorm{f}_{G_1, \ldots, G_{s+1}}^+.$$
        \item (Subgroup property) If $G'_i\subseteq G_i$ for all $i\in[s]$, then $$\nnorm{f}_{G_1, \ldots, G_s}\leq\nnorm{f}_{G'_1, \ldots, G'_s} \quad \textrm{and}\quad \nnorm{f}^+_{G_1, \ldots, G_s}\leq\nnorm{f}_{G'_1, \ldots, G'_s}^+.$$ Conversely,  if the index of $G'_{i}$ in $G_{i}$ is at most $D$ for all $i\in[s]$, then 
\begin{align*}
\nnorm{f}_{G'_1, \ldots, G'_s}^+\ll_{D,s} \nnorm{f}_{G_1, \ldots, G_s}^+.
\end{align*} 
        \item (Rescaling property) Let $K>0$. If for every $i\in[s]$ we have
        \begin{align*}
            G_i = \langle \bg_{i1}, \ldots, \bg_{iK_i}\rangle\quad \textrm{and}\quad G'_i = \langle \alpha_{i1}\bg_{i1}, \ldots, \alpha_{iK_i} \bg_{iK_i}\rangle 
        \end{align*}
        for some $K_i\leq K$ and nonzero real $1/K\leq |\alpha_{ij}|\leq K$, then
    \begin{align*}
        \nnorm{f}^+_{G'_1, \ldots, G'_s}\asymp_{\ell, K, s} \nnorm{f}^+_{G_1, \ldots, G_s}
    \end{align*}
    and
    \begin{align*}
        \nnorm{f}_{G'_1, \ldots, G'_s}\asymp_{\ell, K, s} \nnorm{f}_{G_1, \ldots, G_s} \quad \textrm{(if\; additionally}\; s\geq 2).
    \end{align*}
    \item (Equality of ergodic averages) For all finitely generated subgroups $G_1, \ldots, G_s~\subseteq~\R^\ell$ for which  $\E\limits_{\bm\in G_i} T^{\floor{\bm}}f = \E\limits_{\bm\in G'_i} T^{\floor{\bm}}f$ holds for all $f\in L^\infty(\mu)$ and $i\in[s]$, we have
    \begin{align*}
        \nnorm{f}_{G_1, \ldots, G_s}=\nnorm{f}_{G'_1, \ldots, G'_s}.
    \end{align*}
    \end{enumerate}
\end{lemma}

\subsection{Generalized box factors}\label{SS: factors}

Our goal now is to construct factors $\CZ^+_{G_1, \ldots, G_s}$ with the property that
\begin{align}\label{E: seminorm vs factor}
    \nnorm{f}_{G_1, \ldots, G_s}^+ = 0 \quad \Longleftrightarrow \quad \E(f|\CZ_{G_1, \ldots, G_s}^+) = 0.
\end{align}
If $G_1, \ldots, G_s$ lie inside $\Z^\ell$, then 
there exist factors $\CZ_{G_1, \ldots, G_s}$ with the property
\begin{align}\label{E: seminorm vs factor unplussed}
    \nnorm{f}_{G_1, \ldots, G_s} = 0 \quad \Longleftrightarrow \quad \E(f|\CZ_{G_1, \ldots, G_s}) = 0
\end{align}
(see for example \cite{H09, HK05a,S21,TZ16}).
We however do not know if there exists a suitable factor satisfying (\ref{E: seminorm vs factor unplussed}) whenever $G_1, \ldots, G_s$ do not lie inside $\Z^\ell$. By the results in \cite{DKS20}, there exists a factor corresponding to the seminorm $\nnorm{\cdot}_{\langle\alpha\rangle}$ for a system $(X,\mathcal{X},\mu,T)$ and $\alpha\in [0,1)$ (this factor is spanned by eigenfunctions with eigenvalues coming from the subgroup $(\frac{1}{\alpha}\Z)/\Z$). While this limited evidence suggests a possibility of the existence of a factor satisfying (\ref{E: seminorm vs factor unplussed}) for general finitely generated subgroups $G_1, \ldots, G_s\subseteq\R^\ell$, we do not require such factors in our arguments and hence do not pursue this question in this paper. We refer the readers to the discussion below Definition \ref{D: factor} for an explanation of difficulties inherent in constructing a factor $\CZ_{G_1, \ldots, G_s}$ even when $s=1$.


Following e.g. \cite{BHK05, FH18, GT12}, we define multiparameter nilsequences. 

\begin{definition}[Multiparameter nilsequences]
    A \textit{degree-$s$ $\Z^{K}$-nilsequence} is a bounded sequence  $\psi:\Z^{K}\to\C$, where: 
    \begin{itemize}
        \item $G/\Gamma$ is a nilmanifold (with $G$ connected and simply connected) with a degree-$s$ filtration\footnote{We recall that a filtration $G_\bullet = (G_i)_{i=0}^\infty$ on a group $G$ is a collection of subgroups $G = G_0=G_1\supseteq G_2 \supseteq \cdots$ satisfying $[G_i, G_j]\subseteq G_{i+j}$ for all $i,j\in\N_0$, where $[A,B]$ is the group generated by the commutators $[a,b]=a\inv b\inv a b$ for $a\in A, b\in B$.} $G_\bullet = (G_i)_{i=0}^\infty$;
        \item $g:\Z^{K}\to G$ is a polynomial sequence adapted to the filtration $G_\bullet$, meaning that $$\partial_{h_{i}}\dots\partial_{h_{1}}g(n)\in G_{i}$$
        for all $i\in\N$ and $n,h_{1},\dots,h_{i}\in\Z^{K}$ (here, $\partial_h g(n) = g(n)g(n+h)\inv$); 
        \item $\psi(\um) = F(g(\um)\Gamma)$ for a bounded Riemann integrable function $F:G/\Gamma\to\C$.
    \end{itemize}
\end{definition}


Our next goal is to obtain a suitable notion of nilsequence on general finitely generated subgroups $G\subseteq\R^\ell$. To this end, we need to map $G$ to some $\Z^K$. We note that by the fundamental theorem of finitely generated abelian groups, $G$ is always group-isomorphic to $\Z^K$ for some $K\in\N_0$.
\begin{definition}[Coordinate maps]\label{D: coordinate maps}
    Let $G\subseteq\R^\ell$ be a finitely generated subgroup with rank $K$ (which we recall is the cardinality of the largest linearly independent subset of $G$). We then call a group isomorphism $\phi:\Z^K\to G$ a \textit{coordinate map}.
\end{definition}

\begin{examples}

Here are some examples of coordinate maps:
\begin{enumerate}
    \item if $G = \langle (\sqrt{2}, \sqrt{3}) \rangle\subseteq\R^2$, then $\phi(m) = (\sqrt{2} m, \sqrt{3} m)$ is a coordinate map;
    \item if $G = \langle \sqrt{2}, \sqrt{3} \rangle\subseteq\R$, then $\phi(m_1,m_2) = \sqrt{2} m_1 + \sqrt{3} m_2$ is a  coordinate map. 
\end{enumerate}    
\end{examples}


With the help of coordinate maps, the notion of nilsequences can be extended to general finitely generated subgroups $G\subseteq\R^\ell$. In the entire ensuing discussion, $K$ denotes the rank of $G$.
\begin{definition}[Nilsequences on $G$]
Let $G\subseteq\R^\ell$ be a finitely generated subgroup. Then a sequence $\psi:G\to\C$ is a \textit{degree-$s$ nilsequence on $G$} if there exists a coordinate map $\phi:\Z^K\to G$ 
and a  degree-$s$ $\Z^{K}$-nilsequence $\theta\colon \Z^{K}\to \C$ such that $\theta=\psi\circ \phi$.
 We then define $\Nil_s(G)$ to be the collection of all degree-$s$ nilsequences on~$G$.
\end{definition}

\begin{proposition}\label{P: same coordinate map}
Let $G\subseteq\R^\ell$ be a finitely generated subgroup. If $\psi$ is a degree-$s$ nilsequence on $G$ with respect to some coordinate map, then it is a degree-$s$ nilsequence on $G$ with respect to any coordinate map.
\end{proposition}
\begin{proof}
    Let $\psi\colon G\to\C$ be a map and $\phi_{1},\phi_{2}\colon\Z^{K}\to G$ be coordinate maps. Suppose that $\psi\circ \phi_{1}=\theta_{1}$ for some degree-$s$ $\Z^{K}$-nilsequence $\theta_{1}\colon\Z^{K}\to \C$. It suffices to show that  $\psi\circ \phi_{2}=\theta_{2}$ for some  degree-$s$ $\Z^{K}$-nilsequence $\theta_{2}\colon\Z^{K}\to \C$.

    Note that for all $n\in \Z^K$, we have
    $$\psi\circ\phi_{2}(n)=\psi\circ(\phi_{1}\circ\phi_{1}^{-1})\circ\phi_{2}(n)=(\psi\circ\phi_{1})\circ(\phi_{1}^{-1}\circ\phi_{2})(n)=\theta_1\circ(\phi_{1}^{-1}\circ\phi_{2})(n).$$
  Because $\phi:=\phi_{1}^{-1}\circ\phi_{2}$ is a group automorphism of $\Z^{K}$, it suffices to show that degree-$s$ $\Z^{K}$-nilsequences are closed under compositions with group automorphisms of $\Z^k$.
    
    To see this, assume that $\theta_1=F(g(\cdot)\Gamma)$. Then $\theta_1\circ \phi=F(g\circ\phi(\cdot)\Gamma)$. Since $\phi:\Z^{K}\to\Z^{K}$ is a group isomorphism, we have that $$\partial_{h_{i}}\dots \partial_{h_{1}}(g\circ \phi)(n)=\partial_{\phi(h_{i})}\dots \partial_{\phi(h_{1})}g(\phi(n))$$ for all $i\in\N$ and $n,h_{1},\dots,h_{i}\in\Z^{K}$. This implies that $g\circ \phi$  is a polynomial sequence with the same filtration as $g$. So $\theta_2:=\theta_1\circ \phi$ is a  degree-$s$ $\Z^{K}$-nilsequence and we are done.
\end{proof}

As an immediate consequence, we have the following.

\begin{corollary}\label{C: algebra}
   Let $G\subseteq\R^\ell$ be a finitely generated subgroup. Then $\Nil_s(G)$ is a conjugation-closed $\C$-algebra.
\end{corollary}
\begin{proof}
    Let $\psi_1, \psi_2\in\Nil_s(G)$. By the definition of $\Nil_s(G)$ and Proposition \ref{P: same coordinate map}, we can find a coordinate map $\phi:\Z^K\to G$ and degree-$s$ $\Z^{K}$-nilsequences $\theta_1, \theta_2\colon \Z^{K}\to \C$ such that $\theta_i=\psi_i\circ \phi$ for $i\in\{1,2\}$. Then for any $c\in\C$, the functions $\overline{\psi_1}$, $\psi_1 + c\psi_2$, and $\psi_1\psi_2$ are elements of $\Nil_s(G)$ corresponding to degree-$s$ $\Z^{K}$-nilsequences $\overline{\theta_1}$, $\theta_1 + c\theta_2$, and $\theta_1\theta_2$. This proved the claim.
\end{proof}

The following notion will play a crucial part in the forthcoming definition of the factors. 
\begin{definition}[Dual functions]\label{D: dual functions}
    Set $G = G_1\times G_1 \times \cdots \times G_s\times G_s$, and let $\psi\in\Nil_s(G)$.
Let $\{0,1\}^s_* = \{0,1\}^s\setminus\{\underline{0}\}$ and assume that $f_\ueps\in L^\infty(\mu)$ for $\ueps\in \{0,1\}^s_*$.
We then call a function 
\begin{multline}\label{E: dual function}
    \CD_{G_1, \ldots, G_s}((f_\ueps)_{\{0,1\}^s_*}; \psi) := \E_{(\bm_1, \bm_1', \ldots, \bm_s, \bm_{s'})\in G}\; \psi(\bm_1,\bm_1',\ldots,\bm_s,\bm_s')\\
    \prod_{\ueps\in\{0,1\}^s_*}T^{\sfloor{\bm_1^{\eps_1}}-\floor{\bm_1}+\cdots + \sfloor{\bm_s^{\eps_s}}-\floor{\bm_s}} f_\ueps
\end{multline}
a \textit{level-$s$ dual function of $f$ along $G_1, \ldots, G_s$}. 

In a number of cases, we simplify the notation:
\begin{itemize}
    \item if $f_\ueps = \CC^{|\ueps|}f$ for some $f$, then we also denote the dual function as $\CD_{G_1, \ldots, G_s}(f; \psi)$;
    \item if $\psi = 1$, then we write $\CD_{G_1, \ldots, G_s}((f_\ueps)_{\{0,1\}^s_*})$;
    \item if $G_1 = \cdots = G_s$, then we use the label $\CD_{s, G_1}((f_\ueps)_{\{0,1\}^s_*}; \psi)$, and if all the groups equal $\langle \be_j \rangle$ for some $j\in[\ell]$, then we further write $\CD_{s, T_j}((f_\ueps)_{\{0,1\}^s_*}; \psi)$.
\end{itemize}
We also combine all the following conventions freely, e.g. denoting the dual function by $\CD_{s, T_j}(f)$ when $G_1 = \cdots = G_s = \langle \be_j \rangle$, $f_\ueps = \CC^{|\ueps|}f$, and $\psi = 1$.
\end{definition}


\begin{lemma}
    The limit \eqref{E: dual function} exists in $L^2(\mu)$. Moreover, the single limit along $G$ can be replaced by an iterated limit along $G_1\times G_1$, ..., $G_s\times G_s$.
\end{lemma}
\begin{proof}
    Let $\phi:\Z^K\to G$ be a coordinate map and $(I_N)_N$ be a F\o lner sequence on $G$. Define
    \begin{align*}
         F(\bm) = \prod_{\ueps\in\{0,1\}^s_*}T^{\sfloor{\bm_1^{\eps_1}}-\floor{\bm_1}+\cdots + \sfloor{\bm_s^{\eps_s}}-\floor{\bm_s}} f_\ueps
     \end{align*}
     for $\bm\in G$.
    Recasting the original average over $G$ as an average over $\Z^K$, it suffices to show that the average
    \begin{align}\label{E: dual over Z^K}
       \E_{{\um}\in \phi\inv(I_N)}\; \psi(\phi(\um)) \cdot  F(\phi(\um))
    \end{align}
    converges as $N\to\infty$ to the limit independent of the choice of $(I_N)_N$. Note that in \eqref{E: dual over Z^K}, $\psi(\bm)$ is replaced by a $\Z^{K}$-nilsequence $\psi\circ\phi$ and iterates $T^{\floor{\bm_i}}$ inside $F$ are replaced by expressions like $T^{\floor{\phi(\um_i)}}=T^{\sfloor{\balpha_1 m_{i,1} + \cdots + \balpha_{K_i} m_{i, K_i}}}$, where $G_i = \langle \balpha_1, \ldots, \balpha_{K_i}\rangle$.

    From then on, the proof of the $L^2(\mu)$ convergence is completely analogous to the arguments in \cite[Section 3]{Ko18}. 
    The multiparameter variant of the suspension flow trick \cite[Theorem~3.2]{Ko18} combined with the celebrated convergence result of Walsh \cite{W12} allows us to obtain the $L^2(\mu)$ convergence of \eqref{E: dual over Z^K} whenever $\psi= 1$.  
    To treat the general form of \eqref{E: dual over Z^K}, it suffices to approximate the nilsequence $\psi$ in uniform density with the special weights considered in \cite[Proposition~4.2]{FH18}.\footnote{The paper \cite{FH18} defines nilsequences in terms ``linear sequences'' $\tau_1^{n_1}\cdots \tau_K^{n_K}$ rather than general polynomial sequences, and so here we also implicitly use the argument from \cite[Section 2.11]{L05c} to convert the original ``polynomial nilsequence'' into a ``linear nilsequence'' as defined in \cite{FH18}.} 
    The claim then follows by arguments similar to those in \cite[Section~3]{Ko18}.


    The convergence of the iterated limit can be proved analogously, and the equality of a single and iterated limits can be established using the Fubini principle for ergodic averages \cite[Lemma 1.1]{BL15} much like in \cite[Lemma 3.1]{DKKST24}.
    %
\end{proof}


Dual functions play a crucial role in the following definition of the closed space of functions ``structured'' with respect to the groups $G_1, \ldots, G_s$.
\begin{definition}\label{D: factor}
    Let $G_1, \ldots, G_s\subseteq\R^\ell$ be finitely generated subgroups.
    Given a factor $\CY\subseteq \CX$, let $Z_{G_1, \ldots, G_s}^+(\CY)\subseteq L^\infty(\mu)$ be the $L^2(\mu)$-closed linear span of functions of the form 
$\CD_{G_1, \ldots, G_s}((f_\ueps)_{\{0,1\}^s_*}; \psi)$ for $f_\ueps\in L^\infty(\CY, \mu)$ and $\psi\in\Nil_s(G)$
(where we recall that $G = G_1\times G_1\times\ldots \times G_s\times G_s$). If $\CY = \CX$, then we simply write $Z_{G_1, \ldots, G_s}^+(\CX) = Z_{G_1, \ldots, G_s}^+$.

\end{definition}

This is a good place to explain the presence of a nilsequence $\psi$ in the definition \eqref{E: dual function} as well as the reason why we consider a factor associated with the seminorm $\nnorm{\cdot}_{G_1, \ldots, G_s}^+$ rather than $\nnorm{\cdot}_{G_1, \ldots, G_s}$. To this end, consider the group ${G_1} = \langle (\alpha, \beta)\rangle\subseteq\R^2$ for some irrational $\alpha,\beta\in\R$. If we wanted to find an algebra $Z_{G_1}\subseteq L^\infty(\mu)$ with the property \eqref{E: seminorm vs factor unplussed}, then it would have to be the closed linear span of functions of the form
\begin{align*}
    \CD_{G_1}(f) = \E_{m,m'\in\Z} T_1^{\sfloor{\alpha m'}-\sfloor{\alpha m}}T_2^{\sfloor{\beta m'}-\sfloor{\beta m}}f
\end{align*}
(with no $1$-step-nilsequence weight). However, it is not clear whether the resulting vector space is an algebra (i.e. closed under multiplication). The reason is that the product $\CD_{G_1}(f)\cdot\CD_{G_1}(g)$ equals
\begin{multline*}
    \CD_{G_1}(f)\cdot\CD_{G_1}(g) = \sum_{r_1,r_2=-1}^1\E_{n,n'\in\Z}\E_{m,m'\in\Z}\psi_{r_1,r_2}(m,m',n,n')\\
    T_1^{\sfloor{\alpha m'}-\sfloor{\alpha m}}T_2^{\sfloor{\beta m'}-\sfloor{\beta m}}(f\cdot T_1^{\sfloor{\alpha n'}-\sfloor{\alpha n}+r_1}T_2^{\sfloor{\beta n'}-\sfloor{\beta n}+r_2}g)
\end{multline*}
for some nontrivial weights $\psi_{r_1,r_2}$ that take care of the error terms coming from comparing expressions like $\sfloor{\alpha (m+n)}$ with $\sfloor{\alpha m} + \sfloor{\alpha n}$. It is unclear whether these weights, which can easily be seen to be degree-1 nilsequences, can always be removed. 

Once we allow for the presence of a degree-1 nilsequence weight in the definition of $\CD_{G_1}(f)$, then we get an algebra $Z_{G_1}^+$ corresponding to the seminorm $\nnorm{\cdot}_{{G_1}}^+$ rather than $\nnorm{\cdot}_{{G_1}}$ (see Lemma~\ref{L: algebra} and Proposition~\ref{P: structure theorem} below). More generally, if we want to find a factor associated with the seminorm $\nnorm{\cdot}_{G_1, \ldots, G_s}^+$, then our definition of dual functions needs to include more complicated weights that turn out to be well-approximable by degree-$s$ nilsequences.

\begin{lemma}\label{L: algebra}
    Let $G_1, \ldots, G_s\subseteq\R^\ell$ be finitely generated subgroups. Then $Z_{G_1, \ldots, G_s}^+$ is a closed $\C$-algebra invariant under $T_1, \ldots, T_\ell$ and closed under conjugation.
\end{lemma}
\begin{proof}
    Recall that $Z_{G_1, \ldots, G_s}^+\subseteq L^\infty(\mu)$ is defined to be an $L^2(\mu)$-closed linear span. It is easy to observe from the definition that it is invariant under $T_1, \ldots, T_\ell$ and closed under conjugation. The only property that remains is that it is closed under multiplication. That is, given
    \begin{align*}
D_1 &=\E_{(\bm, \bm')\in G}\; \psi_1(\bm,\bm')
     \prod_{\ueps\in\{0,1\}^s_*}T^{\sfloor{\bm_1^{\eps_1}}-\floor{\bm_1}+\cdots + \sfloor{\bm_s^{\eps_s}}-\floor{\bm_s}} f_{1,\ueps}\\
    D_2 &=\E_{(\bn, \bn')\in G}\; \psi_2(\bn, \bn')
     \prod_{\ueps\in\{0,1\}^s_*}T^{\sfloor{\bn_1^{\eps_1}}-\floor{\bn_1}+\cdots + \sfloor{\bn_s^{\eps_s}}-\floor{\bn_s}} f_{2,\ueps}, 
    \end{align*}
    we want to show that $D_1\cdot D_2$ is also in $Z_{G_1, \ldots, G_s}^+$ (above, we set $G = G_1\times G_1 \times \cdots \times G_s\times G_s$, $(\bm,\bm') = (\bm_1, \bm_1',\ldots, \bm_s, \bm_s')$ and $(\bn,\bn') = (\bn_1, \bn_1',\ldots, \bn_s, \bn_s')$, and we take $f_{j,\ueps}$ and $\psi_j$ to be 1-bounded).
    We note first that the averages converge along any F\o lner sequence, and so $D_2$ will converge to the same limit if we replace all $\bn_i$ and $\bn_i'$ by $\bm_i+\bn_i$ and $\bm_i'+\bn_i'$ respectively for any $\bn_i,\bn_i'\in G_i$. 
    Multiplying the two limits together, we therefore get
    \begin{multline*}
        D_1\cdot D_2 = \sum_{\br\in \{-1,0,1\}^{\ell s}}\E_{(\bm, \bm')\in G}  \E_{(\bn, \bn')\in G} \psi_1(\bm,\bm')\psi_2(\bm+\bn, \bm'+\bn')1_{E_\br}(\bm,\bm',\bn,\bn')\\
        \prod_{\ueps\in\{0,1\}^s_*}T^{\sfloor{\bm_1^{\eps_1}}-\floor{\bm_1}+\cdots + \sfloor{\bm_s^{\eps_s}}-\floor{\bm_s}} (f_{1,\ueps} \cdot T^{\sfloor{\bn_1^{\eps_1}} - \sfloor{\bn_1} + \cdots + \sfloor{\bn_s^{\eps_s}} - \sfloor{\bn_s}+ \ueps\cdot \br}f_{2,\ueps}),
    \end{multline*}
    where 
    $\ueps\cdot \br = \eps_1\cdot \br_1+\cdots + \eps_s\cdot\br_s$, and
    \begin{multline*}
        E_{\br} = \{(\bm,\bm',\bn,\bn')\in G\times G:\; \sfloor{\bm_i^{\eps_i}+\bn_i^{\eps_i}}-\sfloor{\bm_i + \bn_i}\\
        = \sfloor{\bm_i^{\eps_i}}-\sfloor{\bm_i}+\sfloor{\bn_i^{\eps_i}} - \sfloor{\bn_i} + \eps_i\cdot \br_{i}\; \textrm{for\; all}\; i, \eps_i\}.
    \end{multline*}

    Swapping the order of the averages along $(\bm, \bm')$ and $(\bn, \bn')$ {by the Fubini principle for ergodic averages \cite[Lemma 1.1]{BL15}}, the limit will follow if we can establish that the sequence
    \begin{align*}
        \psi_{\br, \bn,\bn'}(\bm,\bm') = \psi_1(\bm,\bm')\psi_2(\bm+\bn, \bm'+\bn')1_{E_\br}(\bm,\bm',\bn,\bn')
    \end{align*}
     is in $\Nil_s(G)$.  Since $\psi_1, \psi_2\in\Nil_s(G)$ by assumption, 
     the result will follow from Corollary~\ref{C: algebra} once we show that $(\bm,\bm')\mapsto 1_{E_\br}(\bm,\bm',\bn,\bn')\in\Nil_1(G)$. 
     Such a sequence is a product of sequences coming from the constraints
    \begin{align*}
        \sfloor{\bm_i^{\eps_i}+\bn_i^{\eps_i}}-\sfloor{\bm_i + \bn_i}
        = \sfloor{\bm_i^{\eps_i}}-\sfloor{\bm_i}+\sfloor{\bn_i^{\eps_i}} - \sfloor{\bn_i} + \eps_i\cdot \br_{i},
    \end{align*}
    and so it suffices to show that the indicator function of the set of $(\bm_i,\bm'_i)$ satisfying such a constraint for fixed $i,\eps_i,\bn,\bn'$ is a nilsequence. A standard argument shows that such an indicator function takes the form $F_{\bn_i,\bn_i'}(\srem{\bm_i}, \srem{\bm_i'})$ for some 1-bounded Riemann integrable function $F_{\bn_i,\bn_i'}:[0,1)^{2\ell}\to\R$, and hence it is indeed a degree-1 nilsequence. 
    \end{proof}

Using a standard identification between factors of a dynamical system and closed invariant subalgebras of its Hilbert space, we obtain the following corollary.
\begin{corollary}
Let $G_1, \ldots, G_s\subseteq\R^\ell$ be finitely generated subgroups. There exists a factor $\CZ_{G_1, \ldots, G_s}^+\subseteq \CX$ such that $L^\infty(\CZ_{G_1, \ldots, G_s}^+) = Z_{G_1, \ldots, G_s}^+$.
\end{corollary}
Whenever $G_i = \langle \bv_i\rangle$, we also write $\CZ^+_{G_1, \ldots, G_s} = \CZ^+_{\bv_1, \ldots, \bv_s}$, and if $G_1 = \cdots = G_s = G$, we additionally abbreviate $\CZ^+_{G_1, \ldots, G_s} = \CZ^+_{G^{\times s}}$.


The next result relates the factor $\CZ_{G_1, \ldots, G_s}^+$ to the seminorm $\nnorm{\cdot}_{G_1, \ldots, G_s}^+$.
\begin{proposition}[Weak structure theorem for $\nnorm{\cdot}_{G_1, \ldots, G_s}^+$]\label{P: structure theorem}
    Let $G_1, \ldots, G_s\subseteq\R^\ell$ be finitely generated subgroups. For any $f\in L^\infty(\mu)$, the relationship \eqref{E: seminorm vs factor} holds.
\end{proposition}

Before we prove Proposition \ref{P: structure theorem}, we need to understand the structure of multicorrelation sequences appearing in the definition of generalized box seminorms.

\begin{definition}[Nullsequences on $G$]
    Let $G\subseteq\R^\ell$ be a finitely generated subgroup. For $\veps>0$, we call $\psi: G\to\C$ an \textit{$\veps$-nullsequence on $G$} if
    \begin{align*}
        \oE_{\bm\in G}|\psi(\bm)|^2\leq \veps.
    \end{align*}
\end{definition}

\begin{proposition}[Decomposition of multicorrelation sequences]\label{P: multicorrelation}
    Let $G_1, \ldots, G_s\subseteq\R^\ell$ be finitely generated subgroups and set $G = G_1\times G_1 \times \cdots \times G_s\times G_s$. For $(\bm,\bm')\in G$ and  1-bounded functions $f_\ueps\in L^\infty(\mu)$, define 
    \begin{align*}
        \psi(\bm, \bm')= \int \prod_{\ueps\in\{0,1\}^s}T^{\sfloor{\bm_1^{\eps_1}}+\cdots + \sfloor{\bm_s^{\eps_s}}} f_\ueps\; d\mu.
    \end{align*}
    Then for every $\veps>0$, we can decompose $\psi = \psi_1 + \psi_2$, where $\psi_1$ is a 1-bounded degree-$s$ nilsequence on $G$ and $\psi_2$ is an $\veps$-nullsequence on $G$.
\end{proposition}
\begin{proof}
   Proposition \ref{P: multicorrelation} can be proved via a straightforward but tedious adaptation of the methods of \cite{Fr15, FH18, Ko18}. 
Let $\phi:\Z^K\to G$ be a coordinate map.
    By a natural multiparameter extension of \cite[Theorem 2.1]{Ko18}, it suffices to show that the sequence $\psi\circ \phi:\Z^K\to\C$ satisfies two conditions: 
    \begin{enumerate}
        \item it is \textit{$(s+1)$-weak-anti-uniform} in that for every uniformly bounded $b:\Z^K\to\C$, the average
    \begin{align*}
        \oE_{\um\in\Z^K}\psi\circ \phi(\um)b(\um)
    \end{align*}
    is quantitatively controlled by a degree-$(s+1)$ uniformity seminorm on $\Z^K$ (in the sense of \cite{Ko18});
        \item it is \textit{$(s+1)$-regular} in that the uniform limit
    \begin{align*}
        \E_{\um\in\Z^K}\psi\circ \phi(\um)b(\um)
    \end{align*}
    exists
    for any $s$-step nilsequence $b:\Z^K\to\C$.
    \end{enumerate}

    
The condition of $(s+1)$-weak-anti-uniformity follows by combining a suspension flow trick \cite[Theorem~4.1]{Ko18} (which allows us to remove integer parts from expressions like $T^{\floor{\bm}}$) with a standard van der Corput argument (which can be found e.g. at the end of the proof of \cite[Theorem~2.6]{FH18}). Similarly, $(s+1)$-regularity can be deduced by combining \cite[Theorem~4.1]{Ko18} with \cite[Proposition~2.1]{FH18}. 
For every $\veps>0$, the natural multiparameter extension of \cite[Theorem 2.1]{Ko18} then allows us to decompose $\psi\circ \phi = \psi_1+\psi_2$, where $\psi_1$ is a degree-$s$ nilsequence on $\Z^{K}$ and $\psi_2$ is an $\veps$-nullsequence on $\Z^K$. This gives the claimed result for ${\psi}.$
%
\end{proof}
    

\begin{proof}[Proof of Proposition \ref{P: structure theorem}]
We show  that $\nnorm{f}_{G_1, \ldots, G_s}^+ = 0$ if and only if $f$ is orthogonal to dual functions $\CD_{G_1, \ldots, G_s}((f_\ueps)_{\{0,1\}^s_*}; \psi)$ with the functions $f_\ueps$ and degree-$s$ nilsequence being 1-bounded. The invariance of the measure gives
\begin{multline*}
     \int f \cdot \CD_{G_1, \ldots, G_s}((f_\ueps)_{\{0,1\}^s_*}; \psi)\;d\mu\\
     = \E_{\bm_1, \bm_1'\in G_1}\cdots \E_{\bm_s, \bm_s'\in G_s}\psi(\bm,\bm')\cdot\int \prod_{\ueps\in\{0,1\}^s}T^{\sfloor{\bm_1^{\eps_1}}+\cdots + \sfloor{\bm_s^{\eps_s}}} f_\ueps\, d\mu,
\end{multline*}
where $f_{\underline{0}} = f$.
    Using the Cauchy-Schwarz inequality, we can remove the nilsequence and bound
    \begin{multline*}
        \abs{\int f \cdot \CD_{G_1, \ldots, G_s}((f_\ueps)_{\{0,1\}^s_*}; \psi)\; d\mu}^2\\
        \leq  \E_{\bm_1, \bm_1'\in G_1}\cdots \E_{\bm_s, \bm_s'\in G_s}\abs{\int \prod_{\ueps\in\{0,1\}^s}T^{\sfloor{\bm_1^{\eps_1}}+\cdots + \sfloor{\bm_s^{\eps_s}}} f_\ueps\; d\mu}^2.
    \end{multline*}
    By the Gowers-Cauchy-Schwarz inequality (e.g. \cite[Lemma 3.3]{DKKST24}), the last expression is bounded by $(\nnorm{f}_{G_1, \ldots, G_s}^+)^{1/2^s}$. In particular, if the seminorm is 0, then $f$ is orthogonal to the function $\CD_{G_1, \ldots, G_s}((f_\ueps)_{\{0,1\}^s_*}; \psi)$.

    In the direction covered so far, $\psi(\bm,\bm')$ could be any bounded weight; its nilsequence structure will be used in the converse direction. Expanding the definition of $\nnorm{f}_{G_1, \ldots, G_s}^+$, we get 
    \begin{align*}
        (\nnorm{f}_{G_1, \ldots, G_s}^+)^{2^{s+1}} = \int f\cdot \CD_{G_1, \ldots, G_s}(f; \psi)\; d\mu
    \end{align*}
    for
    \begin{align*}
        \psi(\bm,\bm') = \int \prod_{\ueps\in\{0,1\}^s}\CC^{|\ueps|}T^{\sfloor{\bm_1^{\eps_1}}+\cdots + \sfloor{\bm_s^{\eps_s}}} \overline{f}\, d\mu.
    \end{align*}
    Fix $\veps>0$, and let $\psi_\veps$ be a 1-bounded degree-$s$ nilsequence coming from Proposition~\ref{P: multicorrelation} such that
    \begin{align*}
        \oE_{(\bm,\bm')\in G}\abs{\psi(\bm,\bm')-\psi_\veps(\bm,\bm')} \leq \veps^{2^{s+1}}.
    \end{align*}
    Then (assuming $f$ is 1-bounded, which we may)
    \begin{align*}
        \abs{(\nnorm{f}_{G_1, \ldots, G_s}^+)^{2^{s+1}} - \int f\cdot \CD_{G_1, \ldots, G_s}(f; \psi_\veps)\; d\mu}\leq \oE_{(\bm,\bm')\in G}\abs{\psi(\bm,\bm')-\psi_\veps(\bm,\bm')}\leq \veps^{2^{s+1}}.
    \end{align*}
    The orthogonality of $f$ to $\CD_{G_1, \ldots, G_s}(f; \psi_\veps)$ gives the upper bound $\nnorm{f}_{G_1, \ldots, G_s}^+\leq \veps$, and the result follows by letting $\veps\to 0$.
\end{proof}

As a consequence of relevant properties of seminorms (see Lemma \ref{L: basic properties of seminorms}), we obtain the following properties of factors.
\begin{lemma}[Basic properties of factors]\label{L: basic properties of factors}
   Let $G_1, \ldots, G_s\subseteq\R^\ell$ be finitely generated subgroups. Then: 
     \begin{enumerate}
        \item (Symmetry) For any permutation $\sigma:[s]\to[s]$, we have $$\CZ_{G_{\sigma(1)}, \ldots, G_{\sigma(s)}}^+ = \CZ_{G_1, \ldots, G_s}^+;$$
        \item (Monotonicity) For any finitely generated $G_{s+1}\subseteq\R^\ell$, we have $$\CZ_{G_1, \ldots, G_s}^+\subseteq\CZ_{G_1, \ldots, G_{s+1}}^+;$$
        \item (Subgroup property) If $G'_i\subseteq G_i$ for all $i\in[s]$, then $$\CZ_{G_1, \ldots, G_s}^+\subseteq \CZ_{G_1', \ldots, G_s'}^+,$$ with equality whenever $G_i/G_i'$ has finite index for all $i\in[s]$;
        \item (Rescaling property) If for every $i\in[s]$ we have
        \begin{align*}
            G_i = \langle \bg_{i1}, \ldots, \bg_{iK_i}\rangle\quad \textrm{and}\quad G'_i = \langle \alpha_{i1}\bg_{i1}, \ldots, \alpha_{iK_i} \bg_{iK_i}\rangle 
        \end{align*}
        for some $K_i\in\N$ and $\alpha_{ij}\neq 0$, then
    \begin{align*}
       \CZ^+_{G'_1, \ldots, G'_s}=\CZ^+_{G_1, \ldots, G_s}.
    \end{align*}
    \end{enumerate}
\end{lemma}

\section{Relative concatenation}\label{S: concatenation}

Having defined generalized box factors and established their various properties, we now show that factors along smaller groups can be merged (``concatenated'') into (higher-degree) factors along larger groups. Once again, we fix a system $(X, \CX, \mu, T_1, \ldots, T_\ell)$ throughout this section. The main result of this section is as follows.

\begin{theorem}[Relative concatenation]\label{T: relative concatenation}
Let $s\in\N_0$, $d,d'\in\N$, and $$G_1, \ldots, G_s, H_1, \ldots, H_d, H'_1, \ldots, H'_{d'}\subseteq\R^\ell$$ be finitely generated subgroups. Then the following statements hold:
\begin{enumerate}
    \item We have
\begin{align*}
    \mathcal{Z}^+_{G_{1},\dots,G_s,H_{1},\dots,H_{d}}\cap \mathcal{Z}^+_{G_{1},\dots,G_s,H'_{1},\dots,H'_{d'}}\subseteq \mathcal{Z}^+_{G_{1},\dots,G_{s},(H_{i}+H'_{i'})_{i\in[d], i'\in[d']}}.
\end{align*}
    \item If moreover there exists a subgroup $H\subseteq \R^\ell$ such that $H_i + H'_{i'} = H$ for all $i\in[d]$ and $i'\in[d']$, then
\begin{align*}
    \mathcal{Z}^+_{G_{1},\dots,G_s, H_{1},\dots,H_{d}}\cap \mathcal{Z}^+_{G_{1},\dots,G_s, H'_{1},\dots,H'_{d'}} \subseteq \CZ^+_{G_{1},\dots,G_s, H^{\times (d+d'-1)}},
\end{align*}
		where we recall that $H^{\times d}$ denotes $d$-copies of $H$.
\end{enumerate}
\end{theorem}
When $s=0$ and $H_i, H'_{i'}\subseteq\Z^\ell$, Theorem \ref{T: relative concatenation} has been established by Tao and Ziegler in \cite[Theorems 1.15 and 1.16]{TZ16};\footnote{Tao-Ziegler's results apply more broadly to box factors along subgroups that are sums of finitely generated and profinite groups. For the second statement, their result additionally assumes that $H_1 = \cdots = H_d$ and $H'_1 = \cdots = H_{d'}'$.} it is their influential paper that introduced the idea of concatenating factors and seminorms to ergodic theory, additive combinatorics, and analytic number theory. We extend Tao-Ziegler's result in two ways: by relativizing it (i.e. allowing $s>0$ groups that do not get concatenated), and by extending it to generalized box factors.

We shall use the following consequence of Theorem~\ref{T: relative concatenation}, which generalizes \cite[Corollaries~1.20 and 1.21]{TZ16}, and which can be deduced from Theorem~\ref{T: relative concatenation} in the same way in which \cite[Corollaries~1.20 and 1.21]{TZ16} are derived from \cite[Theorems 1.15 and 1.16]{TZ16}.

\begin{corollary}\label{C: relative concatenation}
    Let $s\in\N_0$, $d,d'\in\N$, and $$G_1, \ldots, G_s, H_1, \ldots, H_d, H'_1, \ldots, H'_{d'}\subseteq\R^\ell$$ be finitely generated subgroups. Let $f\in L^\infty(\mu)$. Then the following statements hold:
    \begin{enumerate}
        \item If $\E(f|\mathcal{Z}^+_{G_{1},\dots,G_{s},(H_{i}+H'_{i'})_{i\in[d], i'\in[d']}})=0$, then we can decompose $f = g_1 + g_2$ for some $g_i\in L^\infty(\mu)$ satisfying
        \begin{align*}
            \E(g_1|\mathcal{Z}^+_{G_{1},\dots,G_s,H_{1},\dots,H_{d}}) = 0\quad\textrm{and}\quad \E(g_2|\mathcal{Z}^+_{G_{1},\dots,G_s,H'_{1},\dots,H'_{d'}}) = 0.
        \end{align*}
        \item If moreover there exists a subgroup $H\subseteq \R^\ell$ such that $H_i + H'_{i'} = H$ for all $i\in[d]$ and $i'\in[d']$, and if $\E(f|\CZ^{+}_{G_{1},\dots,G_s, H^{\times (d+d'-1)}})=0$, then we can decompose $f = g_1 + g_2$ for some $g_i\in L^\infty(\mu)$ satisfying
        \begin{align*}
            \E(g_1|\mathcal{Z}^+_{G_{1},\dots,G_s, H_{1},\dots,H_{d}}) = 0\quad\textrm{and}\quad \E(g_2|\mathcal{Z}^+_{G_{1},\dots,G_s,H'_{1},\dots,H'_{d'}})=0.
        \end{align*}
    \end{enumerate}
\end{corollary}


If we specialize to subgroups of $\Z^\ell$ in Theorem \ref{T: relative concatenation} and Corollary \ref{C: relative concatenation}, the factors $\CZ^+$ can be replaced by the more classical factors $\CZ$, defined e.g. in \cite{HK05a, TZ16} and satisfying the property \eqref{E: seminorm vs factor unplussed} (recall that in this case  $Z_{G_1, \ldots, G_s}(\CY)\subseteq L^\infty(\mu)$ is defined to be the $L^2(\mu)$-closed linear span of functions of the form 
$\CD_{G_1, \ldots, G_s}((f_\ueps)_{\{0,1\}^s_*})$ for $f_\ueps\in L^\infty(\CY, \mu)$). The proof in this case is almost identical; the only difference is that since we only need to consider dual functions with nilsequences $\psi = 1$, we can control expressions like $\int g \cdot \CD_{G_1, \ldots, G_s}f\; d\mu$ by $\nnorm{g}_{G_1, \ldots, G_s}$ rather than $\nnorm{g}_{G_1, \ldots, G_s}^+$.  We leave the details to the interested readers.

For general finitely generated subgroups $G_1, \ldots, G_s\subseteq\R^\ell$, such factors in general do not exist; see the discussions at the beginning of Section \ref{SS: factors} and below Definition \ref{D: factor}.

\subsection{Fundamental concatenation theorem}

Theorem \ref{T: relative concatenation} will be deduced by iterating the following result whose proof is the main subject of this subsection.
\begin{theorem}[Fundamental concatenation theorem]\label{T: fundamental concatenation}
Let $G_1, \ldots, G_s, G_s'\subseteq\R^\ell$ be finitely generated subgroups.
Then $$\mathcal{Z}^+_{G_1,\dots,G_{s-1},G_s}\cap \mathcal{Z}^+_{G_1,\dots,G_{s-1},G'_s}=\mathcal{Z}^+_{G_1,\dots,G_{s-1},G_s+G'_s}.$$ 
\end{theorem}

Before we supply the proof of Theorem \ref{T: fundamental concatenation}, we state several lemmas that will show up in the proof. The first one is a consequence of Proposition \ref{P: structure theorem} and follows from the same proof as \cite[Corollary 2.5]{TZ16}.
\begin{lemma}[Restriction to a factor]\label{L: substituting factor}
    Let $G_1, \ldots, G_s\subseteq\R^\ell$ be finitely generated subgroups. 
    For any factor $\CY\subseteq\CX$, we have
    \begin{align*}
        \CY \cap \CZ^+_{G_1, \ldots, G_s}(\CX) = \CZ^+_{G_1, \ldots, G_s}(\CY).
    \end{align*}
\end{lemma}
    

Before we state the second lemma, let $\CB_{G_1, \ldots, G_s}$ be the closed convex hull of dual functions $\CD_{G_1,\dots,G_{s}}((f_{\ueps})_{\ueps\in\{0,1\}^{s}_{\ast}};\psi)$ for 1-bounded functions $f_\ueps\in L^\infty(\mu)$ and 1-bounded degree-$s$ nilsequences $\psi$ on $G = G_1\times G_1\times \cdots \times G_s\times G_s$.
\begin{lemma}[Approximate multiplicative closedness of $\CB_{G_1, \ldots, G_s}$]\label{L: convex hull}
    Let $G_1, \ldots, G_s\subseteq\R^\ell$ be finitely generated subgroups.
    If $D_1, D_2\in \CB_{G_1, \ldots, G_s}$, then $D_1\cdot D_2\in 3^{\ell s}\cdot \CB_{G_1, \ldots, G_s}$.
\end{lemma}
\begin{proof}
    By convexity, it suffices to prove the result when $D_1, D_2$ are themselves dual functions $\CD_{G_1,\dots,G_{s}}((f_{\ueps})_{\ueps\in\{0,1\}^{s}_{\ast}};\psi)$ for 1-bounded functions $f_\ueps\in L^\infty(\mu)$ and 1-bounded 
    $\psi\in\Nil_s(G)$ for $G = G_1\times G_1\times \cdots \times G_s\times G_s$. By inspecting the proof of Lemma~\ref{L: algebra}, we see that $D_1\cdot D_2$ is a sum of at most $3^{\ell s}$ elements in $\CB_{G_1, \ldots, G_s}$, corresponding to $3^{\ell s}$ possible values of $\br\in\{-1,0,1\}^{\ell s}$.
\end{proof}

\begin{lemma}[Anti-uniformity of $\CB_{G_1, \ldots, G_s}$]\label{L: correlation with dual}
    Let $G_1, \ldots, G_s\subseteq\R^\ell$ be finitely generated subgroups. 
    If $f\in \CB_{G_1, \ldots, G_s}$, $g\in L^\infty(\mu)$, then $\abs{\int f\cdot g\; d\mu}\leq \nnorm{g}_{G_1, \ldots, G_s}^+$.
\end{lemma}
\begin{proof}
    By convexity, it suffices to consider the case when $f$ is a dual function, in which case the result follows from the Gowers-Cauchy-Schwarz inequality much like in the proof of Proposition~\ref{P: structure theorem}. 
\end{proof}
We are now ready to prove Theorem~\ref{T: fundamental concatenation}.
\begin{proof}[Proof of Theorem~\ref{T: fundamental concatenation}]
The converse inclusion is a straightforward consequence of the monotonicity property of factors (Lemma \ref{L: basic properties of factors} (ii)). For the forward inclusion, it suffices to show via Proposition~\ref{P: structure theorem} that any
\begin{align*}
    f\in L^\infty(\mathcal{Z}^+_{G_1,\dots,G_{s-1},G_s}\cap \mathcal{Z}^+_{G_1,\dots,G_{s-1},G'_s})
\end{align*}
is orthogonal to $g\in L^{\infty}(\mu)$ with $\nnorm{g}^+_{G_1,\dots, G_{s-1}, G_s+G_s'}=0$. Indeed, if the latter statement holds, then 
\begin{align*}
    0 &= \int f\cdot (\bar{f}-\mathbb{E}(\bar{f}\vert \CZ^+_{G_1,\dots, G_{s-1}, G_s+G_s'}))\ d\mu\\
    &= \norm{f}_{L^2(\mu)}^2 - \norm{\mathbb{E}(f\vert \CZ^+_{G_1,\dots, G_{s-1}, G_s+G_s'})}_{L^2(\mu)}^2,
\end{align*}
and so $f = \mathbb{E}(f\vert \CZ^+_{G_1,\dots, G_{s-1}, G_s+G_s'})\in Z^+_{G_1,\dots, G_{s-1}, G_s+G_s'}$. 

By an $L^2(\mu)$ approximation argument, it suffices to consider $$f = \CD_{G_1,\dots,G_{s-1},G_{s}}((f_{\ueps})_{\ueps\in\{0,1\}^{s}_{\ast}};\psi)$$ for 1-bounded functions $f_\ueps\in L^\infty(\mu)$ and 
$\psi\in\Nil_s(G)$ for $G = G_1\times G_1\times \cdots \times G_s\times G_s$. Since the nilsequence $\psi$ can be approximated in  $$\norm{v}_{\ell^2(G)}:=\Bigbrac{\oE_{\bm\in G}|v(\bm)|^2}^{1/2}$$ by a linear combination of nilsequences
\begin{align*}
  \psi_s(\bm_s,\bm_s') \psi_{s-1}(\bm_1, \bm_1', \ldots, \bm_{s-1},\bm_{s-1}')  
\end{align*}
(where $\bm_i,\bm_i'\in G_i$),
we can assume that $\psi$ takes such form.

Applying Lemma \ref{L: substituting factor} with $\CY =  \mathcal{Z}^+_{G_1,\dots,G_{s-1},G'_s}$,
we can take every $f_\veps$ to be $\mathcal{Z}^+_{G_1,\dots,G_{s-1},G'_s}$-measurable. 
A similar approximation argument as before allows us to take each $f_\ueps$ to equal
$f_{\ueps}=\mathcal{D}_{G_1,\dots,G_{s-1},G_{s}'}((f_{\ueps,\ueps'})_{\ueps'\in\{0,1\}^{s}_{\ast}}; \psi_{\ueps})$ for some $\mathcal{Z}^+_{G_1,\dots,G_{s-1},G'_s}$-measurable functions and nilsequences $\psi_\ueps\in\Nil_s(G)$, all of which can be assumed to be 1-bounded. 

Inspecting the definition of dual functions, we observe that $f$ can be recast  as
\begin{multline*}
    f = \E_{\bm_s,\bm_s'\in G_s} \psi_s(\bm_s,\bm_s')\cdot T^{\sfloor{\bm_s'}-\sfloor{\bm_s}}f_{\underline{0}1} \\ \CD_{G_1,\dots,G_{s-1}}((f_{\ueps 0}\cdot T^{\sfloor{\bm_s'}-\sfloor{\bm_s}} f_{\ueps 1})_{\ueps\in\{0,1\}^{s-1}_{\ast}}; \psi_{s-1}).
\end{multline*}
Arguing similarly, we can recast $f_{\underline{0}1}$ as 
\begin{multline*}
    f_{\underline{0}1} = \E_{\bn_s,\bn_s'\in G_s'} \psi_{\underline{0}1, s}(\bn_s,\bn_s')\cdot T^{\sfloor{\bn_s'}-\sfloor{\bn_s}}f_{\underline{0}1, \underline{0}1} \\ \CD_{G_1,\dots,G_{s-1}}((f_{\underline{0}1, \ueps 0}\cdot T^{\sfloor{\bn_s'}-\sfloor{\bn_s}} f_{\underline{0}1, \ueps 1})_{\ueps\in\{0,1\}^{s-1}_{\ast}};\psi_{\underline{0}1,s-1}).
\end{multline*}
Plugging in the formula for $f_{\underline{0}1}$ into the formula for $f$ and setting
\begin{multline*}
    D_{\bm_s,\bm_s', \bn_s,\bn_s'} = \psi_s(\bm_s,\bm_s')\cdot\CD_{G_1,\dots,G_{s-1}}((f_{\ueps 0}\cdot T^{\sfloor{\bm_s'}-\sfloor{\bm_s}} f_{\ueps 1})_{\ueps\in\{0,1\}^{s-1}_{\ast}}; \psi_{s-1})\\
    \psi_{\underline{0}1, s}(\bn_s,\bn_s')\cdot T^{\sfloor{\bm_s'}-\sfloor{\bm_s}}\CD_{G_1,\dots,G_{s-1}}((f_{\underline{0}1, \ueps 0}\cdot T^{\sfloor{\bn_s'}-\sfloor{\bn_s}} f_{\underline{0}1, \ueps 1})_{\ueps\in\{0,1\}^{s-1}_{\ast}};\psi_{\underline{0}1,s-1}),
\end{multline*}
we deduce that
\begin{align*}
    f = \E_{\bm_s,\bm_s'\in G_s} \E_{\bn_s,\bn_s'\in G_s'} T^{\sfloor{\bm_s'} +\sfloor{\bn_s'}- \sfloor{\bm_s}-\sfloor{\bn_s}}f_{\underline{0}1, \underline{0}1} \cdot D_{\bm_s,\bm_s', \bn_s,\bn_s'}.
\end{align*}
By incorporating emerging error terms into $D_{\bm_s,\bm_s', \bn_s,\bn_s'}$, we can reduce to the case where $f$ equals
\begin{align*}
    f = \E_{\bm_s,\bm_s'\in G_s} \E_{\bn_s,\bn_s'\in G_s'} 
    T^{\sfloor{\bm_s'+\bn_s'}- \sfloor{\bm_s+\bn_s}}f_{\underline{0}1, \underline{0}1} \cdot D_{\bm_s,\bm_s', \bn_s,\bn_s'}.   
\end{align*}

The point of all these maneuvers is that we have by now expressed $f$ as (essentially) an ergodic average along $G_s + G_s'$ twisted by 
1-bounded $\CZ^+_{G_1, \ldots, G_{s-1}}$-measurable functions $D_{\bm_s,\bm_s', \bn_s,\bn_s'}$. By removing the dual function in a more or less standard manner, we will be able to control $\int f\cdot g\; d\mu$ by $\nnorm{g}^+_{G_1,\dots, G_{s-1}, G_s+G_s'}$, and this will conclude the proof of the forward direction.

Using the invariance of measure and the Cauchy-Schwarz inequality, we can bound
\begin{align*}
    \abs{\int f\cdot g\; d\mu} &= \abs{\int f_{\underline{0}1, \underline{0}1} \cdot \E_{\bm_s,\bm_s'\in G_s} \E_{\bn_s,\bn_s'\in G_s'} 
    T^{\sfloor{\bm_s+\bn_s}-\sfloor{\bm_s'+\bn_s'}}(g\cdot D_{\bm_s,\bm_s', \bn_s,\bn_s'})\; d\mu}\\
    &\leq \norm{\E_{\bm_s,\bm_s'\in G_s} \E_{\bn_s,\bn_s'\in G_s'} 
    T^{\sfloor{\bm_s+\bn_s}-\sfloor{\bm_s'+\bn_s'}}(g\cdot D_{\bm_s,\bm_s', \bn_s,\bn_s'})}_{L^2(\mu)}.
\end{align*}
Applying the van der Corput lemma (e.g. \cite[Lemma 2.1]{DKS22}), composing the resulting integral with $T^{\sfloor{\bm_s'+\bn_s'}- \sfloor{\bm_s+\bn_s}}$, handling error terms, and then pigeonholing in $\bm_s,\bm_s',\bn_s,\bn_s'$, we deduce that the left-hand side vanishes whenever
\begin{align}\label{E: after vdC}
    \E_{\bh_s,\bh_s'\in G_s} \E_{\bk_s,\bk_s'\in G_s'}\abs{\int D_{\substack{\bh_s,\bh_s', \bk_s,\bk_s'}} \cdot g\cdot T^{\sfloor{\bh_s+\bk_s}-\sfloor{\bh_s'+\bk_s'}}\overline{g}\; d\mu}^2
\end{align}
vanishes for some
1-bounded $\CZ^+_{G_1, \ldots, G_{s-1}}$-measurable $D_{\substack{\bh_s,\bh_s', \bk_s,\bk_s'}}$.

The last step is to apply Lemma \ref{L: correlation with dual} and \cite[Lemma 4.1]{DKKST24} in order to bound the expression above in terms of 
\begin{align*}
    \E_{\bh_s,\bh_s'\in G_s} \E_{\bk_s,\bk_s'\in G_s'}\brac{\nnorm{g\cdot T^{\sfloor{\bh_s+\bk_s}-\sfloor{\bh_s'+\bk_s'}}\overline{g}}_{G_1, \ldots, G_{s-1}}^+}^{2^s} = \brac{\nnorm{g}_{G_1, \ldots, G_{s-1}, G_s+G_s'}^+}^{2^{s+1}}.
\end{align*}
Analyzing how we arrived at \eqref{E: after vdC}, we see that each $D_{\substack{\bh_s,\bh_s', \bk_s,\bk_s'}}$ is a linear combination of $O_{\ell,s}(1)$ 1-bounded degree-$s$ nilsequences times 4 dual functions from $\CB_{G_1, \ldots, G_{s-1}}$. Using Lemma~\ref{L: convex hull}, we can therefore find absolute $C=C(\ell,s)>0$ such that $$D_{\substack{\bh_s,\bh_s', \bk_s,\bk_s'}}\in C\cdot \CB_{G_1, \ldots, G_{s-1}},$$ and so by Lemma~\ref{L: correlation with dual}, the average \eqref{E: after vdC} vanishes whenever $\nnorm{g}_{G_1, \ldots, G_{s-1}, G_s+G_s'}^+=0.$ This concludes the proof.
\end{proof}

\subsection{Induction scheme}

The next goal is to present an induction scheme that allows us to upgrade Theorem \ref{T: fundamental concatenation} to Theorem \ref{T: relative concatenation}. Our induction is based on the following formalism.
\begin{definition}\label{D: CIS}
   Let $G$ be an abelian group, $\text{sub}(G)$ be the set of all subgroups of $G$, and $\text{sub}(G)^{\omega}:=\sqcup_{n=1}^{\infty}\text{sub}(G)^{n}$. We say that a subset $\mathcal{C}$ of $\text{sub}(G)^{\omega}$ is a \emph{CMS-family} if the followings properties hold:
   \begin{enumerate}
 		\item (Concatenation) For all $d\in\mathbb{N}$ and subgroups $H_{1},\dots,H_{d},H'_{d}\subseteq G$,
 		$$(H_{1},\dots,H_{d-1},H_{d}), (H_{1},\dots,H_{d-1},H'_{d})\in \mathcal{C}\Rightarrow (H_{1},\dots,H_{d-1},H_{d}+H'_{d})\in \mathcal{C}.$$
 		\item (Monotonicity) For all $d\in\mathbb{N}$ and subgroups $H_{1},\dots,H_{d}\subseteq G$,  
 		$$(H_{1},\dots,H_{d-1})\in \mathcal{C}\Rightarrow (H_{1},\dots,H_{d-1},H_{d})\in \mathcal{C}.$$
 		\item (Symmetry) For all $d\in\mathbb{N}$, subgroups $H_{1},\dots,H_{d}\subseteq G$ and permutations $\tau:[d]\to[d]$, we have
 		$$(H_{1},\dots,H_{d})\in \mathcal{C}\Rightarrow (H_{\tau(1)},\dots,H_{\tau(d)})\in \mathcal{C}.$$
 	\end{enumerate}
\end{definition}

For $f\in L^\infty(\mu)$, consider the following family of subgroups:
\begin{align*}
    \CC_f \colon= \{(G_1, \ldots, G_s)\in \text{sub}(\R^\ell)^\omega:\; f\in Z^+_{G_1,\ldots, G_s}\}.
\end{align*}
As it turns out, $\CC_f$ satisfies the conditions of Definition \ref{D: CIS}.
\begin{proposition}\label{P: C_f is CIS}
    For every $f\in L^\infty(\mu)$, the collection $\mathcal{C}_{f}$ is a CMS-family. 
\end{proposition}
\begin{proof}
    The Monotonicity and Symmetry properties follow from Lemma \ref{L: basic properties of factors} while the Concatenation property follows from Theorem \ref{T: fundamental concatenation}.
\end{proof}

We will derive Theorem \ref{T: relative concatenation} by combining Proposition \ref{P: C_f is CIS} (which one could see as the ``base case'' of induction) with Proposition \ref{P: CIS induction} below (which provides the inductive scheme).
\begin{proposition}\label{P: CIS induction}
	Let $s\in\N_0$, $d,d'\in\N$, and $G_1, \ldots, G_s, H_1, \ldots, H_d, H_1', \ldots, H_{d'}'$ be subgroups of an abelian group $G$.
    Suppose that $\mathcal{C}$ is a CMS-family. Then the following statements hold:
\begin{enumerate}
    \item We have 
    \begin{gather*}
        (G_{1},\dots,G_{s},H_{1},\dots,H_{d}),(G_{1},\dots,G_{s},H'_{1},\dots,H'_{d'})\in \mathcal{C}\\
        \Rightarrow(G_{1},\dots,G_{s},(H_{i}+H'_{i'})_{i\in[d], i'\in[d']})\in \mathcal{C}.
    \end{gather*}
    \item If $H_{i}+H'_{i'}=H$ for all $i\in[d], i'\in[d']$ for some subgroup $H\subseteq G$, then 
    \begin{gather*}
        (G_1, \ldots, G_s, H_{1},\dots,H_{d}), (G_1, \ldots, G_s, H'_{1},\dots,H'_{d'})\in \mathcal{C}\\
        \Rightarrow (G_1, \ldots, G_s, H^{\times (d+d'-1)})\in \mathcal{C}.
    \end{gather*}
\end{enumerate}
		
%
\end{proposition}

\begin{proof}[Proof of Theorem \ref{T: relative concatenation} assuming Proposition \ref{P: CIS induction}]
    Let $f\in L^\infty(\mu)$. To prove the first claim in the language of CMS-families, it suffices to show that if $\CC_f$ contains 
    \begin{align*}
        (G_{1},\dots,G_s,H_{1},\dots,H_{d})\quad \textrm{and}\quad (G_{1},\dots,G_s,H'_{1},\dots,H_{d}'),
    \end{align*}
    then it also contains 
    \begin{align*}
        (G_{1},\dots,G_{s},(H_{i}+H'_{i'})_{i\in[d], i'\in[d']}).
    \end{align*}
    This follows from the fact that $\CC_f$ is a CMS-family (Proposition \ref{P: C_f is CIS}) and the induction scheme given by the first part of Proposition \ref{P: CIS induction}. Similarly, the second claim follows from Proposition \ref{P: C_f is CIS} and the second part of Proposition \ref{P: CIS induction}.
\end{proof}

We now prove Proposition \ref{P: CIS induction}.

\begin{proof}[Proof of Proposition \ref{P: CIS induction}]

To simplify the notation, we write $[H_{1},\dots,H_{d}]$ if $(H_{1},\dots,H_{d})$ belongs to $\mathcal{C}$, and we write $A+B$ to indicate that claims $A$ and $B$ both hold.

We start with the first part of the proposition.  Our claim can then be rephrased in this language as
\begin{align*}
    [G_1, \ldots, G_s, H_1, \ldots, H_d] + [G_1, \ldots, G_s, H'_1, \ldots, H_d'] \Rightarrow   [G_1, \ldots, G_s, (H_{i}+H'_{i'})_{i\in[d], i'\in[d']}].
\end{align*}

For any $(d,d')\in\N^2$, we say that $(d,d')$ is \textit{good} if the implication
	\begin{equation}\label{E: (d,d')}
	\begin{split}
	[\bold{J},H_{1},\dots,H_d]+[\bold{J},H'_{1},\dots,H'_{d'}]
	\Rightarrow[\bold{J},(H_{i}+H'_{i'})_{i\in[d], i'\in[d']}]
	\end{split}
	\end{equation} 
	holds for any finite multifamily of subgroups $\bold{J}=(K_{j})_{j\in J}$ of $G$ (that is, $\vert J\vert<\infty$, elements are counted with multiplicity, and the family may be empty).
	Then our goal is to show that every $(d,d')\in\N^2$ is good. We will prove this by double induction, first showing that all pairs $(d,1)$ and $(1,d')$ are good and then establishing the statement for arbitrary $(d,d')\in\N^2$.
    
    That $(1,1)$ is good is equivalent to the Concatenation property. Suppose that $(d,1)$ is good for some $d\in\N$, i.e. 
    \begin{align}\label{E: (d,1)}
        [\bJ, H_1, \ldots, H_d] + [\bJ, H_1'] \Rightarrow [\bJ, (H_i+H_1')_{i\in[d]}]
    \end{align}
    holds for all $\bJ$. We want to show that $(d+1,1)$ is good as well. Applying Monotonicity, and then Symmetry as well as the implication \eqref{E: (d,1)} for the family $\bJ\cup\{H_{d+1}\}$, we obtain the implication
    \begin{align*}
        [\bJ, H_1, \ldots, H_{d+1}] + [\bJ, H_1'] &\Rightarrow [\bJ, H_1, \ldots, H_{d+1}] + [\bJ, H_1', H_{d+1}]\\
        &\Rightarrow [\bJ, H_{d+1}, (H_i+H_1')_{i\in[d]}].
    \end{align*}
    It remains to concatenate $H_1'$ with $H_{d+1}$. By Monotonicity, followed by Symmetry and Concatenation,
    we deduce that
    \begin{align*}
        [\bJ, H_{d+1}, (H_i+H_1')_{i\in[d]}] + [\bJ, H_1'] &\Rightarrow [\bJ, H_{d+1}, (H_i+H_1')_{i\in[d]}] + [\bJ, H_1', (H_i+H_1')_{i\in[d]}]\\
        &\Rightarrow [\bJ, (H_i+H_1')_{i\in[d+1]}].
    \end{align*}
    Hence $(d+1,1)$ is good. By induction and symmetry, $(d,1)$ and $(1,d')$ are good for any $d,d'\in\N$.

       Now, suppose that $(d,d')\in\N^2$ is good; we want to show that $(d+1,d')$ is good. In doing so, we will inductively apply the goodness of $(d,d')$ and $(1,d')$. Applying Monotonicity first and then Symmetry and \eqref{E: (d,d')} for the family $\bJ\cup\{H_{d+1}\}$ (this is where we apply the case $(d,d')$), we get
       \begin{align*}
           [\bold{J},H_{1},\dots,H_{d+1}]+[\bold{J},H'_{1},\dots,H'_{d'}] &\Rightarrow [\bold{J},H_{1},\dots,H_{d+1}]+[\bold{J},H'_{1},\dots,H'_{d'}, H_{d+1}]\\
           &\Rightarrow [\bold{J}, H_{d+1}, (H_{i}+H'_{i'})_{i\in[d], i'\in[d']}].
       \end{align*}
        It remains to concatenate $H_{d+1}$ with the subgroups $H'_{i'}$. By Monotonicity, followed by Symmetry and  \eqref{E: (d,d')} for the family $\bJ\cup\{H_{i}+H'_{i'}:\; i\in[d], i'\in[d']\}$ (this is where we apply the case $(1,d')$), we obtain
        \begin{multline*}
            [\bold{J}, H_{d+1}, (H_{i}+H'_{i'})_{i\in[d], i'\in[d']}] + [\bold{J},H'_{1},\dots,H'_{d'}]\\ \Rightarrow [\bold{J}, H_{d+1}, (H_{i}+H'_{i'})_{i\in[d], i'\in[d']}] + [\bold{J},H'_{1},\dots,H'_{d'}, (H_{i}+H'_{i'})_{i\in[d], i'\in[d']}]\\
            \Rightarrow [\bold{J}, (H_{i}+H'_{i'})_{i\in[d+1], i'\in[d']}].
        \end{multline*}
        This establishes the case $(d+1,d')$. The case $(d,d'+1)$ follows by an identical argument, and hence the claim holds by induction.

        \

	We move on to the second part of the proposition. 
	Our starting assumptions can then be rephrased 
    as $[\bJ, H_{1},\dots,H_{d}]$ and $[\bJ, H'_{1},\dots,H'_{d'}]$, where this time we fix $\bJ = \{G_1, \ldots, G_s\}$.
	For $(i,i')\in([0,d]\times[0,d'])\backslash\{(0,0)\}$, we say that $(i,i')$ is \emph{good} if 
    $$[\bJ, H_{1},\dots,H_{d-i},H'_{1},\dots,H'_{d'-i'},H^{\times (i+i'-1)}].$$ 
    Then our goal is to show that $(d,d')$  is good, which will follow from the following subclaims:
    \begin{enumerate}
        \item if $(i+1,i')$ and $(i,i'+1)$ are good for some  $i\in[0,d-1], i'\in[0,d'-1]$, then so is $(i+1,i'+1)$;
        \item $(i,0)$ and $(0,i')$ are good for all $i\in[d], i'\in[d']$.
    \end{enumerate}
      It is not hard to see that both claims imply together that $(d,d')$ is good; to see this, visualize $(i,i')$ as lattice points on the rectangle $([0,d]\times[0,d'])\backslash\{(0,0)\}$.
	
     For the first subclaim, we infer from the Symmetry and Concatenation properties that
	\begin{equation}\nonumber
	\begin{split}
	&\quad [\bJ, H_{1},\dots,H_{d-(i+1)},H'_{1},\dots,H'_{d'-(i'+1)},H'_{d'-i'},H^{\times (i+i')}]
	\\&+[\bJ, H_{1},\dots,H_{d-(i+1)},H'_{1},\dots,H'_{d'-(i'+1)},H_{d-i},H^{\times (i+i')}]
	\\& \Rightarrow[\bJ, H_{1},\dots,H_{d-(i+1)},H'_{1},\dots,H'_{d'-(i'+1)},H^{\times(i+i'+1)}],
	\end{split}
	\end{equation} 
	which implies that $(i+1,i'+1)$ is good whenever $(i+1,i')$ and $(i,i'+1)$ are. 
	
	To show that $(0,i')$ is good, we use the Monotonicity property to deduce that
	$$[\bJ, H_{1},\dots,H_{d}]\Rightarrow [\bJ, H_{1},\dots,H_{d},H'_{1},\dots,H'_{d'-i'},H^{\times(i'-1)}]$$
	for all $i'\in[d']$. Since $[\bJ, H_{1},\dots,H_{d}]$ by assumption, we deduce that $(0,i')$ is good for all $i'\in[d']$. By an analogous argument, $(i,0)$ is good for all $i\in[d]$. This proves the second subclaim and thus finishes the proof.
\end{proof}

\section{Seminorm smoothing}\label{S: smoothing}
Our next goal is to derive improved seminorm estimates on arbitrary multiple ergodic averages of commuting transformations along Hardy sequences of polynomial growth. In \cite[Theorem 10.1]{DKKST24}, we obtained very general estimates on such averages; in this section, we upgrade them to ones that are more useful for our applications via a new variant of the seminorm smoothing method from \cite{DKKST24, FrKu22b, FrKu22a}. This is the first step in proving that the difference and product conditions in Definition \ref{D: good for joint ergodicity} are \textit{sufficient} for joint ergodicity.

To state our main results, we need the following definitions. 

\begin{definition}[Ordered Hardy family]\label{D: ordered Hardy family}
    Let $\CQ = (q_1, \ldots, q_m)\subseteq\CH$ be a family of Hardy sequences such that $q_i\succ 1$ for all $i\in[m]$. We say that $\CQ$ is \textit{an ordered Hardy family} (or simply \textit{ordered}) if there exist integers $1\leq m_1\leq m_2\leq m_3\leq m$ such that the following holds:
        \begin{enumerate}
            \item $1\prec q_1(x) \prec \cdots \prec q_{m_1}(x)\ll\log x$;
            \item $\log x\prec q_{m_1+1}(x) \prec \cdots \prec q_{m_2}(x) \prec x^\veps$ for every $\veps>0$;
            \item for $m_2 < i \leq m_3$, we have $q_i(x) = x^{l_i}$ for positive integers $l_{m_2+1}< \cdots~<~l_{m_3}$;
            \item $x^\veps \prec q_{m_3+1}(x) \prec \cdots \prec q_m(x)$ for some $\veps>0$, and for each $m_3 < i\leq m$, the function $q_i$ is strongly nonpolynomial.
        \end{enumerate}    
\end{definition}
For instance, the family $(\log x, (\log x)^2, x, x^2, x^{2/3}, x\log x, x^{3/2})$ is ordered. It is a priori strange to put monomials $x, x^2$ before slower growing strongly nonpolynomial functions like $x^{2/3}$ of positive fractional degree, but this ordering turns out to come out naturally in the seminorm bounds from \cite{DKKST24} stated below. More specifically, strongly nonpolynomial functions of positive fractional degree like $x^{2/3}$ can be Taylor-expanded as polynomials of arbitrarily high degree on short intervals, and in this sense they attain higher degrees than polynomials $x, x^2, \ldots$ whose Taylor expansions are polynomials of the same degree.
\begin{definition}[Extended span]
    Let $\CQ = (q_1, \ldots, q_m)\subseteq\CH$ be an ordered Hardy family. Then the \textit{extended span} of $\CQ$, denoted by $\extspan_\R\CQ$, is the collection of sequences $a\in\CH$ satisfying
    $$\lim_{x\to\infty}\abs{a(x) - \sum\limits_{i=1}^m \beta_i q_i(x) - \beta_0} = {0}$$ for some $\beta_0, \ldots,\beta_m\in\R$.
\end{definition}

 
\begin{definition}[Leading coefficients]\label{d43}
    Let $\CQ = (q_1, \ldots, q_m)\subseteq\CH$ be an ordered Hardy family, and suppose that $\ba\in(\extspan_\R\CQ)^\ell$ satisfies $$\lim_{x\to\infty}\abs{\ba(x) - \sum\limits_{i=1}^m \bbeta_i q_i(x) - \bbeta_0} = {0}$$ for some $\bbeta_0, \ldots,\bbeta_m\in\R^\ell$. Then we call $\bbeta_i$ the \textit{leading coefficient} of $\ba$ if $\bbeta_i\neq \mathbf{0}$ and $\bbeta_{i'} = \mathbf{0}$ for all $i<i'\leq m$ {(note that the choices of $\bbeta_{i}$ are unique)}. If $\bbeta_0 = \cdots = \bbeta_m = \bf{0}$, then we say that $\ba$ has leading coefficient $\bf{0}.$
\end{definition}

All the concepts in this section involving degrees, leading coefficients and independence are defined with respect to a fixed ordered family. We will not specify the family as it is always clear from the context.

Then, \cite[Theorem 10.1]{DKKST24} gives the following seminorm bounds. 
     \begin{theorem}\label{T: old estimates}

        Let $\ell\in\N$. Let $\CQ\subseteq\CH$ be an ordered Hardy family, and let $a_1, \ldots, a_\ell\in\extspan_\R\CQ$ with $a_1\succ \log$. Then there exists a positive integer $s=O_{\ell, \CQ}(1)$ and vectors $\bv_{1},\dots,\bv_{s}$ coming from the set of the leading coefficients of $$\{a_1 \be_1,\; a_1 \be_1 - a_2 \be_2,\; \ldots,\; a_1 \be_1 - a_\ell \be_\ell\},$$
         such that for all systems $(X, \CX, \mu, T_1, \ldots, T_\ell)$ and functions $f_1, \ldots, f_\ell\in L^\infty(\mu)$, we have 
\begin{align*}
         \lim_{N\to\infty} 
         \norm{\E_{n\in [N]} 
         \prod_{j\in[\ell]}  T_{j}^{\sfloor{a_{j}(n)}} f_{j}}_{L^2(\mu)}=0
    \end{align*}
    whenever $\nnorm{f_1}_{\bv_{1},\dots,\bv_{s}}=0$.
    \end{theorem}

The upgraded seminorm estimates that we prove in this paper can then be restated as follows. We recall that various notions of independence of Hardy sequences are defined in Definition \ref{D: independence}. 
     \begin{theorem}\label{T: new estimates}

        Let $\ell\in\N$. Let $\CQ\subseteq\CH$ be an ordered Hardy family, and let $a_1, \ldots, a_\ell\in\extspan_\R\CQ$ with $a_1\succ \log$.  Then there exists a positive integer $s=O_{\ell, \CQ}(1)$ and vectors $\bv_{1},\dots,\bv_{s}$ coming from the set of the leading coefficients of 
        \begin{align*}
            \{a_1 \be_1\}\cup\{a_i \be_i - a_j \be_j:\; i\neq j,\; a_i, a_j\; \textrm{are\; dependent},\; a_i, a_j\succ \log\},
        \end{align*}
         such that for all systems $(X, \CX, \mu, T_1, \ldots, T_\ell)$ and functions $f_1, \ldots, f_\ell\in L^\infty(\mu)$, we have 
\begin{align}\label{goal00}
        \lim_{N\to\infty} 
         \norm{\E_{n\in [N]} 
         \prod_{j\in[\ell]}  T_{j}^{\sfloor{a_{j}(n)}} f_{j}}_{L^2(\mu)}=0
    \end{align}
    whenever $\nnorm{f_1}_{\bv_{1},\dots,\bv_{s}}=0$. 
    \end{theorem}
    In fact, a close inspection of Theorem~\ref{T: new estimates general}, from which this result follows, shows that we only need to consider the leading coefficient of $a_i \be_i - a_j \be_j$ for those dependent $a_i, a_j$ of the same or higher degree than $a_1$ in the sense of Definition~\ref{D: degree and independence}.

A priori, there seems to be no good reason why the vectors featuring in Theorem~\ref{T: new estimates} should be any better than those from Theorem~\ref{T: old estimates}. However, using the difference ergodicity condition, we will later be able to further replace all $\bv_i$'s in Theorem~\ref{T: new estimates} by copies of $\be_1$. This will give us Host-Kra seminorm control over the average in Theorem~\ref{T: new estimates} assuming the difference ergodicity condition.

\begin{example}
    For comparison, consider the tuple $$(T_1^{\sfloor{n^{3/2}}}, T_2^{\sfloor{n^{3/2}+\log n}}, T_3^{\sfloor{n^{3/2}+n\log n}}, T_4^{\sfloor{n^{3/2}+n\log n}})_n.$$ Then Theorem \ref{T: old estimates} gives that the corresponding multiple ergodic average would be controlled by a seminorm of $f_1$ involving many copies of $\be_1,\; \be_1 - \be_2,\; \be_1-\be_3,\; \be_1 - \be_4$. The iterates can be divided into two families of pairwise dependent functions, $\{n^{3/2},\; n^{3/2} + \log n\}$ and $\{n^{3/2}+ n\log n,\; n^{3/2} + n\log n\}$,  hence Theorem~\ref{T: new estimates} gives us control in terms of a seminorm of $f_1$ involving many copies of $\be_1,\; \be_1 - \be_2,\; \be_3-\be_4$. If the difference ergodicity condition holds, i.e. the sequences
    \begin{equation}\label{E: ergodic actions}
        (T_1^{\sfloor{n^{3/2}}} T_2^{-\sfloor{n^{3/2}+\log n}})_n\quad \textrm{and} \quad (T_3^{\sfloor{n^{3/2} + n\log n}}T_4^{-\sfloor{n^{3/2} + n\log n}})_n
    \end{equation}
    are ergodic, then we can replace each of the vectors $\be_1 - \be_2,\; \be_3-\be_4$ in the seminorm by many copies of $\be_1$. For $\be_3-\be_4$, this is immediate: the ergodicity of the second action in \eqref{E: ergodic actions}
    implies that $T_3T_4\inv$ is ergodic, hence $\nnorm{f}_{\be_3-\be_4}\leq \nnorm{f}_{\be_1}$ by e.g. \cite[Lemma 2.2]{FrKu22a}. For $\be_1-\be_2$, this requires extra work, and the general method for unpacking the ergodicity of sequences like $(T_1^{\sfloor{n^{3/2}}} T_2^{-\sfloor{n^{3/2}+\log n}})_n$ will be presented in the next sections.
\end{example}

It is worth noting that one can use the concatenation theorem to merge Theorems~\ref{T: old estimates} and ~\ref{T: new estimates} into a stronger result. For instance, the previous example gives one estimate involving the vectors $\be_{1},\be_{1}-\be_{2},\be_{1}-\be_{3},\be_{1}-\be_{4}$ and another one involving $\be_{1},\be_{1}-\be_{2},\be_{3}-\be_{4}$. Concatenating them, we get an estimate involving $\be_{1},\be_{1}-\be_{2},\langle\be_{1}-\be_{3},\be_{1}-\be_{4}\rangle$. However, we do not require strengthenings like this in the current paper.

\subsection{General formalism}

The proof of Theorem \ref{T: new estimates} uses no properties of Hardy functions except the initial seminorm estimates provided by Theorem \ref{T: old estimates}. We therefore phrase its proof in a more general language that makes no reference to Hardy functions.

Our basic setup is captured by the following definition.
\begin{definition}\label{d47}[Ordered family, degree, and independence]\label{D: degree and independence}
   An \textit{ordered family of sequences} is an ordered collection $\CQ =(q_1, \ldots, q_m)$ of sequences $q_i: \N\to\R$ whose germs are affinely independent.\footnote{That means that for every  $N\in\N$, the restrictions of $1, q_1, \ldots, q_m$ to $[N, \infty)$ are $\R$-independent.} Additionally, we assume that each ordered family of sequences comes with a fixed integer $0\leq m_1<m$.

    Suppose that $a_1, \ldots, a_\ell:\N\to\R$ satisfy
        $$\lim_{x\to\infty}\abs{a_j(x) - \sum\limits_{i=1}^m \beta_{j,i} q_i(x) - \beta_{j,0}} = {0}$$ for some $\beta_{ji}\in\R$; like before, we call the collection of such sequences the \textit{extended span} of $\CQ$ and denote it via $\extspan_\R\CQ$.
    We define the \textit{degree} of $a_j$ via $$\deg a_j := \max\{i\in[0, m]:\; \beta_{j,i}\neq 0\}$$
    (the notion is well-defined since the choices of $\beta_{j,i}$ are unique),
    letting it be 0 if $\beta_{j,m} = \cdots = \beta_{j,0} = 0$. Lastly, we call $\beta_{j, \deg a_j}$ the \textit{leading coefficient} of $a_j$.
    We call the sequences:
    \begin{itemize}
        \item $a_j, a_j'$  \textit{independent}  if $(\beta_{j, m_1+1}, \ldots, \beta_{j, m})$ and $(\beta_{j', m_1+1}, \ldots, \beta_{j',m})$ are $\R$-linearly independent, otherwise we call them \textit{dependent};
        \item $a_1, \ldots, a_m$ \textit{pairwise independent} if for every $1\leq j < j' \leq \ell$, the sequences $a_j, a_{j'}$ are independent. 
    \end{itemize}
\end{definition}
In other words, pairwise independence means that no nontrivial linear combination $c_j a_j + c_{j'}a_{j'}$ for $1\leq j < j' \leq \ell$ lies in $\extspan_\R\{a_1, \ldots, a_{m_1}\}$. For Hardy sequences defined with respect to an ordered Hardy basis, the notion of pairwise independence above coincides with the analogous notion from Definition \ref{D: independence}.

The definitions of degree, leading coefficients and pairwise (in)dependence naturally extend to sequences $\N^\ell\to\R$ defined as $\R^\ell$-linear combinations of elements of the ordered basis $\CQ$.

The averages studied in this section will typically be twisted by dual sequences defined below. For the definition of dual functions $\CD_{s, T_j}(f)$, see Definition \ref{D: dual functions}. 
\begin{definition}[Dual sequences]\label{D: dual sequences}
    Let $d,\ell\in\N$ and $(X, \CX, \mu, T_1, \ldots, T_\ell)$ be a system. Then $\FD_d$ denotes the class of sequences $\CD:\Z\to L^\infty(\mu)$ given by $\CD(n) = T_j^n\CD_{s, T_j}(f)$ for some $s\in[d]$, $j\in[\ell]$, and $f\in L^\infty(\mu)$. We refer to elements of $\FD_d$ as \textit{level-$d$ dual sequences}.
\end{definition}

The initial seminorm estimates instantiated by Theorem \ref{T: old estimates} are a special case of the following definition. 
\begin{definition}[Families good for smoothing]
    We call an ordered family of sequences $\CQ$ \textit{good for smoothing} if for all $d,\ell\in\N$, $J\in\N_0$, sequences $a_1, \ldots, a_\ell$, $b_1, \ldots, b_J\in \extspan_\R\CQ$  
     satisfying $$\deg a_1 = \max_j\deg a_j > m_1,$$
    systems $(X, \CX, \mu, T_1, \ldots, T_\ell)$, 
    and functions $f_1, \ldots, f_\ell\in L^\infty(\mu)$, we have 
     \begin{align}\label{E: gen box bound}
         \lim_{N\to\infty}\sup_{\substack{\CD_1, \ldots, \CD_J \in \FD_d}} \sup_{|c_n|\leq 1}\norm{\E_{n\in [N]} c_{n}\cdot \prod_{j\in[\ell]}  T_{j}^{\floor{a_j(n)}} f_j \cdot \prod_{j\in[J]}\CD_j(\floor{b_j(n)})}_{L^2(\mu)} = 0
    \end{align}
    whenever $\nnorm{f_1}_{\bv_{1},\dots,\bv_{s}}=0$
    for some $s = O_{d,J,\ell, \CQ}(1)$ and vectors $\bv_1, \ldots, \bv_s$ from the set of the leading coefficients of 
    $$\{a_1 \be_1,\; a_1 \be_1 - a_2 \be_2,\; \ldots,\; a_1 \be_1 - a_\ell \be_\ell\}.$$
    \end{definition}

    Any ordered Hardy family $\CQ$ is good for smoothing via \cite[Corollary~6.4 \& Theorem~10.2]{DKKST24},  where $m_1$ in Definition~\ref{D: degree and independence} is the same as $m_1$ in Definition~\ref{D: ordered Hardy family}. Theorem~\ref{T: new estimates} is then a consequence of the following general result.
         \begin{theorem}\label{T: new estimates general}
        Let $\ell\in\N$. Suppose that an ordered family of sequences $\CQ$ is {good for smoothing}, and let $a_1, \ldots, a_\ell\in\extSpan_\R \CQ$ with $\deg(a_{1})>m_{1}$.
        Then there exists a positive integer $s=O_{\ell, \CQ}(1)$ and vectors $\bv_{1},\dots,\bv_{s}$ coming from the set of the leading coefficients of 
        \begin{align*}
            \{a_1 \be_1\}\cup\{a_i \be_i - a_j \be_j:\; i\neq j,\; a_i, a_j\; \textrm{are\; dependent and } \deg(a_{i})=\deg(a_{j})\geq \deg(a_{1})\},
        \end{align*}
         such that for all systems $(X, \CX, \mu, T_1, \ldots, T_\ell)$ and functions $f_1, \ldots, f_\ell\in L^\infty(\mu)$, we have 
\begin{align*}
                  \lim_{N\to\infty} 
         \norm{\E_{n\in [N]} 
         \prod_{j\in[\ell]}  T_{j}^{\sfloor{a_{j}(n)}} f_{j}}_{L^2(\mu)}=0
    \end{align*}
    whenever $\nnorm{f_1}_{\bv_{1},\dots,\bv_{s}}=0$. 
    \end{theorem}

In the proof of Theorem \ref{T: new estimates general}, it will matter a lot whether pairs of sequences $a_i, a_j$ are dependent or not. We therefore state a slight generalization of Theorem \ref{T: new estimates general} that makes it easier to read off the dependence or independence of particular sequences and better matches with the way the proof goes. In what follows, let $\hat{a}_{i}$ denote the leading coefficient of $a_{i}$. 
     
     \begin{proposition}\label{newestimate3}
        Let $d, r, \ell_{1},\dots,\ell_{r}\in\N$, $\ell=\ell_{1}+\dots+\ell_{r}$, and $J\in\N_0$. Suppose that an ordered family of sequences $\CQ$
        is {good for smoothing}, and let $a_{i,j}\in\extSpan_\R \CQ$ (for $i\in[r], j\in[\ell_i]$) be sequences
        with the following properties:
        \begin{enumerate}
             \item $a_{i,j}$ is independent from $a_{i',j'}$ if and only if $i\neq i'$;
              \item 
              there exists $r'\in[r]$ such that for all $i\in[r]$, $j\in[\ell_1]$, and $j'\in[\ell_i]$, we have that $\deg(a_{1,j})>\deg(a_{i,j'})$ if and only if $i>r'$;
             \item $\deg(a_{1,1})>m_{1}$.
        \end{enumerate}

                 Then there exists a positive integer $s=O_{d,J,\ell, \CQ}(1)$\footnote{Note that $r,r'\leq \ell$, and so the dependence of $s$ and other parameters on $r, r'$ is implicit in their dependence on $\ell$.} and vectors $\bv_{1},\dots,\bv_{s}$ coming from the set 
                 \begin{align*}
                     \CG:=\{\hat{a}_{1,1}\be_{1,1}, \hat{a}_{i,j}\be_{i,j}-\hat{a}_{i,j'}\be_{i,j'}:\; i\in[r'],\; j,j'\in[\ell_{i}],\; j\neq j'\}
                 \end{align*}
         such that for all systems $(X, \CX, \mu, (T_{i,j})_{\substack{i\in[r], j\in[\ell_i]}})$,  functions $f_{i,j}\in L^\infty(\mu)$ and sequences $b_1, \ldots, b_J\in\Span_\R\CQ$, we have 
\begin{multline}\label{goal}
         \limsup_{N\to\infty}\sup_{\substack{
         \CD_1, \ldots, \CD_J \in \FD_d}} \sup_{|c_n|\leq 1}\norm{\E_{n\in [N]} c_{n}\cdot \prod_{i\in[r]}\prod_{j\in[\ell_i]}  T_{i,j}^{\floor{a_{i,j}(n)}} f_{i,j} \cdot \prod_{j\in[J]}\CD_j(\floor{b_j(n)})}_{L^2(\mu)}
    \end{multline}
    is zero whenever $\nnorm{f_1}_{\bv_{1},\dots,\bv_{s}}=0$. 
    \end{proposition}
We observe that the set $\CH$ does not depend on the ``lower-degree'' sequences $a_{i,j}$ with $r'<i\leq r$ and $j\in[\ell_i]$. This is one way in which Proposition \ref{newestimate3} is stronger than Theorem \ref{T: new estimates general}; the other is the presence of dual functions in \eqref{goal}.

\subsection{A motivating example}\label{SS: concatenation example}

Before we give the proof of Proposition \ref{newestimate3} in full, we discuss its strategy for the average
\begin{align}\label{E: smoothing example average}
    \limsup_{N\to\infty}\norm{\E_{n\in [N]} T_{1,1}^{a_{1}(n)} f_{1,1}\cdot T_{1,2}^{a_{1}(n)} f_{1,2}\cdot T_{2,1}^{a_{2}(n)} f_{2,1}\cdot T_{2,2}^{a_{2}(n)} f_{2,2}}_{L^2(\mu)},
\end{align}
where, say, $a_1(n) = \sfloor{n^{3/2}}$ and $a_2(n) = \sfloor{n^{3/2}+n\log n}$ (what matters is that $a_1, a_2$ are independent and have the same leading term $n^{3/2}$). Throughout, all the functions are 1-bounded. While we strive to make this section as self-contained as possible, we strongly encourage the reader to first familiarize themselves with the examples of smoothing carried out in \cite[Section 8.1]{FrKu22a} and \cite[Section 4]{FrKu22b} for more details on various technical steps (e.g. dual-difference interchange) on which we will not elaborate in this expository section.

Theorem \ref{T: old estimates} gives us the following control over \eqref{E: smoothing example average}, which serves as a starting point of our argument.
\begin{proposition}\label{P: old estimates example}
    There exists $s\in\N$ and vectors
    \begin{align}\label{E: old vectors example}
        \bv_1, \ldots, \bv_s\in \{\be_{1,1},\; \be_{1,1}-\be_{1,2},\; \be_{1,1}-\be_{2,1},\; \be_{1,1}-\be_{2,2}\}.
    \end{align}
    such that the average \eqref{E: smoothing example average} is $0$ whenever $\nnorm{f_{1,1}}_{\bv_1, \ldots, \bv_s}=0$.
\end{proposition}
To get the conclusion of Theorem \ref{T: new estimates} (or Proposition \ref{newestimate3}) we want the following.
\begin{proposition}\label{P: new estimates example}
    There exists $s\in\N$ and vectors 
    \begin{align}\label{E: new vectors example}
    \bv_1, \ldots, \bv_s\in\{\be_{1,1},\; \be_{1,1}-\be_{1,2},\; \be_{2,1}-\be_{2,2}\}.
\end{align}
    such that the average \eqref{E: smoothing example average} is $0$ whenever $\nnorm{f_{1,1}}_{\bv_1, \ldots, \bv_s}=0$.
\end{proposition}
    
The goal thus is to show that each vector $\be_{1,1}-\be_{2,j}$ among \eqref{E: old vectors example} can be replaced by a bounded number of vectors from \eqref{E: new vectors example} so that the resulting seminorm still controls \eqref{E: smoothing example average}. In doing so, we will pass to auxiliary averages of the form
\begin{align}\label{E: smoothing example average eta}
        \limsup_{N\to\infty}\norm{\E_{n\in [N]} T_{\eta(1,1)}^{a_{1}(n)} f_{1,1}\cdot T_{\eta(1,2)}^{a_{1}(n)} f_{1,2}\cdot T_{\eta(2,1)}^{a_{2}(n)} f_{2,1}\cdot T_{\eta(2,2)}^{a_{2}(n)} f_{2,2}}_{L^2(\mu)}
\end{align}
for various choices $\eta: I\to I$ (here, $I = \{(1,1), (1,2), (2,1), (2,2)\}$), possibly twisted by weights $c_n$ or dual terms $\CD_j$. For convenience, we say that \eqref{E: smoothing example average eta} has \textit{type} $\eta$
or $$(\eta(1,1),\eta(1,2);\eta(2,1),\eta(2,2)).$$


Much like in previous smoothing arguments \cite{DKKST24, FrKu22b, FrKu22a}, in order to prove Proposition~\ref{P: new estimates example}, we will induct on two parameters: 
\begin{itemize}
    \item a ``complexity'' of the type $\eta$, which can be measured in terms of the variance $\sum\limits_{i,j\in[2]}|\eta\inv(i,j)|^2$ (the higher the variance, the lower the complexity, hence averages in which all transformations are distinct are most complex);
    \item the length of the average, measured by the number of functions appearing in the average (weights $c_n$ or dual terms $\CD_j$ do not count).
\end{itemize}
Hence Proposition \ref{P: new estimates example} can be deduced inductively from the following two statements, one providing seminorm control for averages of types $((2,j),(1,2);(2,1),(2,2))$ (for $j\in[2]$), and the other dealing with averages of length 3.
\begin{claim}\label{C: lower type example}
    Fix $j\in[2]$. There exists $s\in\N$ such that 
    \begin{align*}
    \limsup_{N\to\infty}\norm{\E_{n\in [N]} T_{2,j}^{a_{1}(n)} f_{1,1}\cdot T_{1,2}^{a_{1}(n)} f_{1,2}\cdot T_{2,1}^{a_{2}(n)} f_{2,1}\cdot T_{2,2}^{a_{2}(n)} f_{2,2}}_{L^2(\mu)} = 0
    \end{align*}
    whenever $\nnorm{f_{1,2}}_{\be_{1,2}^{\times s}, (\be_{1,2}-\be_{2,j})^{\times s}, (\be_{2,1}-\be_{2,2})^{\times s}}=0$.
\end{claim}
\begin{claim}\label{C: lower length example}
    Fix $j\in[2]$ and $d\in\N$. There exists $s\in\N$ such that for every level-$d$ dual function $\CD$, we have
    \begin{align*}
    \limsup_{N\to\infty}\norm{\E_{n\in [N]} T_{1,1}^{a_{1}(n)} f_{1,1}\cdot \CD({a_{1}(n)})\cdot T_{2,1}^{a_{2}(n)} f_{2,1}\cdot T_{2,2}^{a_{2}(n)} f_{2,2}}_{L^2(\mu)} = 0
    \end{align*}
    whenever $\nnorm{f_{1,1}}_{\be_{1,1}^{\times s}, (\be_{2,1}-\be_{2,2})^{\times s}}=0$.
\end{claim}
Note that the seminorm appearing in both claims above are of the form provided by Proposition \ref{newestimate3}: we can therefore think of both claims as the induction hypothesis.


\begin{proof}[Sketch of proof of Proposition \ref{P: new estimates example} assuming Claims \ref{C: lower type example} and \ref{C: lower length example}]
   We now outline the proof of Proposition \ref{P: new estimates example} assuming the preceding auxiliary results. Our focus will be on the sequence of necessary reductions to averages of lower complexity or length rather than on the exact way in which these reductions are carried out. To get the sense of the technicalities involved in the latter, the reader should consult the examples from \cite[Section 8.1]{FrKu22a} and \cite[Section 4]{FrKu22b}.

   Our goal is to show that \eqref{E: smoothing example average} is 0 if $\nnorm{f_{1,1}}_{\bv_{1},\dots,\bv_{s}}=0$ for some $s\in\N$ and $$\bv_{1},\dots,\bv_{s}\in \mathcal{G}:=\{\be_{1,1},\; \be_{1,1}-\be_{1,2},\; \be_{2,1}-\be_{2,2}\}.$$

By Proposition \ref{P: old estimates example}, we have a preliminary control which asserts that \eqref{E: smoothing example average} is controlled by
  $\nnorm{f_{1,1}}_{\bv_{1},\dots,\bv_{t},\bu_{1},\dots,\bu_k}$ for some $k,t\in\N$ and vectors $\bv_{1},\dots,\bv_{t}\in \mathcal{G}$ (in fact, they lie in $\{\be_{1,1},\; \be_{1,1}-\be_{1,2}\}$) and $\bu_{1},\dots,\bu_k\in\{\be_{1,1}-\be_{2,1},\be_{1,1}-\be_{2,2}\}$. 
  Our aim is to remove $\bu_k$ at the cost of enlarging the number $t$, and then iterate this procedure with other $\bu_1, \ldots, \bu_{k-1}$.

 \smallskip
   \textbf{Step 1: Ping.} 
   \smallskip
   
   This step emulates the Ping step from the earlier smoothing arguments \cite{DKKST24, FrKu22b, FrKu22a}. Assume for concreteness that $\bu_k=\be_{1,1}-\be_{2,1}$ (the case $\bu_k=\be_{1,1}-\be_{2,2}$ being handled analogously). Then the Ping step essentially asserts that if some $\nnorm{f_{1,1}}_{\bold{w}_{1},\dots,\bold{w}_{r},\bu_k}$ controls (\ref{E: smoothing example average}) and if some $\nnorm{f_{i,j}}_{\bold{w}'_{1},\dots,\bold{w}'_{r'}}$ with $(i,j)\neq (1,1)$ controls
   \begin{equation}\label{E: smoothing example average ping}
    \limsup_{N\to\infty}\norm{\E_{n\in [N]} T_{2,1}^{a_{1}(n)} f_{1,1}\cdot T_{1,2}^{a_{1}(n)} f_{1,2}\cdot T_{2,1}^{a_{2}(n)} f_{2,1}\cdot T_{2,2}^{a_{2}(n)} f_{2,2}}_{L^2(\mu)}
\end{equation}
(which is an expression derived from (\ref{E: smoothing example average}) by replacing the transformation $T_{1,1}$ with $T_{2,1}$), then (\ref{E: smoothing example average}) is controlled by $\nnorm{f_{i,j}}_{\bold{w}_{1},\dots,\bold{w}_{r}, \bold{w}'_{1},\dots,\bold{w}'_{r'}}$.

In our case, Claim \ref{C: lower type example} gives us that \eqref{E: smoothing example average ping} is controlled by $\nnorm{f_{1,2}}_{\bu'_{1},\dots,\bu'_{k'},\be_{1,2}^{\times d}}$ for some $d, k'\in\N$ and $\bu'_{1},\dots,\bu'_{k'}\in \{\be_{1,2}-\be_{2,1},\be_{2,1}-\be_{2,2}\}$. 
Therefore, we conclude that (\ref{E: smoothing example average}) is controlled by 
\begin{align}\label{E: auxiliary ping control}
    \nnorm{f_{1,2}}_{\bv_{1},\dots,\bv_{t},\bu_{1},\dots,\bu_{k-1},\bu'_{1},\dots,\bu'_{k'}, \be_{1,2}^{\times d}}.
\end{align}

 \smallskip
   \textbf{Step 2: Pong.} 
   \smallskip 
   
This step again emulates the Pong step from \cite{DKKST24, FrKu22b, FrKu22a}, and it essentially asserts that if some $\nnorm{f_{1,2}}_{\bold{w}_{1},\dots,\bold{w}_{r},\be_{i,j}^{\times d}}$ controls (\ref{E: smoothing example average}), and if $\nnorm{f_{1,1}}_{\bold{w}'_{1},\dots,\bold{w}'_{r'}}$ controls
   \begin{equation}\label{E: smoothing example average pong}
    \limsup_{N\to\infty}\norm{\E_{n\in [N]} T_{1,1}^{a_{1}(n)}f_{1,1}\cdot \CD(a_1(n)) \cdot T_{2,1}^{a_{2}(n)} f_{2,1}\cdot T_{2,2}^{a_{2}(n)} f_{2,2}}_{L^2(\mu)},
\end{equation}
where $\CD$ is a level-$d$ dual sequence of $T_{1,2}$,
then $\nnorm{f_{1,1}}_{\bold{w}_{1},\dots,\bold{w}_{r},\bold{w}'_{1},\dots,\bold{w}'_{r'}}$ controls (\ref{E: smoothing example average}). 

In our case, it follows from the Ping step that (\ref{E: smoothing example average}) is controlled by \eqref{E: auxiliary ping control}. 
On the other hand, Claim \ref{C: lower length example} gives control of (\ref{E: smoothing example average pong}) by $\nnorm{f_{1,1}}_{\bu''_{1},\dots,\bu''_{k''}}$ for some $k''\in\N$ and $\bu''_{1},\dots,\bu''_{k''}\in \mathcal{G}$. So we have that (\ref{E: smoothing example average}) is controlled by 
\begin{align}\label{E: auxiliary pong control}
    \nnorm{f_{1,1}}_{\bv_{1},\dots,\bv_{t},\bu_{1},\dots,\bu_{k-1},\bu'_{1},\dots,\bu'_{k'},\bu''_{1},\dots,\bu''_{k''}}.
\end{align}

 \smallskip
\textbf{Step 3: Relative concatenation}.
 \smallskip

This step is a completely new contribution to the seminorm smoothing machinery, and it crucially relies on the relative concatenation results from Section \ref{S: concatenation}. The preliminary estimate and the conclusion from the Pong step asserts that (\ref{E: smoothing example average}) is zero when $f_{1,1}$ is orthogonal to one of
\begin{align*}
    \mathcal{Z}^+_{\bv_{1},\dots,\bv_{t},\bu_{1},\dots,\bu_k}\quad \textrm{or}\quad  \mathcal{Z}^+_{\bv_{1},\dots,\bv_{t},\bu_{1},\dots,\bu_{k-1},\bu'_{1},\dots,\bu'_{k'},\bu''_{1},\dots,\bu''_{k''}}. 
\end{align*}
It then follows from Corollary \ref{C: relative concatenation} that (\ref{E: smoothing example average}) is zero if $f_{1,1}$ is orthogonal to $$\mathcal{Z}^+_{\bv_{1},\dots,\bv_{t},\bu_{1},\dots,\bu_{k-1},\{\langle \bu_k,\bu \rangle:\; u\in\mathcal{R}\}},$$ or equivalently if $$\nnorm{f_{1,1}}^+_{\bv_{1},\dots,\bv_{t},\bu_{1},\dots,\bu_{k-1}, \{\langle \bu_k,\bu \rangle:\; u\in\mathcal{R}\}}=0,$$ where $\mathcal{R}$ is the multi-set $\{\bu'_{1},\dots,\bu'_{k'},\bu''_{1},\dots,\bu''_{k''}\}$.
However, for any $\bu\in\mathcal{R}$,  either $\bu\in\mathcal{G}$ or $\bu=\be_{1,2}-\be_{2,1}$. Now, for any such $\bu$, the group $\langle \bu_k,\bu \rangle$ contains an element of $\CG$. In the former case, this is immediate, while in the latter case, the group $\langle \bu_k,\bu \rangle=\langle \be_{1,1}-\be_{2,1}, \be_{1,2}-\be_{2,1} \rangle$ contains $\be_{1,1}-\be_{1,2}$. By Lemma \ref{L: basic properties of seminorms}, this implies that  (\ref{E: smoothing example average}) is controlled by $\nnorm{f_{1,1}}_{\bv_{1},\dots,\bv_{t'},\bu_{1},\dots,\bu_{k-1}}=0$ for some $t'\in\N$ and $\bv_{1},\dots,\bv_{t'}\in\mathcal{G}$.

We may then repeat this process to remove all of $\bu_{1},\dots,\bu_{k-1}$ and end up with a control of (\ref{E: smoothing example average}) by $\nnorm{f_{1,1}}_{\bv_{1},\dots,\bv_{t''}}=0$ for some $t''\in\N$ and $\bv_{1},\dots,\bv_{t''}\in\mathcal{G}$,  as desired.
    \end{proof}

\subsection{Proof of Proposition \ref{newestimate3}}
We move on to prove Proposition \ref{newestimate3}. From now on, we fix a system $(X, \CX, \mu, (T_{i,j})_{\substack{i\in[r], j\in[\ell_i]}})$.




\smallskip
\textbf{Case 1: $a_{1,1}$ has the maximum degree.}
\smallskip

We first prove this result under the assumption that $a_{1,1}$ has the maximum degree. 
In this case, all of $a_{i,j}$'s (for $i\in[r']$ and $j\in[\ell_{i}]$) have the same degree greater than $m_{1}$.
 We first consider the case when $\ell_{1}+\dots+\ell_{r'}=1$. Then $a_{1,1}$ has a higher degree than all other $a_{i,j}$'s, and so the maps $a_{1,1}$ and $a_{1,1}-a_{i,j}$ have the same leading coefficients. So the fact that $\CQ$ is {good for smoothing} implies that (\ref{goal}) is zero if $\nnorm{f_1}_{\be_{1,1}^{\times s}}=0$ for some $s=O_{d,J,\ell, \CQ}(1)$.

Now suppose that the conclusion holds when $\ell_{1}+\dots+\ell_{r'}=\ell'-1$ for some $\ell'\geq 2$; our aim is to show that it also holds when $\ell_{1}+\dots+\ell_{r'}=\ell'$. The argument below will make use of the following definitions.
\begin{definition}[Types, minority images and pre-images, and good terms]
    Let $$I:=\{(i,j)\in\N^{2}\colon\; i\in[r], j\in[\ell_{i}]\}\quad \text{and}\quad I':=\{(i,j)\in\N^{2}\colon\; i\in[r'], j\in[\ell_{i}]\}.$$
    We call a map $\eta\colon I\to I$ a \textit{type}, and we call it \emph{proper} if it satisfies the following conditions:
    \begin{enumerate}
        \item $\eta(I') \subseteq I'$;
        \item $\eta|_{I\backslash I'}$ is an identity;
        \item $\eta(i,j)\neq \eta(i,j')$ for all $i\in[r']$ and distinct $ j,j'\in[\ell_{i}]$.
    \end{enumerate}
    We say that $\tilde{i}\in I'$ is a \emph{minority image} of $\eta$ if
   $$\vert \eta^{-1}(\tilde{i})\vert=\min_{\tilde{i}'\in I'}\max(\vert\eta^{-1}(\tilde{i}')\vert,1).$$
   In other words, $\be_{\tilde{i}}$ is a vector  that appears least frequently in the multi-set $$\{\be_{\eta(i,j)}\colon\; i\in[r'],\; j\in [\ell_{i}]\}$$
(but it does appear at least once). We call every element in $\eta^{-1}(\tilde{i})$ a \emph{minority pre-image} of $\eta$. 
   Lastly, we say that $(i_{0},j_{0})$ is a \emph{good term of} $\eta$  if
   there exists a positive integer $s=O_{d,J,\ell,\CQ}(1)$ and vectors $\bv_{1},\dots,\bv_{s}$ coming from the set
             $$
             \{\hat{a}_{\eta(i_{0},j_{0})}\be_{\eta(i_{0},j_{0})},\hat{a}_{\eta(i,j)}\be_{\eta(i,j)}-\hat{a}_{\eta(i,j')}\be_{\eta(i,j')}:\; i\in[r'],\; j,j'\in[\ell_{i}],\; j\neq j'\}$$
         such that for all 1-bounded functions $f_{i,j}\in L^\infty(\mu)$ and sequences $b_1, \ldots, b_J\in\Span_\R\CQ$, the average 
\begin{multline}\label{E: average to be controled}
         \limsup_{N\to\infty}\sup_{\substack{
         \CD_1, \ldots, \CD_J \in \FD_d}} \sup_{|c_n|\leq 1}\norm{\E_{n\in [N]} c_{n}\cdot \prod_{i\in[r]}\prod_{j\in[\ell_{r}]}  T_{\eta(i,j)}^{\floor{\frac{\hat{a}_{\eta(i,j)}}{\hat{a}_{i,j}}a_{i,j}(n)}} f_{i,j}\cdot \prod_{j\in[J]}\CD_j(\floor{b_j(n)})}_{L^2(\mu)}
    \end{multline}
    is zero whenever $\nnorm{f_{i_{0},j_{0}}}_{\bv_{1},\dots,\bv_{s}}=0$.

\end{definition}

   To prove the statement, it suffices to show that any minority pre-image of any proper $\eta$  is good, and then apply this to the identity map $\eta(\tilde{i})=\tilde{i}$, which is proper and in which every term is minority.

    The next definition explains how we pass from more to less complex types.
    \begin{definition}[Descendance] 
   For $\tilde{i},\tilde{i}'\in I'$, let $\sigma_{\tilde{i}\to \tilde{i}'}(\eta)\colon I\to I$ be the map given by $\sigma_{\tilde{i}\to \tilde{i}'}(\eta)(\tilde{i})=\eta(\tilde{i}')$ and $\sigma_{\tilde{i}\to \tilde{i}'}(\eta)(\tilde{j})=\eta(\tilde{j})$ for all  $\tilde{j}\in I\backslash\{\tilde{i}\}$. For proper maps $\eta$ and $\eta'$, we say that $\eta$ \emph{descends to} $\eta'$ (write as $\eta\searrow \eta'$) if $\eta'=\sigma_{\tilde{i}\to \tilde{i}'}(\eta)$ for some $\tilde{i},\tilde{i}'\in I'$ such that $\eta(\tilde{i})\neq \eta(\tilde{i}')$ and $\vert\eta^{-1}(\eta(\tilde{i}))\vert\leq \vert\eta^{-1}(\eta(\tilde{i}'))\vert$.     
    \end{definition}
   It is not hard to see that
   \begin{equation}\nonumber
      \eta\searrow \eta'\Rightarrow \sum_{\tilde{i}\in I}\vert\eta^{-1}(\tilde{i})\vert^{2}<\sum_{\tilde{i}\in I}\vert{\eta'}^{-1}(\tilde{i})\vert^{2},
   \end{equation}
 as an immediate consequence of which we have that the relation $\searrow$ is \emph{loopless}, meaning that there exist no $n\in\N$ and no proper types $\eta_{1},\dots,\eta_{n}$ with $\eta_{1}\searrow \eta_{2}\searrow\dots\searrow\eta_{n}\searrow\eta_{1}$.

 The base case of induction will come from averages of the following type.\footnote{A reader familiar with the earlier seminorm smoothing argument might inquire about the connection between extreme and basic types, the latter of which were playing the role of the base case in the smoothing arguments from \cite{DKKST24, FrKu22b, FrKu22a}. Despite superficial similarity, these two notions are rather different. For instance, the smoothing argument for pairwise dependent integer polynomials from \cite{FrKu22b} would allow a descendance $((1,1),(1,2);(2,1),(2,2))\searrow ((2,1),(1,2);(2,1),(2,2))\searrow((2,1),(2,1);(2,1),(2,2))$. The latter reduction is not permissible in this smoothing argument even though it is perfectly allowed in \cite{FrKu22b}. The type $((2,1),(2,1);(2,1),(2,2))$ reached this way is basic in the sense of \cite{FrKu22b}, but neither proper not extreme according to this paper's convention. Conversely, the arguments of this paper allow us to reduce types via the descendance $((1,1),(1,2), (1,3);(2,1),(2,2))\searrow ((1,1),(1,2), (2,1);(2,1),(2,2))\searrow ((1,1),(1,2), (2,1);(2,1),(1,2))$. The type reached at the end is extreme but not basic as defined in \cite{FrKu22b}.}
 \begin{definition}[Extreme types]
 We say that a proper map $\eta$ is \emph{extreme} if there is no proper map $\eta'$ with $\eta\searrow\eta'$.      
 \end{definition}

 For any proper map $\eta$, denote $w(\eta)=0$ if $\eta$ is extreme. Otherwise, 
 let $w(\eta)$ denote the largest number $n\in\N$ such that $\eta=\eta_{0}\searrow \eta_{1}\searrow\dots\searrow\eta_{n}$ for some proper maps $\eta_{0},\dots,\eta_{n}$ with $\eta_{n}$ being extreme. Since the relation $\searrow$ is loopless, $w(\eta)$ is well defined and is at most $O_{\ell}(1)$.  
    
   Using the fact that $\eta\searrow\eta'\Rightarrow w(\eta')<w(\eta)$, 
   to complete the proof, it suffices to prove the following claim (and then apply it repeatedly $O_\ell(1)$ times).

\begin{claim}[Minority pre-images are good]\label{C: smoothing claim}
    Let $\eta$ be a proper type. Suppose that for any proper type $\eta'$ with 
   $\eta\searrow\eta'$, any minority pre-image of $\eta'$ is good. Then any minority pre-image of $\eta$ is good (in particular, any minority pre-image of $\eta$ is good when $\eta$ is extreme).
\end{claim}

\begin{example}
    This example shows how Claim \ref{C: smoothing claim} is used iteratively in the proof of Proposition \ref{P: new estimates example}. Phrased in the language of this section, the purpose of this proposition was to show that $(1,1)$ is a good term for $\eta = ((1,1),(1,2);(2,1),(2,2))$. Note that $\eta\searrow\eta'$ by changing the first entry of $\eta$ if and only if $\eta'$ is of type $((2,1),(1,2);(2,1),(2,2))$ or $((2,2),(1,2);(2,1),(2,2))$ (the other possible type $((1,2),(1,2);(2,1),(2,2))$ is not proper). Claim \ref{C: lower type example}, which showed that $(1,2)$ is good for either of these two proper types, can be seen as a special case of Claim \ref{C: smoothing claim}. 
    
    Recall from the proof of Proposition \ref{P: new estimates example} that Claim \ref{C: lower type example} was applied in the Ping step of the proof. This will remain true more generally: Claim \ref{C: smoothing claim} will always turn up in the Ping step of the smoothing argument.

   To prove Claim \ref{C: lower type example}, we would similarly want to show that the second term of $\eta' = ((2,1),(1,2);(2,1),(2,2))$ is good (and analogously with $((2,2),(1,2);(2,1),(2,2))$). Note that  $\eta'\searrow\eta''$ by changing the second entry of $\eta$ if and only if $\eta''=((2,1),(2,2);(2,1),$ $(2,2))$  (the other possible type $((2,1),(2,1);(2,1),(2,2))$ is not proper). This is an extreme type in which all four terms are minorities. Seminorm estimates for this type that are in line with Proposition \ref{newestimate3} will follow directly from the assumption of being good for smoothing.
   Indeed, the expression we wish to control is of the form
   $$\limsup_{N\to\infty}\norm{\E_{n\in [N]} T_{2,1}^{a_{1}(n)} f_{1,1}\cdot T_{2,2}^{a_{1}(n)} f_{1,2}\cdot T_{2,1}^{a_{2}(n)} f_{2,1}\cdot T_{2,2}^{a_{2}(n)} f_{2,2}}_{L^2(\mu)},$$
and since $\CQ$ is good for smoothing, this average is 0 if $\nnorm{f_{1,1}}_{\be_{2,1}^{\times s}, (\be_{2,1}-\be_{2,2})^{\times s}}=0$ for some $s\in\N$. This implies that the first term of $\eta''$ is good, and an analogous argument would give desired seminorm control in terms of seminorms of all the other terms.
\end{example}

We proceed to prove Claim \ref{C: smoothing claim}. Assume without loss of generality that $(1,1)$ is a minority pre-image of $\eta$. 
Let 
\begin{align*}
    \CG_\eta :=\{\hat{a}_{\eta(1,1)}\be_{\eta(1,1)}, \hat{a}_{\eta(i,j)}\be_{\eta(i,j)}-\hat{a}_{\eta(i,j')}\be_{\eta(i,j')}:\; i\in[r'],\; j,j'\in[\ell_{i}],\; j\neq j'\}.
\end{align*}
The properness of $\eta$ implies in particular that the elements of $\CG_\eta$ are nonzero.
It suffices to show that (\ref{E: average to be controled})  can be controlled by a box seminorm of $f_{1,1}$ involving only transformations from $\CG_{\eta}$, which we will do by establishing the following property for $k = 0$.

\begin{definition}[Property $k$]
    We say that \emph{Property $k$} holds if there exist $t\in\N$ and vectors  $\bv_1, \ldots, \bv_t, \bu_1, \ldots, \bu_k\in\R^\ell$ (where $\bv_i\in \CG_{\eta}$ while $\bu_i$'s are the leading coefficients of $\frac{\hat{a}_{\eta(1,1)}}{\hat{a}_{1,1}}a_{1,1}\be_{\eta(1,1)}-\frac{\hat{a}_{\eta(i,j)}}{\hat{a}_{i,j}}a_{i,j}\be_{\eta(i,j)}$ for some $(i,j)\in I'\backslash\{(1,1)\}$), such that for all functions $f_{i,j}\in L^\infty(\mu)$ and sequences $b_1, \ldots, b_J\in\Span_\R\CQ$, the average \eqref{E: average to be controled} is 0 whenever $$\nnorm{f_{1,1}}_{\bv_{1},\dots,\bv_{t},\bu_{1},\dots,\bu_k}=0.$$
\end{definition}
Since $\CQ$ is good for smoothing 
and since $$\deg(a_{1,1})=\dots=\deg(a_{r',1})>\max_{r'< i\leq r}\max_{j\in[\ell_i]}\deg(a_{i,j}),$$  Property $k$ holds for some $k\in\N$.     
 Our goal is to show that Property 0 holds. To achieve this, it suffices to show that if Property $k$ holds for some $k\in\N$, then Property $k-1$ also holds.
     

          Assume that 
$\bu_k$ is the leading coefficient of $$\frac{\hat{a}_{\eta(1,1)}}{\hat{a}_{1,1}}a_{1,1}\be_{\eta(1,1)}-\frac{\hat{a}_{\eta(i_{0},j_{0})}}{\hat{a}_{i_{0},j_{0}}}a_{i_{0},j_{0}}\be_{\eta(i_{0},j_{0})}$$ for some $(i_{0},j_{0})\in I'\backslash\{(1,1)\}$. We consider several possible scenarios.

If $\eta(i_{0},j_{0})=\eta(1,1)$, then since $\eta$ is proper, we have that $i_{0}\geq 2$ and thus $a_{1,1}$ is independent from $a_{i_{0},j_{0}}$. This implies that $\bu_k$ is a nonzero multiple of $\be_{\eta(1,1)}$. Since $\CG_{\eta}$ contains $\hat{a}_{\eta(1,1)}\be_{\eta(1,1)}$, Property $k-1$ holds by Lemma \ref{L: basic properties of seminorms} after relabeling $\bu_k$ as $\bv_{t+1}$. 

If instead $\eta(i_{0},j_{0})=\eta(1,j)$ for some $2\leq j\leq \ell_1$, then $\bu_k$ belongs to $\CG_{\eta}$ and again Property $k-1$ holds by by relabeling $\bu_k$ as $\bv_{t+1}$.
       
       It remains to consider the case when $\eta(i_{0},j_{0})\neq \eta(1,j)$ for all $1\leq j\leq \ell_1$. In this case,  
       $i_{0}\neq 1$. 
       Denote $\eta':=\sigma_{(1,1)\to (i_{0},j_{0})}(\eta)$.
  %
     Since 
  $$\eta'(1,1)=\eta(i_{0},j_{0})\neq \eta(1,j)=\eta'(1,j)$$ for all $2\leq j\leq \ell_1$, 
     the map $\eta'$ is proper. Moreover, $\eta\searrow \eta'$ 
     as a consequence of $(1,1)$ being a minority pre-image of $\eta$.

      In particular, if $\eta$ is extreme (i.e. $w(\eta) = 0$), then the non-existence of such $\eta'$ forces $\eta(i_{0},j_{0})=\eta(1,j)$ for some $j\in[\ell_1]$, and we are back in one of the previous cases. 
      
      So, in what follows we assume that $\eta(i_{0},j_{0})\neq \eta(1,j)$ for all $1\leq j\leq \ell_1$ since all the remaining cases can be covered, and we take $\eta'=\sigma_{(1,1)\to (i_{0},j_{0})}(\eta)$.
      
     
     

     \smallskip
     \textbf{Step 1: Ping.}
     \smallskip
     
     Suppose  that (\ref{E: average to be controled}) is positive.
     It follows from  \cite[Lemma 11.7]{DKKST24} that 
     \begin{multline*}
         \limsup_{N\to\infty}\sup_{\substack{
         \CD_1, \ldots, \CD_J \in \FD_d}} \sup_{|c_n|\leq 1}\lim_{N\to\infty}\Bigl\Vert\E_{n\in[N]}c_{n}\cdot T_{\eta(1,1)}^{\floor{\frac{\hat{a}_{\eta(1,1)}}{\hat{a}_{1,1}}a_{1,1}(n)}}\tilde{f}_{1,1}\cdot\prod_{(i,j)\in I\backslash\{(1,1)\}}T_{\eta(i,j)}^{\floor{\frac{\hat{a}_{\eta(i,j)}}{\hat{a}_{i,j}}a_{i,j}(n)}}f_{i,j}\\
         \prod_{j\in[J]}\mathcal{D}_{j}(\floor{b_{j}(n)})\Bigr\Vert_{L^2(\mu)}>0
     \end{multline*}
     for some $\tilde{f}_{1,1}$ which is ``structured'' in that it is the weak limit $\lim\limits_{w\to\infty} A_{N_{w}}$ for 
\begin{multline}\label{E: tilded}
             A_{N_{w}}:=\E_{n\in[N_{w}]} \overline{c}'_{n}T_{\eta(1,1)}^{-\floor{\frac{\hat{a}_{\eta(1,1)}}{\hat{a}_{1,1}}a_{1,1}(n)}}\overline{g}_{w}\prod_{(i,j)\in I\backslash\{(1,1)\}}T_{\eta(i,j)}^{\floor{\frac{\hat{a}_{\eta(i,j)}}{\hat{a}_{i,j}}a_{i,j}(n)}}T_{\eta(1,1)}^{-\floor{\frac{\hat{a}_{\eta(1,1)}}{\hat{a}_{1,1}}a_{1,1}(n)}}\overline{f}_{i,j}\\
         \prod_{j\in [J]}T_{\eta(1,1)}^{-\floor{\frac{\hat{a}_{\eta(1,1)}}{\hat{a}_{1,1}}a_{1,1}(n)}}\overline{\mathcal{D}}'_{j}(\floor{b_{j}(n)})
\end{multline}
     for some increasing sequence $(N_{w})_{w}$,  1-bounded weights $(c'_{n})_n$,  1-bounded functions $g_{w}$, and dual sequences $\mathcal{D}'_{1},\dots,\mathcal{D}'_{J}\in\mathfrak{D}_{d}$. 
     By assumption, we have $$\nnorm{\tilde{f}_{1,1}}_{\bv_{1},\dots,\bv_{t},\bu_{1},\dots,\bu_k}>0.$$
      We then unpack this seminorm using \cite[Lemma 11.11]{DKKST24} and the formula \eqref{E: tilded} for $\tilde{f}_{1,1}$; after composing the resulting expression with $T_{\eta(1,1)}^{\floor{\frac{\hat{a}_{\eta(1,1)}}{\hat{a}_{1,1}}a_{1,1}(n)}}$ and applying the Cauchy-Schwarz inequality, we have  
\begin{multline*}
    \liminf_{M\to\infty}\E_{\um, \um',\um'',\um'''\in [\pm M]^{t+k-1}}\limsup_{N\to\infty} \sup_{|c_n|\leq 1}\Bigl\Vert
   \E_{n\in[N]}c_n\cdot T_{\eta(1,1)}^{\floor{\frac{\hat{a}_{\eta(1,1)}}{\hat{a}_{1,1}}a_{1,1}(n)}} h_{\um,\um',\um'',\um'''}
   \\
   \prod_{(i,j)\in I\backslash\{(1,1)\}}T_{\eta(i,j)}^{\floor{\frac{\hat{a}_{\eta(i,j)}}{\hat{a}_{i,j}}a_{i,j}(n)}}f_{i,j,\um,\um',\um'',\um'''}\cdot \prod_{j\in[J]} \CD'_{j,\um,\um',\um'',\um'''}(\floor{b_j(n)})\Bigr\Vert_{L^2(\mu)}>0
\end{multline*}
        where
        \begin{align*}
            f_{i,j,\um,\um',\um'',\um'''} &= \Delta_{\substack{(\floor{\bv_1 m_1'} - \floor{\bv_1 m_1}, \floor{\bv_1 m_1'''} - \floor{\bv_1 m_1''}), \ldots,\\ (\floor{\bu_{k-1} m_{t+k-1}'} - \floor{\bu_{k-1} m_{t+k-1}}, \floor{\bu_{k-1} m_{t+k-1}'''} - \floor{\bu_{k-1} m_{t+k-1}''})}}\overline{f_{i,j}}\\
            \textrm{and}\quad\CD'_{j,\um,\um',\um'',\um'''} &= \Delta_{\substack{(\floor{\bv_1 m_1'} - \floor{\bv_1 m_1}, \floor{\bv_1 m_1'''} - \floor{\bv_1 m_1''}), \ldots,\\ (\floor{\bu_{k-1} m_{t+k-1}'} - \floor{\bu_{k-1} m_{t+k-1}}, \floor{\bu_{k-1} m_{t+k-1}'''} - \floor{\bu_{k-1} m_{t+k-1}''})}}\overline{\CD'_{j}}
        \end{align*}
    and each $h_{\um,\um',\um'',\um'''}$ is a product of $2^{t+k-1}$ 1-bounded \emph{approximately $\bu_k$-invariant functions}, meaning functions of the form $\E\limits_{n\in\Z}T^{\floor{\bu_kn}}f$ for some $f\in L^\infty(\mu)$.
    Since $\bu_k=\hat{a}_{\eta(1,1)}\be_{\eta(1,1)}-\hat{a}_{\eta'(1,1)}\be_{\eta'(1,1)}$, it follows from \cite[Lemma 11.10]{DKKST24} that
\begin{multline}\label{qtrgg1}
       \liminf_{M\to\infty}\E_{\um, \um',\um'',\um'''\in [\pm M]^{t+k-1}}\limsup_{N\to\infty} \sup_{|c_n|\leq 1}\Bigl\Vert
   \prod_{j\in[J]}\E_{n\in[N]}c_n\cdot T_{\eta'(1,1)}^{\floor{\frac{\hat{a}_{\eta'(1,1)}}{\hat{a}_{1,1}}a_{1,1}(n)}} h'_{\um,\um',\um'',\um'''}
   \\\prod_{(i,j)\in I\backslash\{(1,1)\}}T_{\eta'(i,j)}^{\floor{\frac{\hat{a}_{\eta'(i,j)}}{\hat{a}_{i,j}}a_{i,j}(n)}}f_{i,j,\um,\um',\um'',\um'''}\cdot  \prod_{j\in[J]} \CD'_{j,\um,\um',\um'',\um'''}(\floor{b_j(n)})\Bigr\Vert_{L^2(\mu)}>0
\end{multline}
 for some 1-bounded functions $h'_{\um,\um',\um'',\um'''}$. It is at this step that we have passed from an average of type $\eta$ to one of type $\eta'$, and this key reduction was enabled by the $\bu_k$-almost invariance of $h_{\um,\um',\um'',\um'''}$'s.

      Note that we may find some $\tilde{u}=(u',u'')\in I'\backslash\{(1,1)\}$ which is a minority pre-image of $\eta'$ (to see this, let $\tilde{v}$ be a minority pre-image of $\eta'$; if $\tilde{v}\neq (1,1)$, then we pick $\tilde{u}=\tilde{v}$ whereas if $\tilde{v}=(1,1)$, then we pick $\tilde{u}=(i_{0},j_{0})$ since   $\eta'(1,1)=\eta'(i_{0},j_{0})$). By assumption,
      the $\tilde{u}$-th term is good for $\eta'$.
           Therefore, it follows 
           from our inductive assumption\footnote{And specifically by its soft quantitative strengthening, which is a consequence of \cite[Proposition~A.2]{FrKu22a}.}
           that for a set of $(\um,\um',\um'',\um'')$ of positive lower density, the seminorm
\begin{multline*}
    \nnorm{f_{\tilde{u},\um,\um',\um'',\um'''}}_{\bu'_{1},\dots,\bu'_{t'},(\hat{a}_{\eta(\tilde{u})}\be_{\eta(\tilde{u})})^{\times s'}}
   \\=\Bignnorm{\Delta_{\substack{(\floor{\bv_1 m_1'} - \floor{\bv_1 m_1}, \floor{\bv_1 m_1'''} - \floor{\bv_1 m_1''}), \ldots,\\ (\floor{\bu_{k-1} m_{t+k-1}'} - \floor{\bu_{k-1} m_{t+k-1}}, \floor{\bu_{k-1} m_{t+k-1}'''} - \floor{\bu_{k-1} m_{t+k-1}''})}}\overline{f_{\tilde{u}}}}_{\bu'_{1},\dots,\bu'_{t'},(\hat{a}_{\eta(\tilde{u})}\be_{\eta(\tilde{u})})^{\times s'}}
\end{multline*}
           is (uniformly) bounded away from 0
         for some $s',t'\in\N$ (depending only on $d,J,\ell,\CQ$) and $\bu'_{1},\dots,\bu'_{t'}\in \CG_{\eta'}\backslash\{\hat{a}_{\eta(\tilde{u})}\be_{\eta(\tilde{u})}\}$.
   Averaging over all $\um,\um',\um'',\um''$, raising it to the power $2^{s'+t'}$ using the H{\"o}lder inequality and applying the inductive formula for generalized box seminorms, we get the auxiliary norm control
      $$\nnorm{f_{\tilde{u}}}_{(\hat{a}_{\eta(\tilde{u})}\be_{\eta(\tilde{u})})^{\times s'}, \bu'_{1},\dots,\bu'_{t'},\bv_{1},\dots,\bv_{t},\bu_{1},\dots,\bu_{k-1}}>0.$$

  \smallskip
         \textbf{Step 2: Pong.} 
        \smallskip
         
     Suppose that  (\ref{E: average to be controled}) is positive. 
      It follows from \cite[Lemma 11.7]{DKKST24} that 
     \begin{multline*}
         \limsup_{N\to\infty}\sup_{\substack{
         \CD_1, \ldots, \CD_J \in \FD_d}} \sup_{|c_n|\leq 1}\lim_{N\to\infty}\Bigl\Vert\E_{n\in[N]}c_{n}\cdot T_{\eta(\tilde{u})}^{\floor{\frac{\hat{a}_{\eta(\tilde{u})}}{\hat{a}_{\tilde{u}}}a_{\tilde{u}}(n)}}\tilde{f}_{\tilde{u}}\cdot\prod_{(i,j)\in I\backslash\{\tilde{u}\}}T_{\eta(i,j)}^{\floor{\frac{\hat{a}_{\eta(i,j)}}{\hat{a}_{i,j}}a_{i,j}(n)}}f_{i,j}\\
         \prod_{j\in[J]}\mathcal{D}_{j}(\floor{b_{j}(n)})\Bigr\Vert_{2}>0
     \end{multline*}
     for some $\tilde{f}_{\tilde{u}}$ which is the weak limit $\lim\limits_{w\to\infty} A_{N_{w}}$ for 
\begin{multline}\label{E: tilded2}
    A_{N_{w}}:=\E_{n\in[N_{w}]} \overline{c}'_{n}\cdot T_{\eta(\tilde{u})}^{-\floor{\frac{\hat{a}_{\eta(\tilde{u})}}{\hat{a}_{\tilde{u}}}a_{\tilde{u}}(n)}}\overline{g}_{w}\cdot\prod_{(i,j)\in I\backslash\{\tilde{u}\}}T_{\eta(i,j)}^{\floor{\frac{\hat{a}_{\eta(i,j)}}{\hat{a}_{i,j}}a_{i,j}(n)}}T_{\eta(\tilde{u})}^{-\floor{\frac{\hat{a}_{\eta(\tilde{u})}}{\hat{a}_{\tilde{u}}}a_{\tilde{u}}(n)}}\overline{f}_{i,j}
        \\
       \prod_{j\in [J]}T_{\eta(\tilde{u})}^{-\floor{\frac{\hat{a}_{\eta(\tilde{u})}}{\hat{a}_{\tilde{u}}}a_{\tilde{u}}(n)}}\overline{\mathcal{D}}'_{j}(\floor{b_{j}(n)})
\end{multline}
     for some increasing sequence $(N_{w})_{w}$, 1-bounded weights $(c'_{n})_n$, 1-bounded function $g_{w}$, and dual sequences $\mathcal{D}'_{1},\dots,\mathcal{D}'_{J}\in\mathfrak{D}_{d}$.
    The seminorm control established in the Ping step then gives $$\nnorm{\tilde{f}_{\tilde{u}}}_{(\hat{a}_{\eta(\tilde{u})}\be_{\eta(\tilde{u})})^{\times s'},\bu'_{1},\dots,\bu'_{t'},\bv_{1},\dots,\bv_{t},\bu_{1},\dots,\bu_{k-1}}>0.$$
      We unpack this seminorm using \cite[Lemma 11.11]{DKKST24} and the formula \eqref{E: tilded2} for $\tilde{f}_{\tilde{u}}$; after composing the resulting expression with $T_{\eta(\tilde{u})}^{\floor{\frac{\hat{a}_{\eta(\tilde{u})}}{\hat{a}_{\tilde{u}}}a_{\tilde{u}}(n)}}$ and applying the Cauchy-Schwarz inequality, we have that 
\begin{multline*}
    \liminf_{M\to\infty}\E_{\um, \um',\um'',\um'''\in [\pm M]^{t+t'+k-1}}\limsup_{N\to\infty} \sup_{|c_n|\leq 1}\Bigl\Vert
  \E_{n\in[N]}c_n\cdot T_{\eta(\tilde{u})}^{\floor{\frac{\hat{a}_{\eta(\tilde{u})}}{\hat{a}_{\tilde{u}}}a_{\tilde{u}}(n)}} h_{\um,\um',\um'',\um'''}
   \\
   \prod_{(i,j)\in I\backslash\{\tilde{u}\}}T_{\eta(i,j)}^{\floor{\frac{\hat{a}_{\eta(i,j)}}{\hat{a}_{i,j}}a_{i,j}(n)}}f_{i,j,\um,\um',\um'',\um'''}\cdot
   \prod_{j\in[J]} \CD'_{j,\um,\um',\um'',\um'''}(\floor{b_j(n)})\Bigr\Vert_{L^2(\mu)} > 0,
\end{multline*}
        where
        \begin{align*}
            f_{i,j,\um,\um',\um'',\um'''} &= \Delta_{\substack{(\floor{\bv_1 m_1'} - \floor{\bv_1 m_1}, \floor{\bv_1 m_1'''} - \floor{\bv_1 m_1''}), \ldots,\\ (\floor{\bu_{k-1} m_{t+t'+k-1}'} - \floor{\bu_{k-1} m_{t+t'+k-1}}, \floor{\bu_{k-1} m_{t+t'+k-1}'''} - \floor{\bu_{k-1} m_{t+t'+k-1}''})}}\overline{f_{i,j}},\\
            \CD'_{j,\um,\um',\um'',\um'''} &= \Delta_{\substack{(\floor{\bv_1 m_1'} - \floor{\bv_1 m_1}, \floor{\bv_1 m_1'''} - \floor{\bv_1 m_1''}), \ldots,\\ (\floor{\bu_{k-1} m_{t+t'+k-1}'} - \floor{\bu_{k-1} m_{t+t'+k-1}}, \floor{\bu_{k-1} m_{t+t'+k-1}'''} - \floor{\bu_{k-1} m_{t+t'+k-1}''})}}\overline{\CD'_{j}},
        \end{align*}
    and each $h_{\um,\um',\um'',\um'''}$ is a product of $2^{t+t'+k-1}$ 1-bounded dual functions of level $s'$ with respect to $T_{\eta(\tilde{u})}$.
       By induction on $k'$,  for a set of $(\um,\um',\um'',\um'')$ of positive lower density, we have that
       \begin{multline*}
           \nnorm{f_{1,1,\um,\um',\um'',\um'''}}_{\bu''_{1},\dots,\bu''_{t''}}\\
           =\Bignnorm{\Delta_{\substack{(\floor{\bv_1 m_1'} - \floor{\bv_1 m_1}, \floor{\bv_1 m_1'''} - \floor{\bv_1 m_1''}), \ldots,\\ (\floor{\bu_{k-1} m_{t+t'+k-1}'} - \floor{\bu_{k-1} m_{t+t'+k-1}}, \floor{\bu_{k-1} m_{t+t'+k-1}'''} - \floor{\bu_{k-1} m_{t+t'+k-1}''})}}\overline{f_{1,1}}}_{\bu''_{1},\dots,\bu''_{t''}}
       \end{multline*}
           is (uniformly) bounded away from 0
         for some $t''\in\N$ (depending only on $d,J,\ell,r',\CQ$) and $\bu''_{1},\dots,\bu''_{t''}$ belonging to the set $\CG_{\eta}$.\footnote{Once again, we use here \cite[Proposition A.2]{FrKu22a} to upgrade qualitative seminorm control to a soft quantitative estimate.}
   Averaging over all $\um,\um',\um'',\um''$, raising it to the power $2^{t''}$, using the H{\"o}lder inequality and applying the inductive formula for generalized box seminorms, we get the auxiliary norm control
       $$\nnorm{f_{1,1}}_{\bv_{1},\dots,\bv_{t},\bu_{1},\dots,\bu_{k-1},\bu'_{1},\dots,\bu'_{t'},\bu''_{1},\dots,\bu''_{t''}}>0.$$
     In conclusion, we have that  (\ref{E: average to be controled}) is zero if $\nnorm{f_{1,1}}_{\bv_{1},\dots,\bv_{t},\bu_{1},\dots,\bu_{k-1},\bu'_{1},\dots,\bu'_{t'},\bu''_{1},\dots,\bu''_{t''}}=0$.

    \smallskip
      \textbf{Step 3: Relative concatenation.} 
      \smallskip

      Combining our original seminorm control with the conclusion of the Pong step and Lemma \ref{L: basic properties of seminorms}, we deduce that     
       (\ref{E: average to be controled}) is zero if
       \begin{align*}
           \min(\nnorm{f_{1,1}}^{+}_{\bv_{1},\dots,\bv_{t},\bu_{1},\dots,\bu_k}, \nnorm{f_{1,1}}^{+}_{\bv_{1},\dots,\bv_{t},\bu_{1},\dots,\bu_{k-1},\bu'_{1},\dots,\bu'_{t'},\bu''_{1},\dots,\bu''_{t''}}) = 0.
       \end{align*}
    We want to combine the previous estimates to show that \eqref{E: average to be controled} is 0 if 
       \begin{equation}\label{pojj}
           \nnorm{f_{1,1}}^{+}_{\bv_{1},\dots,\bv_{t},\bu_{1},\dots,\bu_{k-1},\{\langle \bu_k, \bu\rangle: \bu\in\mathcal{R}\}} = 0
       \end{equation}
 where $\mathcal{R}$ is the multi-set $\{\bu'_{1},\dots,\bu'_{t'},\bu''_{1},\dots,\bu''_{t''}\}$. Using Proposition \ref{P: structure theorem}
 and Corollary \ref{C: relative concatenation}, we decompose $f = g_1 + g_2$ where $g_1, g_2\in L^\infty(\mu)$ are 1-bounded, and 
     \begin{align*}
     \nnorm{g_1}^+_{\bv_{1},\dots,\bv_{t},\bu_{1},\dots,\bu_{k-1},\bu'_{1},\dots,\bu'_{t'},\bu''_{1},\dots,\bu''_{t''}} =  \nnorm{g_2}^+_{\bv_{1},\dots,\bv_{t},\bu_{1},\dots,\bu_k} = 0.
     \end{align*}
      Hence the contribution of $g_1$ to \eqref{E: average to be controled} vanishes because of the seminorm control established in the Pong step while the contribution of $g_2$ is 0 due to the original seminorm control. It follows by the triangle inequality that \eqref{E: average to be controled} is 0 as well, and so the seminorm in \eqref{pojj} controls \eqref{E: average to be controled}.

 Recall that $\bu'_{1},\dots,\bu'_{t'}\in \CG_{\eta'}\backslash\{\hat{a}_{\eta(\tilde{u})}\be_{\eta(\tilde{u})}\}$  and $\bu''_{1},\dots,\bu''_{t''}\in\CG_\eta$. Hence  every element $\bu\in\mathcal{R}$ lies either in $\CG_{\eta}$ or in $\CG_{\eta'}\backslash(\CG_{\eta}\cup\{\hat{a}_{\eta(\tilde{u})}\be_{\eta(\tilde{u})}\})$. In the latter case, it takes the form 
 $$\hat{a}_{\eta'(1,1)}\be_{\eta'(1,1)}-\hat{a}_{\eta'(1,j)}\be_{\eta'(1,j)}=\hat{a}_{\eta(i_{0},j_{0})}\be_{\eta(i_{0},j_{0})}-\hat{a}_{\eta(1,j)}\be_{\eta(1,j)}$$ 
 for some $2\leq j\leq \ell_1$. Recalling that $\bu_k=\hat{a}_{\eta(1,1)}\be_{\eta(1,1)}-\hat{a}_{\eta(i_{0},j_{0})}\be_{\eta(i_{0},j_{0})}$, we deduce that the group $\langle \bu_k, \bu\rangle$ contains the element $\hat{a}_{\eta(1,1)}\be_{\eta(1,1)}-\hat{a}_{\eta(1,j)}\be_{\eta(1,j)}$ which belongs to $\CG_{\eta}$.
     In conclusion, it follows from (\ref{pojj}) that (\ref{E: average to be controled}) is zero if 
     $$\nnorm{f_{1,1}}^{+}_{{\bv}_{1},\dots,{\bv}_{t''},\bu_{1},\dots,\bu_{k-1}}=0$$
     for some $t''\in\N$ and $\bv_1, \ldots, \bv_{t''}\in\CG_{\eta}$. 
   By Lemma \ref{L: basic properties of seminorms}, we have that  (\ref{E: average to be controled}) is zero if 
     $$\nnorm{f_{1,1}}_{\bv_{1},\dots,{\bv}_{t''},\bv,\bu_{1},\dots,\bu_{k-1}}=0,$$
     for any $\bv\in\CG_{\eta}$.
     So Property $k-1$ holds and we are done.

\

\smallskip
\textbf{Case 2: $a_{1,1}$ does not have the maximum degree.}
\smallskip

We now drop the assumption that $a_{1,1}$ has maximum degree.
Suppose that (\ref{goal}) is positive. By the previous case, Proposition \ref{newestimate3} holds when $\ell_1+\dots+\ell_{r'}=1$. We now induct on $\ell'$. Suppose that Proposition \ref{newestimate3} holds when $\ell_1+\dots+\ell_{r'}=\ell'-1$ for some $\ell'\geq 2$. We show that it also holds when $\ell_1+\dots+\ell_{r'}=\ell'$. 
If $a_{1,1}$ has maximum degree, then the conclusion follows from the previous case. If not, then we may assume without loss of generality that $a_{2,1}$ has maximum degree. Denote $\tilde{u}:=(2,1)$.
   It follows from  \cite[Lemma 11.7]{DKKST24} that 
     \begin{multline*}
         \limsup_{N\to\infty}\sup_{\substack{
         \CD_1, \ldots, \CD_J \in \FD_d}} \sup_{|c_n|\leq 1}\lim_{N\to\infty}\Bigl\Vert\E_{n\in[N]}c_{n}\cdot T_{\tilde{u}}^{\floor{a_{\tilde{u}}(n)}}\tilde{f}_{\tilde{u}}\cdot\prod_{(i,j)\in I\backslash\{\tilde{u}\}}T_{i,j}^{\floor{a_{i,j}(n)}}f_{i,j}
         \\
         \prod_{j\in[J]}\mathcal{D}_{j}(\floor{b_{j}(n)})\Bigr\Vert_{2}>0
     \end{multline*}
     for some $\tilde{f}_{\tilde{u}}$ which is the weak limit $\lim\limits_{w\to\infty} A_{N_{w}}$ for 
\begin{multline}\label{E: tilded22}
    A_{N_{w}}:=\E_{n\in[N_{w}]} \overline{c}'_{n}\cdot T_{\tilde{u}}^{-\floor{a_{\tilde{u}}(n)}}\overline{g}_{w}\cdot\prod_{(i,j)\in I\backslash\{\tilde{u}\}}T_{i,j}^{\floor{a_{i,j}(n)}}T_{\tilde{u}}^{-\floor{a_{\tilde{u}}(n)}}\overline{f}_{i,j}
        \\
        \prod_{j\in [J]}T_{\tilde{u}}^{-\floor{a_{\tilde{u}}(n)}}\overline{\mathcal{D}}'_{j}(\floor{b_{j}(n)})
\end{multline}
     for some increasing sequence $(N_{w})_{w}$, 1-bounded weights $(c'_{n})_n$, 1-bounded function $g_{w}$, and dual sequences $\mathcal{D}'_{1},\dots,\mathcal{D}'_{J}\in\mathfrak{D}_{d}$.
     Since $a_{2,1}$ has maximum degree, we can apply the previous case to conclude that $\nnorm{\tilde{f}_{\tilde{u}}}_{(\hat{a}_{\tilde{u}}\be_{\tilde{u}})^{\times s'},\bu_{1},\dots,\bu_{t}}>0$ for some $t\in\N$ (depending only on $d,J,\ell,\CQ$) and some $\bu_{1},\dots,\bu_{t}\in\CG$.\footnote{In fact, the previous case implies that $\bu_{1},\dots,\bu_{t}$ necessarily take the form $\hat{a}_{i,j}\be_{i,j}-\hat{a}_{i,j'}\be_{i,j'}$ with ($1\leq j<j'\leq \ell_i$) for those $i\in[r']$ for which $a_{i,j}, a_{i,j'}$  have the highest degree; this is a subset of $\CG$.}
      We then unpack this seminorm using \cite[Lemma~11.11]{DKKST24} and the formula \eqref{E: tilded2} for $\tilde{f}_{\tilde{u}}$; after composing the resulting expression with $T_{\tilde{u}}^{\floor{a_{\tilde{u}}(n)}}$ and applying the Cauchy-Schwarz inequality, we have that 
      \begin{multline*}
          \liminf_{M\to\infty}\E_{\um, \um',\um'',\um'''\in [\pm M]^{t}}\limsup_{N\to\infty} \sup_{|c_n|\leq 1}\Bigl\Vert\E_{n\in[N]}c_n\cdot T_{\tilde{u}}^{\floor{a_{\tilde{u}}(n)}} h_{\um,\um',\um'',\um'''}
   \\\prod_{(i,j)\in I\backslash\{\tilde{u}\}}T_{i,j}^{\floor{a_{i,j}(n)}}f_{i,j,\um,\um',\um'',\um'''}\cdot \prod_{j\in[J]} \CD'_{j,\um,\um',\um'',\um'''}(\floor{b_{j}(n)})\Bigr\Vert_{L^2(\mu)} >0,
      \end{multline*}
        where
        \begin{align*}
            f_{i,j,\um,\um',\um'',\um'''} &= \Delta_{\substack{(\floor{\bu_1 m_1'} - \floor{\bu_1 m_1}, \floor{\bu_1 m_1'''} - \floor{\bu_1 m_1''}), \ldots,\\ (\floor{\bu_{t} m_{t}'} - \floor{\bu_{t} m_{t}}, \floor{\bu_{t} m_{t}'''} - \floor{\bu_{t} m_{t}''})}}\overline{f_{i,j}},\\
            \CD'_{j,\um,\um',\um'',\um'''} &= \Delta_{\substack{(\floor{\bu_1 m_1'} - \floor{\bu_1 m_1}, \floor{\bu_1 m_1'''} - \floor{\bu_1 m_1''}), \ldots,\\ (\floor{\bu_{t} m_{t}'} - \floor{\bu_{t} m_{t}}, \floor{\bu_{t} m_{t}'''} - \floor{\bu_{t} m_{t}''})}}\overline{\CD'_{j}},
        \end{align*}
    and each $h_{\um,\um',\um'',\um'''}$ is a product of $2^{t}$ 1-bounded dual functions of level $s'$ with respect to $T_{\tilde{u}}$.
     Since this average involves only $\ell'-1$ terms of degree greater than or equal to $\deg a_{1,1}$, we can apply (the soft quantitative version of) the induction hypothesis to conclude that for a set of $(\um,\um',\um'',\um'')$ of positive lower density, the seminorm
$$\nnorm{f_{1,1,\um,\um',\um'',\um'''}}_{\bu'_{1},\dots,\bu'_{t'}}=\Bignnorm{\Delta_{\substack{(\floor{\bu_1 m_1'} - \floor{\bu_1 m_1}, \floor{\bu_1 m_1'''} - \floor{\bu_1 m_1''}), \ldots,\\ (\floor{\bu_{t} m_{t}'} - \floor{\bu_{t} m_{t}}, \floor{\bu_{t} m_{t}'''} - \floor{\bu_{t} m_{t}''})}}\overline{f_{1,1}}}_{\bu'_{1},\dots,\bu'_{t'}}$$
           is (uniformly) bounded away from 0
         for some $t'\in\N$ (depending only on $d,J,\ell,\CQ$) and $\bu'_{1},\dots,\bu'_{t'}$ belonging to the set $\CG$.
   Averaging over all $\um,\um',\um'',\um''$, raising it to the power $2^{t'}$ using the H{\"o}lder inequality and applying the inductive formula for generalized box seminorms, we get the claimed norm control
       $$\nnorm{f_{1,1}}_{\bu_{1},\dots,\bu_{t},\bu'_{1},\dots,\bu'_{t'}}>0$$
by a seminorm with all vectors lying in $\CG$.

\section{Seminorm comparison under ergodicity assumptions I: preliminaries}\label{S: seminorm comparison I}
From now on, our goal is to use the difference ergodicity condition to infer the following result, which is crucial for ensuring the Host-Kra seminorm control of our averages.
\begin{proposition}[Seminorm comparison]\label{P: seminorm comparison}
    Let $\ell\in\N, s\in\N_0$, and let $(X, \CX, \mu, T_1, \ldots, T_\ell)$ be a system. 
    Let $a_1, a_2\in\CH$
    and suppose that there exist nonzero $\beta_1, \beta_2\in\R$ for which $\beta_2 a_1 - \beta_1 a_2 \ll \log$. Suppose moreover that $(T_1^{\floor{a_1(n)}}T_2^{-\floor{a_2(n)}})_n$ is ergodic for $(X, \CX, \mu)$.    Then for any $f\in L^\infty(\mu)$, the following holds:
    \begin{enumerate}
        \item $\nnorm{f}_{\langle\beta_1 \be_1 - \beta_2 \be_2\rangle} = 0$ whenever $\nnorm{f}_{\Z^\ell, \langle\be_1, \be_2\rangle}^+ = 0$;
        \item if $s\geq 1$ and $G_1, \ldots, G_s\subseteq\R^\ell$ are finitely generated, then 
    \end{enumerate}
    \begin{align*}
        \nnorm{f}_{G_1, \ldots, G_s, \langle\beta_1 \be_1 - \beta_2 \be_2\rangle} = 0 \quad \textrm{whenever} \quad \nnorm{f}_{G_1, \ldots, G_s, \langle\be_1, \be_2\rangle}^+ = 0.
    \end{align*}
\end{proposition}

Proposition \ref{P: seminorm comparison} needs no lower bound on the growth of $a_1, a_2$. In the degenerate case when they are both slower than log, the ergodicity assumption forces $\mu$ to be the Dirac measure $\delta_x$ at some $x\in X$ and $T_1 x = T_2 x = x$ (see Lemma \ref{L: slowly growing} for a related computation).

The proof of Proposition \ref{P: seminorm comparison} will be broken down into several cases depending on the rationality or irrationality of $\beta_1/\beta_2$ as well as the properties of $a_1, a_2$. Its first step is the following reduction which allows us to simplify the iterates appearing in the difference ergodicity condition.
\begin{proposition}\label{P: reduction of ergodicity}
    Let $(X, \CX, \mu, T_1, T_2)$ be a system. Let $a_1, a_2\in\CH$,
    and suppose that there exist nonzero $\beta_1, \beta_2\in\R$ for which $\beta_2 a_1 - \beta_1 a_2 \ll \log$. Suppose moreover that $(T_1^{\floor{a_1(n)}}T_2^{-\floor{a_2(n)}})_n$ is ergodic for $(X, \CX, \mu)$. 
    Then there exist $u\in\R$ and $a\in \CH$, which is either in $\R[x]$ or stays logarithmically away from $\R[x]$, such that $(T_1^{\floor{{a}(n)}}T_2^{-\floor{\beta{a}(n)+u}})_n$ is also ergodic for $(X, \CX, \mu)$, where $\beta = \beta_2/\beta_1$. 
\end{proposition}

In other words, Proposition \ref{P: reduction of ergodicity} allows us to cast away several (partially overlapping) cases: 
\begin{enumerate}
    \item $1\prec \beta_2 a_1 - \beta_1 a_2 \ll \log$, i.e. $a_2 = (\beta_2/\beta_1)a_1 + g$ for some $1\prec g \ll \log$;
    \item $\beta_2 a_1 - \beta_1 a_2$ is not constant but converges to a constant, i.e. $a_2 = (\beta_2/\beta_1)a_1 + u + g$, where $g$ is nonconstant but goes to 0;
    \item $a_1 = a + g_1$ and $a_2 = \beta a + u + g_2$ for some $a\in\R[x]$, $u\in\R$, and nonconstant $g_i$, each of which either goes to 0 or satisfies $1\prec g_i\ll \log$.
\end{enumerate}

Proposition \ref{P: reduction of ergodicity} will be proved in Section \ref{SS: preliminary reduction}.

The gist of the proof of Proposition \ref{P: seminorm comparison} consists in obtaining the following result, which one should (morally) think of as the $s=0$ case of Proposition \ref{P: seminorm comparison}.
\begin{proposition}\label{P: seminorm comparison, s = 0}
    Let $(X, \CX, \mu, T_1, T_2)$ be a system.
    Let $u\in\R$ and $a\in\CH$, where $a$ is either in $\R[x]$ or stays logarithmically away from $\R[x]$.
    Suppose that $(T_1^{\floor{a(n)}}T_2^{-\floor{\beta a(n)+u}})_n$ is ergodic for $(X, \CX, \mu)$. Then for any $f\in L^\infty(\mu)$, we have
    \begin{align*}
        \nnorm{f}_{\be_1 - \beta \be_2} = 0 \quad \textrm{whenever} \quad \nnorm{f}_{\langle\be_1, \be_2\rangle}^+ = 0.
    \end{align*}
\end{proposition}

The proof of Proposition \ref{P: seminorm comparison} proceeds rather differently depending on whether $\beta$ is rational or irrational. These two cases are covered separately in Corollary \ref{C: seminorm comparison, beta rational 2} and Proposition \ref{P: beta irrational}, and their proofs occupy Sections \ref{SS: beta rational} and \ref{S: seminorm comparison II} respectively.

\begin{proof}[Proof of Proposition \ref{P: seminorm comparison} assuming Proposition \ref{P: seminorm comparison, s = 0}]

Let $a_1, a_2\in\CH$ be two sequences such that $\beta_2 a_1 - \beta_1 a_2 \ll \log$ for some nonzero $\beta_1, \beta_2\in\R$, and suppose that $(T_1^{\floor{a_1(n)}}T_2^{-\floor{a_2(n)}})_n$ is ergodic for $(X, \CX, \mu)$. Let $\beta = \beta_2/\beta_1$. By Proposition \ref{P: reduction of ergodicity}, there exist a number $u\in\R$ and a sequence $a\in\CH$, which is either in $\R[x]$ or stays logarithmically away from $\R[x]$, such that $(T_1^{\floor{a(n)}}T_2^{-\floor{\beta a(n)+u}})_n$ is ergodic for $(X, \CX, \mu)$. By \cite[Proposition A.2]{FrKu22a}, we obtain the following soft quantitative strengthening of Proposition \ref{P: seminorm comparison, s = 0}: for every $\veps>0$, we can find $\delta>0$ such that for any $f\in L^\infty(\mu)$,
\begin{align}\label{E: soft quantitative seminorm comparison}
    \nnorm{f}_{\be_1 - \beta \be_2} \leq \veps \quad \textrm{whenever} \quad \nnorm{f}_{\langle \be_1, \be_2\rangle}^+\leq \delta.
\end{align}
Importantly, the parameter $\delta$ is allowed to depend on the system and $\beta$, but is uniform in the choice of $f\in L^\infty(\mu)$. We will use \eqref{E: soft quantitative seminorm comparison} in conjunction with the basic properties of seminorms in order to derive Proposition \ref{P: seminorm comparison}.

Suppose first that $s\geq 1$. Lemma \ref{L: basic properties of seminorms} gives the upper bound 
\begin{align*}
    \nnorm{f}_{G_1, \ldots, G_s, \langle\beta_1 \be_1 - \beta_2 \be_2\rangle}\ll_{\beta} \nnorm{f}_{G_1, \ldots, G_s, \langle \be_1 - \beta \be_2\rangle}.
\end{align*}
To show Proposition \ref{P: seminorm comparison}, it suffices to show that the rightmost seminorm vanishes whenever $\nnorm{f}_{G_1, \ldots, G_s, \langle\be_1, \be_2\rangle}^+ = 0.$

Suppose then that $\nnorm{f}_{G_1, \ldots, G_s, \langle\be_1, \be_2\rangle}^+=0$; setting
\begin{align*}
 G = G_1\times \cdots\times G_{s},\quad \textrm{and} \quad\bm = (\bm_1, \ldots, \bm_{s})\quad \textrm{for} \quad\bm_i\in G_i,   
\end{align*}
we then have
\begin{align*}
    \E_{(\bm, \bm')\in G\times G}\brac{\nnorm{\Delta_{(\bm_1, \bm'_1), \ldots, (\bm_{s}, \bm'_{s})}f}^+_{\langle \be_1, \be_2\rangle}}^{2^{s+1}} =0.
\end{align*}
Let $\veps>0$, and let $\delta>0$ be the parameter satisfying \eqref{E: soft quantitative seminorm comparison}. The set of $(\bm, \bm')\in G\times G$ satisfying
\begin{align*}
    \nnorm{\Delta_{(\bm_1, \bm'_1), \ldots, (\bm_{s}, \bm'_{s})}f}^+_{\langle \be_1, \be_2\rangle}> \delta
\end{align*}
has 0 density with respect to any F\o lner sequence on $G\times G$. 
For all the rest, \eqref{E: soft quantitative seminorm comparison} gives
\begin{align*}
    \nnorm{\Delta_{(\bm_1, \bm'_1), \ldots, (\bm_{s}, \bm'_{s})}f}_{\be_1 -\beta\be_2}\leq \veps.
\end{align*}
Hence
\begin{align*}
    \nnorm{f}_{G_1, \ldots, G_s, \langle\be_1 - \beta \be_2\rangle}^{2^{s+1}} = \E_{(\bm, \bm')\in G\times G}\nnorm{\Delta_{(\bm_1, \bm'_1), \ldots, (\bm_{s}, \bm'_{s})}f}_{\be_1 -\beta\be_2}^{2^{s}} \leq\veps^{2^{s}},
\end{align*}
and the claim follows on taking $\veps\to 0$.

For the case $s=0$, Lemma \ref{L: basic properties of seminorms} gives
\begin{align*}
    \nnorm{f}_{\langle\beta_1 \be_1 - \beta_2 \be_2\rangle}\ll_{\beta} \nnorm{f}^+_{\langle \be_1 - \beta \be_2\rangle}\leq \nnorm{f}_{\Z^\ell, \langle \be_1 - \beta \be_2\rangle}.
\end{align*}
The previous argument shows that the rightmost seminorm vanishes if $\nnorm{f}_{\Z^\ell, \langle\be_1, \be_2\rangle}^+ = 0$ (by substituting $G$ with $\Z^{\ell}$).
\end{proof}

\subsection{Preliminary reduction}\label{SS: preliminary reduction}

This section is devoted to the proof Proposition \ref{P: reduction of ergodicity}, a preliminary reduction in the proof of Proposition \ref{P: seminorm comparison}. We begin with an auxiliary result on the density of some Hardy orbits on certain subtori.
\begin{lemma}\label{L: density}
    Let $\ell, m\in\N$. Let $g_1, \ldots, g_m\in \CH$ with $1\prec g_1(x)\prec \cdots \prec g_m(x)\prec x$, and define
    \begin{align*}
        a_j(x) &= \sum_{i=1}^m \alpha_{j,i}g_i(x) + \alpha_{j0}+r_j(x),\\ 
        b_j(x_1, \ldots, x_m) &= \sum_{i=1}^m \alpha_{j,i}x_i + \alpha_{j0}
    \end{align*}
    for all $j\in[\ell]$, some $\alpha_{j,i}\in\R$, and some $r_j\in\CH$ with $\lim\limits_{x\to\infty}r_j(x)=0$.
    Then $$(\{a_1(n)\}, \ldots, \{a_\ell(n)\})_n$$ is dense in
    \begin{align*}
        Z = \overline{\{(\{b_1(x_1, \ldots, x_m)\}, \ldots, \{b_\ell(x_1, \ldots, x_m)\}):\; x_1, \ldots, x_m\in\R\}},
    \end{align*}
    which is an affine subtorus of $\T^\ell$.\footnote{I.e. a set $\bu + H\subseteq \T^\ell$ for some compact (not necessarily connected) subgroup $H\subseteq \T^\ell$.}
\end{lemma}
Note that because of possible rational dependencies between the coefficients $\alpha_{j,i}$, it is easier to describe the subtorus less explicitly (as we did above) than try to explicitly list its continuous generators.
\begin{proof} Let $\bu = (\alpha_{10}, \ldots, \alpha_{\ell 0})$ and $\Lambda\subseteq\Z^\ell$ be the subgroup of frequencies defining $Z$ via
\begin{align*}
    Z = \{\bz\in \T^\ell:\; \bk\cdot(\bz-\bu)\in\Z\;\; \textrm{for\; all}\;\; \bk\in \Z^\ell\}
\end{align*}
(note that $\bk\cdot\bz := k_1z_1 + \cdots + k_\ell z_\ell$ is well-defined for any $\bk\in\Z^\ell$ and $\bz\in\T^\ell$).
    Take a function $W\in\CH$ with growth $1\prec W(x) \ll x$ satisfying
    \begin{equation*}
        1\prec \log W(x)\prec a_j(x)\prec x
    \end{equation*}
    for all $j\neq \ell$ for which $a_j\succ 1$ (the simplest such example would be $W(x) = g_1(x)$). 
    Then, using~\cite[Theorem 5.1]{R22}, these growth inequalities yield 
    \begin{equation*}
     \lim\limits_{N\to\infty}   \frac{1}{W(N)}\sum_{n=1}^N \left(W(n+1)-W(n)\right)e\left(k_1a_1(n)+\cdots + k_\ell a_\ell(n)\right)=0
    \end{equation*}
    for all $\bk\in\Z^\ell\setminus{\Lambda}$. In the language of \cite{R22}, this means that $(\{a_1(n)\}, \ldots, \{a_\ell(n)\})_n$ is equidistributed on $Z$ with respect to the $W$-averaging scheme above. 

    A standard variation of Weyl's equidistribution theorem implies that for any Riemann integrable function $f: Z\to\C$, we get 
    $$\lim\limits_{N\to\infty}   \frac{1}{W(N)}\sum_{n=1}^N \left(W(n+1)-W(n)\right)F\left(\{a_1(n)\},\ldots, \{a_\ell(n)\}\right)=\int F(x)\,d\mu_Z(x),$$
    where $\mu_Z$ is the Haar measure on $Z$.
    For any (nonempty) open ball $U \subseteq Z$, the previous asymptotic implies that 
    \begin{equation*}
        \frac{1}{W(N)}\sum_{n=1}^N \left(W(n+1)-W(n)\right)1_U\left(\{a_1(n)\},\ldots, \{a_\ell(n)\}\right)=\mu_Z(U)>0,
    \end{equation*}
    and the result follows.
\end{proof}

All the information that we actually need is captured in the following corollary, which is an easy consequence of Lemma \ref{L: density} and a simple decomposition result for Hardy sequences \cite[Lemma A.3]{R22}. 
\begin{corollary}\label{C: density}
    Let $\ell\in\N$, and let $a_1, \ldots, a_\ell\in\CH$ with $a_j(x)\prec x$ for each $j\in[\ell]$. Then $$(\{a_1(n)\}, \ldots, \{a_\ell(n)\})_n$$ is dense in an affine subtorus of $\T^\ell$.
\end{corollary}

We also need one more convergence result, an easy corollary of the convergence results of Boshernitzan, Kolesnik, Quas, and Wierdl \cite{BKQW05}.
\begin{lemma}\label{L: convergence of exp sum}
    Let $\ell\in\N$. Let $a\in\CH$, and suppose that $a$ is either in $\R[x]$ or stays logarithmically away from $\R[x]$. Then:
    \begin{enumerate}
        \item For every $\beta_1, s_1, u_1, \ldots, \beta_\ell, s_\ell, u_\ell\in\R$, the exponential sum
        \begin{align}\label{E: convergence of exponential sum}
        \E_{n\in[N]}e\brac{s_1\sfloor{\beta_1 a(n)+u_1} + \cdots + s_\ell\sfloor{\beta_\ell a(n)+u_\ell}}    
        \end{align}
         converges as $N\to\infty$.
        \item For every system $(X, \CX, \mu, T_1, \ldots, T_\ell)$ and every $f\in L^2(\mu)$, the average
        \begin{align*}
            \E_{n\in[N]}T_1^{\sfloor{\beta_1 a(n)+u_1}} \cdots T_\ell^{\sfloor{\beta_\ell a(n)+u_\ell}}f
        \end{align*}
        converges in $L^2(\mu)$ as $N\to\infty$.
    \end{enumerate}
\end{lemma}
\begin{proof}
    For the first statement, we first decompose \eqref{E: convergence of exponential sum} as
    \begin{align*}
        \E_{n\in[N]}e\brac{s a(n) + u - s_1\srem{\beta_1 a(n)+u_1} - \cdots - s_\ell\srem{\beta_\ell a(n)+u_\ell}}
    \end{align*}
    for $s = s_1 \beta_1 + \cdots + s_\ell\beta_\ell$ and $u = s_1 u_1 + \cdots + s_\ell u_\ell$. Since the functions $f_i(x) = e(s_i\srem{x})$ are Riemann integrable, we can approximate them arbitrarily well by trigonometric polynomials, and so it suffices to prove the convergence of 
    \begin{align*}
        \E_{n\in[N]}e\brac{\tilde{s}a(n) + \tilde{u}},
    \end{align*}
    where $\tilde{s} = s + m_1 \beta_1 + \cdots + m_\ell \beta_\ell$ and $\tilde{u}= u + m_1 u_1 + \cdots + m_\ell u_\ell$ for all choices of $m_1, \ldots, m_\ell\in\Z$. This in turn is a straightforward consequence of \cite[Theorem A]{BKQW05}. The second part of the lemma follows from the first part and the spectral theorem.
\end{proof}

Lastly, we require the following averaging lemma, which shows that Ces\`aro convergence in Hilbert space implies convergence along dyadic intervals. 
\begin{lemma}\label{L: dyadic intervals}
Let $(a_N)_N$ be a sequence whose sequence of ratios $(a_{N+1}/a_N)_N$ converges to some number in $(1, \infty]$, and let $(v_n)_n$ be a bounded sequence in a Hilbert space. Suppose that $\E\limits_{n\in[N]}v_n$ converges strongly (resp. weakly) to some limit $L$. Then $\E\limits_{n\in(a_N, a_{N+1}]} v_n$ converges strongly (resp. weakly) to $L$ as well.
\end{lemma}
\begin{proof}
We assume that all the limits are strong; the case when they are all weak follows identically.
    Let $S_N=v_1+\dots+v_N$. {For all sufficiently large $N$, the ratio condition implies that $a_{N+1}>a_N$,} in which case the average in question equals
    \begin{align*}
        \E\limits_{n\in(a_N, a_{N+1}]} v_n =\frac{1}{a_{N+1}-a_N} \big(S_{a_{N+1}} -S_{a_N}\big) = \frac{a_{N+1}}{a_{N+1}-a_N}\brac{\E\limits_{n\in[a_{N+1}]} v_n - \frac{a_N}{a_{N+1}}\E\limits_{n\in[a_{N}]} v_n}.
    \end{align*}
    By assumption, we have $\lim\limits_{N\to\infty}\frac{a_{N+1}}{a_N}>1$. If the limit is infinite, then $\frac{a_{N+1}}{a_{N+1}-a_N}\to 1$, and so
\begin{align*}
    \lim_{N\to\infty} \E\limits_{n\in(a_N, a_{N+1}]} v_n = \lim_{N\to\infty} \E\limits_{n\in[a_{N+1}]} v_n = L.
\end{align*}
If $\lim\limits_{N\to\infty}\frac{a_{N+1}}{a_N} = 1+c$ for some $c>0$, then $\frac{a_{N+1}}{a_{N+1}-a_N}\to \frac{1+c}{c}$, in which case
\begin{align*}
    \lim_{N\to\infty}\E\limits_{n\in(a_N, a_{N+1}]} v_n = \frac{1+c}{c}\Bigbrac{1-\frac{1}{1+c}}L = L.
\end{align*}
\end{proof}

We now proceed to prove Proposition \ref{P: reduction of ergodicity}. 
\begin{proof}[Proof of Proposition \ref{P: reduction of ergodicity}]
If  $\beta_2 a_1 - \beta_1 a_2 \ll \log$ and $\beta_1,\beta_2\neq 0$, then $(\beta_2/\beta_1) a_1 - a_2 \ll \log$ as well. To facilitate the notation, we reparametrize $a_1 = \tilde a$, $\beta = \beta_2/\beta_1$, and $a_2 = \beta \tilde a + \tilde g$ for some $\tilde g\ll\log$. Then our assumption reduces to $(T_1^{\floor{\tilde a(n)}}T_2^{-\floor{\beta \tilde a(n) + \tilde g(n)}})_n$ being ergodic.

We now split $\tilde a ={a} + {g}$ as follows:
    \begin{enumerate}
        \item if $\tilde a$ stays logarithmically away from $\R[x]$, then $a = \tilde{a} $ and ${g} = 0$;
        \item if $\tilde a = p + {g}$ for some $p\in\R[x]$ and ${g} \ll \log$, then ${a} = p$ (the representation is unique up to shifts by constants). 
    \end{enumerate}
    Then we can recast
    \begin{align*}
        T_1^{\floor{\tilde a(n)}}T_2^{-\floor{\beta \tilde a(n) + \tilde g(n)}} = T_1^{\floor{{a}(n)+g_1(n)}}T_2^{-\floor{\beta {a}(n) + g_2(n)}}
    \end{align*}
    for $g_1 = {g}$ and $g_2 = \beta {g} +\tilde g$.
    The statement of the proposition then amounts to showing that the ergodicity of $$(T_1^{\floor{{a}(n)+g_1(n)}}T_2^{-\floor{\beta {a}(n) + g_2(n)}})_n$$ implies the ergodicity of $(T_1^{\floor{{a}(n)}}T_2^{-\floor{\beta {a}(n) + u}})_n$ for some $u\in\R$.

    We now take a step back and discuss the distribution of $(\{g_1(n)\}, \{g_2(n)\})_n$ in $\T^2$, which we will need shortly. By Corollary~\ref{C: density}, $(\{g_1(n)\}, \{g_2(n)\})_n$ is dense in an {affine subtorus} $Z = (u_1,u_2)+H$ of $\T^2$. By absorbing $u_1$ in the definition of ${a}$, we can replace the shift $(u_1, u_2)$ by $(0,u)$ for some $u\in\R$. Then $(\{g_1(n)\}, \{g_2(n)-u\})_n$ is dense in $H$. Now, one of the two things happens: either $\{g_1(n)\} = \{g_2(n)-u\} = 0$ for all sufficiently large $n\in\N$, or $Z$ is 1- or 2-dimensional. In both cases, Corollary~\ref{C: density} implies that for all $\delta>0$, we can find an increasing sequence $(N_k)_k$ such that $(\{g_1(N_k)\}, \{g_2(N_k)-u\})_k$ lies in $[0, \delta)^2$ for all $k\in\N$.
    
    
Let us see how to complete the proof of the proposition using this newly acquired information about the distribution of $(\{g_1(n)\}, \{g_2(n)\})_n$ in $\T^2$. We fix a small parameter $\delta>0$.
By the ergodicity assumption, we have
\begin{align*}
    \lim_{N\to\infty}\E_{n\in[N]} T_1^{\floor{{a}(n)+g_1(n)}}T_2^{-\floor{\beta {a}(n) + g_2(n)}}f = \int f\; d\mu,
\end{align*}
and so 
\begin{align*}
    \lim_{N\to\infty}\E_{n\in [N, (1+\delta)N)} T_1^{\floor{{a}(n)+g_1(n)}}T_2^{-\floor{\beta {a}(n) + g_2(n)}}f = \int f\; d\mu
\end{align*}
as well by Lemma \ref{L: dyadic intervals}. Now, while integer parts are not additive, we can split
\begin{align*}
    \floor{{a}(n)+g_1(n)} &= \floor{{a}(n)} + \floor{g_1(n)}\\
    \floor{\beta {a}(n) + g_2(n)} &= \floor{\beta {a}(n)+u} + \floor{g_2(n)-u}
\end{align*}
unless
\begin{align}\label{E: boundary condition}
    \rem{{a}(n)} > 1 - \rem{g_1(n)} \quad \textrm{or}\quad \rem{\beta {a}(n)+u} > 1 - \rem{g_2(n)-u}.
\end{align}
Our goal is to show that \eqref{E: boundary condition} happens very rarely. 

Because $g_j\ll\log$, we have $g_j(n) = g_j(N) + O(\delta)$ for $n\in[N, (1+\delta)N)$, so if the boundary condition \eqref{E: boundary condition} is satisfied, then we must have
 \begin{align*}
     \rem{{a}(n)} > 1 - \rem{g_1(N)}-O(\delta) \quad \textrm{or}\quad \rem{\beta{a}(n)+u} > 1 - \rem{g_2(N)-u}-O(\delta).
 \end{align*}
 By the density considerations above, we can find an increasing sequence $(N_k)_k$ such that
\begin{align}\label{E: a dense subsequence}
    \rem{g_1(N_k)}, \rem{g_2(N_k)-u}\leq \delta,
\end{align}
so that \eqref{E: boundary condition} holds for $n\in[N_k, (1+\delta)N_k)$ only if
 \begin{align}\label{E: boundary condition 1}
     \min(\rem{{a}(n)}, \srem{\beta {a}(n)+u}) > 1 -O(\delta). 
 \end{align}
Now, since ${a}$ is either in $\R[x]$ or stays logarithmically away from $\R[x]$, we can deduce that at most a $O(\delta)$-proportion of $n\in[N_k, (1+\delta)N_k)$ can satisfy \eqref{E: boundary condition 1} for all sufficiently small $\delta>0$.\footnote{
If $a$ is in $\R[x],$ then each of the polynomials $a, \beta a+u$ is either in $\Q[x]+\R$ or in its complement. Hence the claim follows by periodicity (for sufficiently $\delta>0$) or by well distribution respectively. If $a$ stays logarithmically away from $\R[x],$ so does $\beta a +u,$ and the claim follows by equidistribution and Lemma~\ref{L: dyadic intervals}.}
Restricting to the sequence $(N_k)_k$, we then have
\begin{multline}\label{E: slow 1}
        \E_{n\in [N_k, (1+\delta)N_k)} T_1^{\floor{{a}(n)+g_1(n)}}T_2^{-\floor{\beta {a}(n) + g_2(n)}}f \\
        = T_1^{\floor{g_1(N_k)}}T_2^{-\floor{g_2(N_k)-u}}\E_{n\in [N_k, (1+\delta)N_k)} T_1^{\floor{{a}(n)}}T_2^{-\floor{\beta {a}(n)+u}}f +O(\delta).
\end{multline}

For reasons that will become clear later on, we want the average on $[N_k, (1+\delta)N_k)$ to converge along the usual interval $[N]$. Fortunately this is the case, and the $L^2(\mu)$ limit
\begin{align}\label{E: slow 2}
    L:= \lim_{N\to\infty}\E\limits_{n\in[N]} T_1^{\floor{{a}(n)}}T_2^{-\floor{\beta {a}(n)+u}}f
\end{align}
exists by Lemma \ref{L: convergence of exp sum}.
Therefore 
\begin{align*}
    \E_{n\in [N_k, (1+\delta)N_k)} T_1^{\floor{{a}(n)}}T_2^{-\floor{\beta {a}(n)+u}}f \to L
\end{align*}
in $L^2(\mu)$ as $k\to\infty$. If we can show that $L$ is constant, then $L = \int L\; d\mu = \int f\; d\mu$, and Proposition~\ref{P: reduction of ergodicity} will follow.

Comparing both sides of \eqref{E: slow 1} with the fact that the left-hand side converges to $\int f\; d\mu = \int L\; d\mu$, we get that
\begin{align*}
    \limsup_{k\to\infty}\norm{T_1^{\floor{g_1(N_k)}}T_2^{-\floor{g_2(N_k)-u}}\E_{n\in [N_k, (1+\delta)N_k)} T_1^{\floor{{a}(n)}}T_2^{-\floor{\beta {a}(n)+u}}f - \int L\; d\mu}_{L^2(\mu)}\ll \delta. 
\end{align*}
Using the invariance of measures, we can compose the integral with $T_1^{-\floor{g_1(N_k)+}}T_2^{\floor{g_2(N_k)-u}}$ and then use \eqref{E: slow 2} to get
\begin{align*}
    \norm{L - \int L\; d\mu}_{L^2(\mu)}\ll \delta.
\end{align*}
Since $\delta>0$ was arbitrary, we get that $L = \int L\; d\mu$, implying the result.
\end{proof}



\subsection{The case of rational $\beta$}\label{SS: beta rational}


The goal now is to prove Proposition \ref{P: seminorm comparison, s = 0} under the assumption that $\beta\in\Q$. This will follow from the more general result presented below.

\begin{proposition}\label{P: beta rational}
    Let $k,l\in\mathbb{Z}\backslash\{0\}$ and $u\in\R$. Then there exists a finite index subgroup $H\subseteq\mathbb{Z}^{2}$, which contains $(k,l)$ and depends only on $k$ and $l$, such that for any system $(X,\CX,\mu, T_{1},T_{2})$ and any sequence $a\colon\mathbb{N}\to\mathbb{R}$, the ergodicity of $(T_{1}^{\floor{ka(n)}}T_{2}^{-\floor{la(n)+u}})_{n}$ implies that $\CI(H)=\CI(T_{1}^{k}T_{2}^{-l})$.   
\end{proposition}
We emphasize the level of generality of Proposition \ref{P: beta rational}: it works for \textit{any} sequence $a:\N\to\R$, not necessarily a Hardy one. 

As an immediate consequence of Proposition \ref{P: beta rational} and Lemma \ref{L: basic properties of seminorms} (v), we get the following corollary.
\begin{corollary}\label{C: seminorm comparison, beta rational}
    Under the assumptions of Proposition \ref{P: beta rational}, we have $\nnorm{f}_{k\be_1 - l\be_2} = \nnorm{f}_H$ for any $f\in L^\infty(\mu)$.
\end{corollary}
Corollary \ref{C: seminorm comparison, beta rational} combined with Proposition \ref{P: reduction of ergodicity} quickly yield Proposition \ref{P: seminorm comparison, s = 0} whenever $\beta\in\Q$, which is restated in the corollary below.
\begin{corollary}\label{C: seminorm comparison, beta rational 2}
    Proposition \ref{P: seminorm comparison, s = 0} holds for nonzero $\beta\in\Q$. 
    
    More specifically, let $(X, \CX, \mu, T_1, T_2)$ be a system, let $u\in\R$ and $a\in\CH$, where $a$ is either in $\R[x]$ or stays logarithmically away from $\R[x]$, and suppose that $(T_1^{\floor{a(n)}}T_2^{-\floor{\beta a(n)+u}})_n$ is ergodic for $(X, \CX, \mu)$. Then for any $f\in L^\infty(\mu)$, we have
    \begin{align*}
        \nnorm{f}_{\be_1 - \beta \be_2} = 0 \quad \textrm{whenever} \quad \nnorm{f}_{\langle\be_1, \be_2\rangle}^+ = 0.
    \end{align*}
\end{corollary}
\begin{proof}[Proof of Corollary \ref{C: seminorm comparison, beta rational 2} assuming Corollary \ref{C: seminorm comparison, beta rational}]
Let $\beta = l/k\in\Q$ for nonzero $k,l\in\Z$. Then Proposition \ref{P: beta rational} applied to the sequence $(a(n)/k)_n$ gives a subgroup $H\subseteq \langle \be_1, \be_2\rangle$ of index $O_{\beta}(1)$ for which $\CI(T_1^k T_2^{-l}) = \CI(H)$. Then
\begin{align*}
    \nnorm{f}_{\langle \be_1, \be_2\rangle}^+ \asymp_{\beta} \nnorm{f}_{H}^+ = \nnorm{f}_{k\be_1 - l\be_2}^+ \asymp_{\beta}\nnorm{f}_{\be_1 - \beta\be_2}^+\geq \nnorm{f}_{\be_1 - \beta\be_2}
\end{align*}
by (various parts of) Lemma \ref{L: basic properties of seminorms}, from which the claimed result follows.
\end{proof}

We finally proceed to prove Proposition \ref{P: beta rational}.
\begin{proof}[Proof of Proposition \ref{P: beta rational}]
Rescaling $k,l$ and $a$ if necessary, we may assume without loss of generality that $k$ is coprime to $l$. Let $I_1, \ldots, I_K$ be the finite partition of $[0,1)$ into all the intervals with endpoints from the set
    \begin{align}\label{E: endpoints}
        \Bigl\{\frac{i}{\vert k\vert},\frac{j}{\vert l\vert}, \frac{j-\{u\}}{\vert l\vert}\colon 0\leq i\leq \vert k\vert,\; 0\leq j\leq \vert l\vert\Bigr\}.
    \end{align}
By the Bolzano-Weierstrass theorem, there exists a subsequence $(N_j)_j$ of $\mathbb{N}$ such that
the limit 
\begin{align}\label{E: weights}
    c_i =\lim_{j\to\infty}\frac{\vert\{n\in[N_j]\colon\; \{a(n)\}\in I_i\}\vert}{N_{j}}
\end{align}
    exists for all $i\in[K]$.


From now on, assume that $f$ has zero integral and is invariant under the transformation $R:=T_1^kT_2^{-l}$. Our objective is to show that $f$ is $\CI(H)$-measurable for some fixed finite index subgroup $H\subseteq \Z^2$ {which contains $(k,l)$}.

Since $(T_{1}^{\floor{ka(n)}}T_{2}^{-\floor{la(n)+u}})_{n}$ is ergodic, we have
\begin{equation}\label{kl1}
    \lim_{j\to\infty}\E_{n\in[N_j]}T_{1}^{\floor{ka(n)}}T_{2}^{-\floor{la(n)+u}}f=0
\end{equation}
in $L^2(\mu)$.
Using the identities
\begin{gather*}
 \floor{n x} = n\floor{x} + i \quad\textrm{whenever}\quad n\in\N,\; x\in\left[\frac{i}{n}, \frac{i+1}{n}\right)\\
 \textrm{and}\quad  \floor{x + y}=\floor{x}+\floor{y}+1_{\rem{x}+\rem{y}\geq 1},
\end{gather*}
we can find a function $r:[0,1)\to\Z^2$, constant on each $I_1, \ldots, I_K$ and satisfying the bound
\begin{align}\label{E: bound on r}
\norm{r}_\infty \leq  \max(|k|-1, |l|),    
\end{align}
 such that 
\begin{align*}
    T_{1}^{\floor{ka(n)}}T_{2}^{-\floor{la(n)+u}}=T_{1}^{k\floor{a(n)} + r_1(n)}T_{2}^{-l\floor{a(n)}-\floor{u}-r_2(n)} = R^{\floor{a(n)}}T_1^{r_1(n)}T_2^{-\floor{u}-r_2(n)},
\end{align*}
where $r(n)=(r_1(n), r_2(n))$.
Composing (\ref{kl1}) with $T_2^{\floor{u}}$ and using the $R$-invariance of $f$, we obtain that
\begin{equation}\label{kl2}
    0=\lim_{j\to\infty}\E_{n\in[N_{j}]}T_1^{r_1(n)}T_2^{-r_2(n)}f=\sum_{i=1}^{K}c_{i}\cdot T_1^{r_{1,i}}T_2^{-r_{2,i}} f,
\end{equation}
where 
the pair $(r_{1,i}, r_{2,i})$ is the value of $r$ on $I_i$ and $c_i$ is defined in \eqref{E: weights}.

Since $k$ is coprime to $l$, there exist $w,v\in\mathbb{Z}$ with $\max(\vert w\vert, |v|)=O_{k,l}(1)$ such that $wl-vk=1$. For all $i\in[K]$, we can then write $(r_{1,i}, r_{2,i}) = d_{i}(w,v)+d'_{i}(k,-l)$ for a unique choice of $d_i, d'_i\in\Z$, which are of size $O_{k,l}(1)$ by \eqref{E: bound on r}. Denote $S:=T_1^w T_2^v$. Using the $R$-invariance of $f$ again, we can recast
$T_1^{r_{1,i}}T_2^{-r_{2,i}} f = S^{d_i} f$ for all $i\in[K]$. 
So it follows from (\ref{kl2}) that 
\begin{equation}\label{kl3}
    0=\sum_{i=1}^{K}c_{i}S^{d_{i}}f. 
\end{equation}
Denote 
$$L(f):=\sum_{i=1}^{K}c_{i}S^{d_{i}+d_{\ast}}f$$
for all $f$, where $d_{\ast}=\max\{-d_{1},\dots,-d_{K},0\}$; clearly, $L(f) = 0$. 
By the definition of $d_*$ and the bounds on $d_i$, we have $0\leq d_{i}+d_{\ast} \ll_{k,l}1$, and hence
$$g(z):=\sum_{i=1}^{K}c_{i}z^{d_{i}+d_{\ast}}$$ is a real polynomial of degree $O_{k,l}(1)$. Since $c_{i}\geq 0$ and at least one $c_{i}$ is nonzero, the polynomial $g$ is nonzero. 

Thus, the study of the limit $L(f)$ can be reduced to the examination of the polynomial $g$. The point of the next few claims is that we can throw away all the roots of $g$ that are not roots of unity, and that we can combine them all together to find a finite-index subgroup $H\subseteq\Z^2$ containing $(k,-l)$ such that $f$ is $H$-invariant.
    
\begin{claim}\label{Claim 1}
Suppose that $f\neq 0$ and $Sf=\lambda f$ for some $\lambda\in\mathbb{C}$.
Then $\lambda$ must be a root~of~$g$ and a root of unity. 
\end{claim}

Indeed, assume that $\lambda$ is not a root of unity. Then for all $n\in \N$, $$\int f^n\; d\mu = \int S(f^n)\; d\mu= \int (Sf)^n\; d\mu=\lambda^n\int f^n\; d\mu,$$ so $f^n$ has zero integral.  
Since $f^n$ is $R$-invariant by assumption, we get $L(f^n)=0$ by \eqref{kl3}. Because $f^n\neq 0$, we further deduce that $\lambda\neq 0$ and $g(\lambda^n)=0$. By the fundamental theorem of algebra, the polynomial $g$ has only finitely many roots, and so the only way in which $g(\lambda^n)=0$ can happen for all $n\in\N$ is if $\lambda$ is a root of unity.

\

As an immediate corollary, we get the following.

\begin{claim}\label{Claim 2}
   Suppose that $\lambda\in\mathbb{C}$ is not a root of unity. If $(S-\lambda)f=0$, then $f=0$.
\end{claim}

 Next, we move on to show that we can remove all the roots of $g$ that are not roots of unity. Since $g\neq 0$, we may factorize $g$ as 
 $$g(z)=c\prod_{j=1}^{M}(z-z_{j})$$
 for some $c\neq 0, M=O_{k,l}(1)$ and $z_{1},\dots,z_{M}\in\mathbb{C}$. Then it follows from (\ref{kl3}) that
 $$L(f)=c\prod_{j=1}^{M}(S-z_{j})f=0.$$

If $M=0$, then since $c\neq 0$, we must have $f=0$, in which case we are done. So we assume that $M>0$. Rearranging the roots, we can assume that  $z_{1},\dots,z_{M'}$ are roots of unity while $z_{M'+1},\dots,z_{M}$ are not. By repeatedly applying Claim \ref{Claim 2} to $\prod_{j=1}^{M'+t}(S-z_{j})f$ in place of $f$, we deduce that
$$\prod_{j=1}^{M'+t}(S-z_{j})f=0$$
for all $0\leq t\leq M-M'-1$. Hence
 $$\prod_{j=1}^{M'}(S-z_{j})f=0.$$
 
 Assume that $z_j$ is a primitive $q_j$-th root of unity for all $1\leq j\leq M'$.
Then  
\begin{equation}\label{kl5}
    \prod_{j=1}^{M'}(S^{q_{j}}-1)f=0
\end{equation}
since $z-z_{j}$ divides $z^{q_{j}}-1$.

The following claim allows us to bound the size of $q_j$'s, and hence the index of the subgroup $H$ later on.
\begin{claim}\label{Claim 3}
    We have $q_{j}=O_{k,l}(1)$.
\end{claim}
 
 To see this, recall that since  $z_j$ is a primitive $q_j$-th root of unity, its minimal polynomial $h$ is of degree $\phi(q_{j})$, where $\phi$ is the Euler totient function. Since $z_{j}$ is also a root of $g$, the polynomial $h$ must divide $g$, and so $\phi(q_{j})=\deg(h)\leq \deg(g)=O_{k,l}(1)$. Since $\lim\limits_{t\to\infty}\phi(t)=\infty$, we have $\phi(q_{j})=O_{k,l}(1)$ only if $q_{j}=O_{k,l}(1)$.


 

The next claim enables us to combine the contributions of all the roots of unity $z_1, \ldots, z_{M'}$.

\begin{claim}\label{Claim 4}
Let $m,n\in\mathbb{N}$. If $(S^{m}-1)(S^{n}-1)f=0$, then $(S^{2mn}-1)f=0$.    
\end{claim}

To see this, note that $(S^{m}-1)(S^{n}-1)f=0$ implies that $(S^{mn}-1)^{2}f=0$ since $z^{n}-1$ and $z^{m}-1$ both divide $z^{mn}-1$. Then 
$(S^{2mn}+1)f=2S^{mn}f.$
     So $$4\int \vert f\vert^{2}=4\int \vert S^{mn}f\vert^{2}=\int \vert (S^{2mn}+1)f\vert^{2}=2\int \vert f\vert^{2}+\langle S^{2mn}f,f\rangle+\langle f,S^{2mn}f\rangle.$$
     By the Cauchy-Schwarz inequality, $\langle S^{2mn}f,f\rangle\leq \vert f\vert^{2}$ and the equality holds only if $S^{2mn}f=f$. Hence $(S^{2mn}-1)f=0$, and the claim follows. 

     \

     Recall that $M'\leq M=O_{k,l}(1)$ and $q_{j}=O_{k,l}(1)$, the latter of which follows from Claim \ref{Claim 3}. A repeated application of Claim \ref{Claim 4} and (\ref{kl5}) allows us to find $W=O_{k,l}(1)$ and $q = 2^{M'-1} q_1 \cdots q_{M'}\leq W$ such that $S^{q}f=f$.
     Let $H = \langle W! (w,v), (k,-l)\rangle$: this is a finite index subgroup of $\mathbb{Z}^{2}$ that depends only on $k$ and $l$. 
     Since $q|W!$, the fact that $f$ is invariant under both $S^{q}$ and $R$ implies that $f$ is $H$-invariant. This means that $\CI(R)=\CI(H)$. 
     \end{proof}

\section{Seminorm comparison under ergodicity assumptions II: irrational $\beta$}\label{S: seminorm comparison II}

We continue our exploration of consequences that can be drawn from the ergodicity of $(T_1^{\floor{a_1(n)}}T_2^{-\floor{a_2(n)}})_n$ whenever $\beta_2 a_1 -\beta_1 a_2\ll \log$ for some nonzero $\beta_1,\beta_2\in\R$. 
In the preceding section, we have proved Corollary \ref{C: seminorm comparison, beta rational}, a special case of Proposition \ref{P: seminorm comparison, s = 0} for rational $\beta$. Our objective now is to complete the proof of Proposition \ref{P: seminorm comparison, s = 0} (and hence also of Proposition \ref{P: seminorm comparison}) by covering the case of $\beta\notin\Q$. This turns out to require significantly more ad hoc computations than the case of rational $\beta$. For convenience, we state the main output of this section as a separate proposition.

\begin{proposition}\label{P: beta irrational}
    Proposition \ref{P: seminorm comparison, s = 0} holds for $\beta\notin\Q$. 
    
    More specifically, let $(X, \CX, \mu, T_1, T_2)$ be a system, let $u\in\R$ and $a\in\CH$, where $a$ is either in $\R[x]$ or stays logarithmically away from $\R[x]$, and suppose that $(T_1^{\floor{a(n)}}T_2^{-\floor{\beta a(n)+u}})_n$ is ergodic for $(X, \CX, \mu)$. Then for any $f\in L^\infty(\mu)$, we have
    \begin{align*}
        \nnorm{f}_{\be_1 - \beta \be_2} = 0 \quad \textrm{whenever} \quad \nnorm{f}_{\langle\be_1, \be_2\rangle}^+ = 0.
    \end{align*}
\end{proposition}

From now on, all the way up to the end of Section \ref{S: seminorm comparison II}, we fix all the data in the statement of Proposition \ref{P: beta irrational}: the system $(X, \CX, \mu, T_1, T_2)$, the numbers $u\in\R$ and $\beta\notin\Q$, and the  {function} $a\in\CH$. {We also make the following standing assumption:

\smallskip
($\clubsuit$)\quad \emph{$a$ is either in $\R[x]$ or stays logarithmically away from $\R[x]$.}
\smallskip}

For the discussion below, we also fix $f\in L^\infty(\mu)$.
Our main tool to analyze the ergodic averages induced by the action $(T_1^{\floor{ a(n)}}T_2^{-\floor{\beta a(n)+u}})_n$ is the classical Herglotz-Bochner theorem. It shows that
\begin{align}\label{E: L^2 for A}
    \lim_{N\to\infty}\norm{\E_{n\in[N]} T_1^{\floor{ a(n)}}T_2^{-\floor{\beta a(n)+u}} f}_{L^2(\mu)} 
    = \norm{A}_{L^2(\sigma_f)},
\end{align}
where $\sigma_f$ is the spectral measure on $[0,1)^2$ associated with $f$ and
\begin{multline*}
    A(t,s):= e(us)\cdot \lim_{N\to\infty}\E_{n\in[N]} e(\floor{ a(n)}t -\floor{\beta a(n)+u}s)\\
    =  \lim_{N\to\infty}\E_{n\in[N]} e((t - \beta s) a(n)- \{ a(n)\}t +\{\beta a(n)+u\}s).
\end{multline*}
We recall that the limits in \eqref{E: L^2 for A} exist by Lemma \ref{L: convergence of exp sum}, just like the limits below.

The exponential sum that we have to compare $A(t,s)$ with 
 is
\begin{multline*}
    B(t,s):= e(s) \cdot \lim_{N\to\infty}\E_{n\in[N]} e(t n +\floor{-\beta n}s) = \lim_{N\to\infty}\E_{n\in[N]} e(t n -\floor{\beta n}s)\\
    =  \lim_{N\to\infty}\E_{n\in[N]} e((t - \beta s) n+\{\beta n\}s)
\end{multline*}
(in the second equality, we use the irrationality of $\beta$ to write $\floor{-\beta n} = -\floor{\beta n} - 1$ for $n\in\Z\backslash\{0\}$)
 because 
\begin{align}\label{E: L^2 for B}
    \nnorm{f}_{\be_1 - \beta \be_2} = \lim_{N\to\infty}\norm{\E_{n\in[N]} T_1^n T_2^{\floor{-\beta n}} f}_{L^2(\mu)} 
    = \norm{B}_{L^2(\sigma_f)}
\end{align}
by the Herglotz-Bochner theorem again.

Our main goal now is to obtain a precise description of the values
$(t,s)\in[0,1)^2$ for which any of $A(t,s), B(t,s)$ vanishes, so that we can compare \eqref{E: L^2 for A} with \eqref{E: L^2 for B} in a way that allows us to control \eqref{E: L^2 for B} by a seminorm along $\langle \be_1, \be_2\rangle$ as in Proposition \ref{P: beta irrational}.
To this end, we split 
\begin{align*}
    [0,1)^2 = E_1 \cup E_2 \cup E_3,
\end{align*}
where 
\begin{align*}
    E_1 &= \{(t,s)\in[0,1)^2:\; A(t,s) = B(t,s) = 0\},\\
    E_2 &= \{(t,s)\in[0,1)^2:\; A(t,s) \neq 0\},\\
    E_3 &= \{(t,s)\in[0,1)^2:\; A(t,s) = 0,\; B(t,s) \neq 0\}.
\end{align*}

Naturally, the contribution of $E_1$ to \eqref{E: L^2 for A} and \eqref{E: L^2 for B} is zero. To evaluate the contribution of $E_2$, we use the ergodicity assumption on our action to observe that 
\begin{align*}
    \abs{\int f\, d\mu}^2 = \sigma_f(\{(0,0)\})
    = \norm{A}_{L^2(\sigma_f)}.
\end{align*}
Since $(0,0)\in E_2$ and $|A(t,s)|>0$ for every $(t,s)\in E_2$, we must have $\sigma_f(E_2) = \sigma_f(\{0,0\})$ from the continuity of the spectral measure. In particular, it follows that 
\begin{align*}
    \int_{E_2}|A(t,s)|^2\, d\sigma_f(t,s) = \int_{E_2}|B(t,s)|^2\, d\sigma_f(t,s) = \sigma_f(\{0,0\}) = \abs{\int_X f\, d\mu}^2,
\end{align*}
and so the contribution of $E_2$ to \eqref{E: L^2 for A} and \eqref{E: L^2 for B} is equal. Our objective therefore is to evaluate the contribution of $E_3$, and for that we need to understand when $A$ vanishes while $B$ does not. 
After doing so, we arrive at the following description of $E_3$.
\begin{proposition}\label{P: E_3}
The set $E_3$ is countable, and we have
\begin{align*}
E_3 \subseteq \{\brac{\beta \gamma_i, \gamma_i}:\; i\in\N\} \mod\Z^2
\end{align*}
for some sequence $(\gamma_n)_n$ of nonzero real numbers. 
\end{proposition}
For a set $E\subseteq\R^2$, we interpret $E$ mod $\Z^2$ as $\{(\rem{t}, \rem{s}):\; (t,s)\in E\}$.

The proof of Proposition \ref{P: E_3} will proceed by explicit computations of various exponential sums and can be broken down into the following four cases:
\begin{enumerate}
    \item $a(x)$ does not equal
    \begin{align}\label{E: special form of a}
        \frac{p(x)-ku}{k\beta + l} \quad \textrm{for\; any}\;\; (k, l)\in\Z^2\setminus{\{(0,0)\}},\; p\in\Z[x]{+\R}; 
    \end{align}
    \item $a(x)$ equals \eqref{E: special form of a} with $k\neq 0, l=0$;
    \item $a(x)$ equals \eqref{E: special form of a} with $k=0, l\neq 0$;
    \item $a(x)$ equals \eqref{E: special form of a} with $k, l\neq 0$.
\end{enumerate}

The analysis of $A(t,s)$ and $B(t,s)$ differs markedly depending on whether $t-\beta s, \beta, 1$ are $\Q$-independent or $\Q$-dependent. The first case can be excluded almost immediately by the lemma below. 
\begin{lemma}\label{L: Q-independence}
    Let $(t,s)\in[0,1)^2$ so that $t-\beta s, \beta, 1$  are $\Q$-independent.
     Then $B(t,s) = 0$.  As a consequence, $$E_3\subseteq \{(t,s)\in[0,1)^2:\; t - \beta s = q + \beta r\; \textrm{for\; some}\; q,r\in\Q\}.$$   

\end{lemma}

\begin{proof}
By the assumption on $(t,s)$ and Weyl's equidistribution theorem, we have
\begin{align*}
    B(t,s) = \lim_{N\to\infty}\E_{n\in[N]} e((t - \beta s) n+\{\beta n\}s) =  \int_{[0,1)^2}e(x + zs)\; dxdz = 0.
\end{align*}
In particular, such $(t,s)$ cannot lie in $E_3$.
\end{proof}

\subsection{Exponential sum computations for $\Q$-dependent $t-\beta s, \beta, 1$}\label{SS: Q-dependent}
The computations become much more intricate when $ t-\beta s, \beta, 1$ are $\Q$-dependent, in which case $t-\beta s = q+\beta r$ for some $q,r\in\Q$.
The following lemmas give sufficient and necessary conditions for $A(t,s) = 0$ to hold for different cases of $a$ under the assumption of $\Q$-dependence of $ t-\beta s, \beta, 1$.
\begin{lemma}\label{L: a, beta a eq}
Let $q, r\in\Q$ and suppose that $a\in\CH$ does not equal $\frac{p - ku}{k\beta + l}$ for any nonconstant $p\in\Z[x]{+\R}$ and $(k, l)\in\Z^2\setminus{\{(0,0)\}}$. Then
\begin{multline*}
    \lim_{N\to\infty}\E_{n\in[N]} e(q a(n)+r \beta a(n)-\{ a(n)\}t +\{\beta a(n)+u\}s)\\
=    e(-ru)1_\Z(q) \brac{1_{q=t} + 1_{\R\setminus\Z}(q-t)\frac{1}{2\pi i(q-t)}(e(q-t)-1)}\\
1_\Z(r)\brac{1_{r=-s} + 1_{\R\setminus\Z}(r+s)\frac{1}{2\pi i(r+s)}(e(r+s)-1)}.
\end{multline*}
 In particular, this is 0 if and only if one of the followings holds:
    \begin{enumerate}
    \item $q\notin \Z$;
    \item $q-t\in\mathbb{Z}\backslash\{0\}$;
    \item $r\notin \mathbb{Z}$;
    \item $r+s\in\mathbb{Z}\backslash\{0\}$.
\end{enumerate}
\end{lemma}
\begin{proof}
By Theorem \ref{T: Boshernitzan} and Weyl's equidistribution theorem, the condition on $a$ is equivalent to the sequence
\begin{align}\label{E: 2-dim seq 1}
    (\{a(n)/c\}, \{(\beta a(n)+u)/c\})_n    
\end{align}
being equidistributed on $[0,1)^2$ for every nonzero $c\in\Z$. {Indeed, for the sake of equidistribution, we want $l\frac{a(x)}{c}+k\frac{\beta a(x)+u}{c} = \frac{k\beta + l}{c}a(x)+\frac{ku}{c}$ not to lie in $\Q[x]+\R$ for any $k,l\in\Z$ not both 0 (the cases when it equals an element of $\Q[x]+\R$ plus a nonconstant function $g$ such that $g\to 0$ or $1\prec g\ll \log$ are excluded by the assumption $\clubsuit$).}  This happens precisely if $a$ takes the prescribed form.

Let $c\in\Z$ be such that $q, r\in\frac{1}{c}\Z$.
We therefore rephrase our exponential sum as
 \begin{align*}
     \lim_{N\to\infty}\E_{n\in[N]} F(\{a(n)/c\}, \{(\beta a(n)+u)/c\})
 \end{align*}
 for $F(x,y) =  e(cq x+cr y-\{cx\}t +\{cy\}s-ru)$. The equidistribution of \eqref{E: 2-dim seq 1}
 on $[0,1)^2$ gives that our average equals
 \begin{align*}
      \int_{[0,1)^2}F(x,y)\, dx dy
     &= e(-ru)\int_{[0,1)^2}e(cqx+cr y-\{cx\}t +\{cy\}s)\, dx dy.
 \end{align*}
 This expression then equals
 \begin{equation*}
     \begin{split}
       & 
       e(-ru)\int_{[0,1)^2}e(cq x+ cry-\{cx\}t +\{cy\}s)\, dxdy
       \\&=e(-ru)\sum_{i,j=0}^{c-1}\int_{[0,\frac{1}{c})^2}e\brac{cq\brac{x+i/c}+cr\brac{y+j/c}-\{c\brac{x+i/c}\}t +\{c\brac{y+j/c}\}s}\, dxdy
       \\&=e(-ru)\sum_{i,j=0}^{c-1}e\brac{qi}e\brac{rj}\int_{[0,\frac{1}{c})^2}e(cqx + cry-\{cx\}t +\{cy\}s)\, dxdy
        \\&=e(-ru)\frac{1}{c}\sum_{i=0}^{c-1}e\brac{qi}\cdot \frac{1}{c}\sum_{j=0}^{c-1}e\brac{rj}\cdot \int_{[0,1)}e((q-t)x)\, dx\cdot\int_{[0,1)}e((r+s)y)\, dy\\
        &= e(-ru)1_\Z(q)  \brac{1_{q=t} + 1_{\R\setminus\Z}(q-t)\frac{1}{2\pi i(q-t)}(e(q-t)-1)}\\
        &\qquad\qquad\qquad\qquad\qquad\qquad 1_\Z(r)\brac{1_{r=-s} + 1_{\R\setminus\Z}(r+s)\frac{1}{2\pi i(r+s)}(e(r+s)-1)}.
     \end{split}
 \end{equation*}
 The four terms in the product are zero if and only if (iii), (iv), (i) and (ii) holds respectively.
\end{proof}

The next few lemmas cover the cases when $a = \frac{p - ku}{k\beta + l}$ for some nonconstant {$p\in\Z[x]+\R$} and various choices of $k,l\in\Z$.

\begin{lemma}[{Case $l=0$}]\label{L: beta a  = p + g}
Let $c\in\N$, $k\in\Z\backslash\{0\}$, $q,r\in\frac{1}{c}\Z$ and suppose that $a = \frac{p - ku}{k\beta}$ for some nonconstant $p\in\Z[x]{+\R}$. Then
\begin{multline*}
    \lim_{N\to\infty}\E_{n\in[N]} e(qa(n)+\beta r a(n)-\{ a(n)\}t +\{\beta a(n)+u\}s)\\
    =    e\brac{-ru}\E_{j\in[ck]}e\brac{\frac{r p(j)}{k}+ \rem{\frac{p(j)}{k}}s}1_\Z(q)\\
    \brac{1_{q=t} + 1_{\R\setminus\Z}(q-t)\frac{1}{2\pi i(q-t)}(e(q-t)-1)}.
\end{multline*}
 In particular, this is 0 if and only if one of the followings holds:
    \begin{enumerate}
    \item $\E\limits_{j\in[ck]}e\brac{\frac{r p(j)}{k}+ \rem{\frac{p(j)}{k}}s}=0$;
    \item $q\notin\Z$;
    \item $q-t\in\mathbb{Z}\backslash\{0\}$.
\end{enumerate}
\end{lemma}

We remark that the first condition is tantamount to saying that the infinite average $\lim\limits_{N\to\infty}\E\limits_{n\in[N]}e\brac{\frac{rp(n)}{k}+ \rem{\frac{p(n)}{k}}s}$ converges to 0.

\begin{proof}

{If $k<0$, then we can make it positive by letting $k\mapsto - k$, $p\mapsto - p$, so we assume that $k>0$.}
On splitting the average into arithmetic progressions of difference $ck$, it equals
\begin{multline*}
 e\brac{-ru}\cdot\E_{j\in[ck]}e\brac{\frac{r p(j)}{k}+\rem{\frac{p(j)}{k}}s} \\
 \lim_{N\to\infty}\E_{n\in[N]}e\left(\frac{qp(ckn+j)-qku}{k\beta} -\rem{\frac{p(ckn+j)-ku}{k\beta}}t\right)
\end{multline*}
since
\begin{align*}
\beta a(ckn+j) +u= \frac{p(j)}{k}\!\!\!\!\mod \Z\quad \textrm{and}\quad 
    \beta r a(ckn+j) = \frac{r p(j)}{k}-ru \!\!\!\!\mod \Z.
\end{align*}
By Theorem \ref{T: Boshernitzan}, the sequence $\brac{\rem{\frac{p(ckn+j)-ku}{ck\beta}}}_n$ is equidistributed on $[0,1)$; 
the reason why we divide by $c$ is for $\frac{qp(ckn+j)-qku}{k\beta}$ 
to be an integer multiple of $\frac{p(ckn+j)-ku}{ck\beta}$. Then our average equals
\begin{align*}
         e\brac{-ru}\cdot\E_{j\in[ck]}e\brac{\frac{r p(j)}{k}+ \rem{\frac{p(j)}{k}}s}\cdot \int_0^1 e(cq x- \{cx\}t)\, dx.
\end{align*}
The same integral computation as in the proof of Lemma \ref{L: a, beta a eq} shows that this is
\begin{align*}
    e\brac{-ru}\E_{j\in[ck]}e\brac{\frac{r p(j)}{k}+ \rem{\frac{p(j)}{k}}s} 1_\Z(q)\brac{1_{q=t} + 1_{\R\setminus\Z}(q-t)\frac{1}{2\pi i(q-t)}(e(q-t)-1)},
\end{align*}
as claimed.
\end{proof}

\begin{lemma}[{Case $k=0$}]\label{L: a = p + g}
Let $c\in\N$, $l\in\Z\backslash\{0\}$, $q,r\in\frac{1}{c}\Z$ and suppose that $a = \frac{p}{l}$ for some nonconstant $p\in\Z[x]{+\R}$. Then
\begin{multline*}
    \lim_{N\to\infty}\E_{n\in[N]} e(qa(n)+\beta r a(n)-\{ a(n)\}t +\{\beta a(n)+u\}s)\\
=e(-ru)\E_{j\in[cl]}e\brac{\frac{q p(j)}{l}-\rem{\frac{p(j)}{l}}t}1_\Z(r)\\
\brac{1_{r=-s} + 1_{\R\setminus\Z}(r+s)\frac{1}{2\pi i(r+s)}(e(r+s)-1)}.
\end{multline*}
 In particular, this is 0 if and only if one of the followings holds:
    \begin{enumerate}
    \item $\E\limits_{j\in[cl]}e\brac{\frac{q p(j)}{l}-\rem{\frac{p(j)}{l}}t}=0$;
    \item $r\notin \Z$;
    \item $r+s\in\mathbb{Z}\backslash\{0\}$.
\end{enumerate}
\end{lemma}



\begin{proof}
{Like in the previous lemma, we can assume without loss of generality that $l>0$.}
On splitting the average into arithmetic progressions of difference $cl$, it equals
\begin{align*}
\E_{j\in[cl]}e\brac{\frac{q p(j)}{l}-\rem{\frac{p(j)}{l}}t} \lim_{N\to\infty}\E_{n\in[N]} e\left(\frac{\beta r p(cln+j)}{l} \right. 
    \left.    +\rem{\frac{\beta p(cln+j)}{l} +u}s\right)
\end{align*}
since
\begin{align*}
    qa(cl n + j) = \frac{q p(j)}{l}\!\!\!\! \mod \Z \quad \textrm{and}\quad 
    a(cl n + j) = \frac{p(j)}{l}\!\!\!\! \mod\Z.
\end{align*}
By Theorem \ref{T: Boshernitzan}, the sequence $\brac{\rem{\frac{\beta p(cln+j)}{cl}+u}}_n$ is equidistributed on $[0,1)$, and hence the average equals
\begin{multline*}
    e(-ru)\E_{j\in[cl]}e\brac{\frac{q p(j)}{l}-\rem{\frac{p(j)}{l}}t}\int_{0}^{1} e\big(rcx+s\{cx\}\big)dx=\\
     e(-ru)\E_{j\in[cl]}e\brac{\frac{q p(j)}{l}-\rem{\frac{p(j)}{l}}t}1_\Z(r)\brac{1_{r=-s} + 
    1_{\R\setminus\Z}(r+s)\frac{1}{2\pi i(r+s)}(e(r+s)-1)},
\end{multline*}
where the integral is evaluated as in the proof of Lemma \ref{L: a, beta a eq}.
\end{proof}

In the special case where $a(n) = n$ and $u=0$, Lemma \ref{L: a = p + g} gives the following corollary that can be used to describe the zero set of $B$.
\begin{corollary}[Case $a(n)=n$]\label{C: a = n}
Let $q,r\in\Q$.
Then
\begin{multline*}
    \lim_{N\to\infty}\E_{n\in[N]} e((q+\beta r)n +\{\beta n\}s)\\
    =     1_\Z(q) 1_\Z(r)
\brac{1_{r=-s} + 1_{\R\setminus\Z}(r+s)\frac{1}{2\pi i(r+s)}(e(r+s)-1)}.
\end{multline*}
    In particular, it is 0 if and only if one of the followings holds:
    \begin{enumerate}
        \item $q\notin\Z$;
    \item $r\notin \Z$;
    \item $r+s\in\mathbb{Z}\backslash\{0\}$.
\end{enumerate}
\end{corollary}

In the last remaining case, the computations become so complicated that the conclusion is less explicit than in the previous results. In what follows, $\sgn(x) = 1_{x\geq 0}-1_{x\leq 0}$ is the usual sign function.

\begin{lemma}[Case $k,l\neq 0$]\label{L: (k beta + l) a  = p + g}
Let $c\in\N$, $k, l\in\Z\backslash\{0\}$ {with $l>0$}, and $q,r\in\frac{1}{c}\Z$, and suppose that $a = \frac{p-ku}{k\beta + l}$ for some nonconstant $p\in\Z[x]{+\R}$. For $j\in[ck]$ and $j'~\in~[0,|ckl|~-~1]$, define
\begin{align*}
  w &:= {\frac{k(q-t)-l(r+s)}{|kl|}}\\
  x_{j,j'} &:= |k|\rem{\frac{p(j)-\sgn(k)j'}{k}}\\
  b_{j, j'} &:=\floor{\frac{\sgn(k)(j' +1_{k<0})}{l}}t-\floor{\frac{p(j)-\sgn(k)j'}{k}}s. 
\end{align*}
Then
\begin{equation*}
    \lim_{N\to\infty}\E_{n\in[N]} e(qa(n)+\beta r a(n)-\{ a(n)\}t +\{\beta a(n)+u\}s) = 0
\end{equation*}
 if and only if one of the followings holds: 
      \begin{enumerate}
    \item $w=0$ and 
    \begin{align*}
    \E_{j\in[ck]}e\left(\frac{{r+s}}{k}p(j)\right)\E_{j'\in[0, |ckl|-1]}e\left(b_{j,j'}\right)(e(s) + \min(x_{j,j'},1)(1-e(s))) = 0;
\end{align*}
    \item $w\neq 0$ and 
\begin{multline*}
\E_{j\in[ck]}e\left(\frac{{r+s}}{k}p(j)\right)\E_{j'\in[0, |ckl|-1]}e\left({j'w}+b_{j,j'}\right)\\
    \brac{e(s)\Bigbrac{e({w})-1}+(1-e(s))\Bigbrac{e({w\min(x_{j,j'},1)})-1}} = 0.
\end{multline*}    
\end{enumerate}
\end{lemma}
{We observe that there is no loss of generality in assuming that $l>0$ for if $l<0$, then we can reduce to the case $l>0$ by changing $l\mapsto -l, k\mapsto -k, p\mapsto -p$.} 
\begin{proof}

Let $\tilde{q} = qc, \tilde{r} = rc\in\Z$, and $\tilde{p}(n) = p(n)-ku$. We can recast our average as
\begin{align*}
    \lim_{N\to\infty}\E_{n\in[N]} e\brac{{\frac{(\tilde{q} +\beta\tilde{r})\tilde{p}(n)}{c(k\beta + l)}}-\rem{\frac{\tilde{p}(n)}{k\beta + l}}t +\rem{\frac{\beta \tilde{p}(n)}{k\beta + l}+u}s}.
\end{align*}
Using the identity
\begin{align*}
    \frac{x\beta + y}{k\beta + l} = \frac{ky - lx}{k(k\beta + l)} + \frac{x}{k},
\end{align*}
we want to bring all the fractions above to the form $\Q + \frac{\Z}{ck(k\beta + l)}$ and then apply Theorem~\ref{T: Boshernitzan} and Weyl's equidistribution theorem. Indeed, our average equals
\begin{multline*}
    \lim_{N\to\infty}\E_{n\in[N]} e\left({\brac{\frac{\tilde{q}k -\tilde{r}l}{ck(k\beta + l)}+\frac{\tilde{r}}{ck}}\tilde{p}(n)}-\rem{\frac{ck}{ck(k\beta + l)}\tilde{p}(n)}t\right.\\
    \left.+\rem{\brac{\frac{-cl}{ck(k\beta + l)}+\frac{1}{k}} \tilde{p}(n)+u}s\right).
\end{multline*}
To remove the rational shifts, we split the range of $n$ into arithmetic progressions of difference $ck$. Thus our average becomes
\begin{multline*}
    \E_{j\in[ck]} e\brac{\frac{\tilde{r}\tilde p(j)}{ck}}  \lim_{N\to\infty}\E_{n\in[N]}e\left({{\frac{\tilde{q}k -\tilde{r}l}{ck(k\beta + l)}}\tilde p(ckn + j)}\right.\\
    \left. -\rem{\frac{ck}{ck(k\beta + l)}\tilde p(ckn+j)}t+\rem{{\frac{-cl}{ck(k\beta + l)}} \tilde p(ckn+j)+\frac{\tilde p(j)}{k}+u}s\right).
\end{multline*}
By Theorem \ref{T: Boshernitzan}, the sequence $\brac{\rem{\frac{\tilde p(ckn+j)}{ck(k\beta + l)}}}_n$ is equidistributed on $[0,1)$, and hence our average  equals
\begin{align*}
    \E_{j\in[ck]}e\left(\frac{\tilde{r}\tilde p(j)}{ck}\right)\int_0^1 e\brac{(\tilde{q}k - \tilde{r}l)x  - \rem{ck x}t + \rem{-clx +\frac{\tilde p(j)}{k}+u}s}\; dx.
\end{align*}
Recalling that $\tilde{r} = cr$ and $\tilde{p}(j) = p(j) - ku$, we can slightly simplify this expression to
\begin{align*}
    e(-ru)\E_{j\in[ck]}e\left(\frac{{r}p(j)}{k}\right)\int_0^1 e\brac{(\tilde{q}k - \tilde{r}l)x  - \rem{ck x}t + \rem{-clx +\frac{ p(j)}{k}}s}\; dx.
\end{align*}

To remove the fractional parts, we split the domain of integration into intervals of length $\frac{1}{|ckl|}$. 
If $x\in\left(\frac{j'}{|ckl|}, \frac{j'+1}{|ckl|}\right)$, then 
\begin{align*}
    \floor{ck x} = \floor{\frac{\sgn(k)(j' +1_{k<0})}{l}}. 
\end{align*}
On the other hand, let $x_{j,j'} := |k|\rem{\frac{p(j)-\sgn(k)j'}{k}}$.  Then an easy computation shows that
\begin{align*}
     \floor{-clx + \frac{p(j)}{k}} = \begin{cases}
         \floor{\frac{p(j)-\sgn(k)j'}{k}},\; &x\in \left(\frac{j'}{|ckl|}, \frac{j'+\min(x_{j,j'},1)}{|ckl|} \right]=:C_{j,j'}\\
         \floor{\frac{p(j)-\sgn(k)j'}{k}}-1,\; &x\in \left(\frac{j'+\min(x_{j,j'},1)}{|ckl|},\frac{j'+1}{|ckl|}\right) =: C'_{j,j'}.
     \end{cases}
\end{align*}



Now, letting
\begin{align*}
  w &:= {\frac{\tilde{q}k - \tilde{r}l-ckt-cls}{|ckl|} = \frac{k(q-t)-l(r+s))}{|kl|}}\\
  \textrm{and} \quad b_{j, j'} &:=\floor{\frac{\sgn(k)(j' +1_{k<0})}{l}}t-\floor{\frac{p(j)-\sgn(k)j'}{k}}s  
\end{align*}
 to simplify the discussion, our expression equals $e(-ru)$ times  
\begin{align*}
    & \E_{j\in[ck]}e\left(\frac{{r+s}}{k}p(j)\right)\sum_{j'\in[0, |ckl|-1]}\int_{\frac{j'}{|ckl|}}^{\frac{j'+1}{|ckl|}} e\left({|ckl|}wx  +\floor{ckx}t-\floor{-clx + \frac{p(j)}{k}}s\right)\; dx\\
    &=\E_{j\in[ck]}e\left(\frac{{r+s}}{k}p(j)\right)\sum_{j'\in[0, |ckl|-1]}\left(\int_{C_{j,j'}} e({|ckl|}wx  +b_{j,j'})\; dx+\int_{C'_{j,j'}} e\left({|ckl|}wx  +b_{j,j'}+s\right)\; dx\right)\\
    &=\E_{j\in[ck]}e\left(\frac{{r+s}}{k}p(j)\right)\sum_{j'\in[0, |ckl|-1]}\left(e(s)\int_{\frac{j'}{|ckl|}}^{\frac{j'+1}{|ckl|}} e({|ckl|}wx  +b_{j,j'})\; dx\right.\\
    &\qquad\qquad\qquad\qquad\qquad\qquad\qquad\qquad\qquad\qquad\left.+(1-e(s))\int_{C_{j,j'}} e\left({|ckl|}wx  +b_{j,j'}\right)\; dx\right).
    \end{align*}
Performing the change of variables, this equals
\begin{align*}
    &\E_{j\in[ck]}e\left(\frac{{r+s}}{k}p(j)\right)\E_{j'\in[0, |ckl|-1]}e\left({j'w}+b_{j,j'}\right)\\
    &\qquad\qquad\qquad\qquad\left(e(s)\int_{0}^1 e\left({w}x\right)\; dx+(1-e(s))\int_{0}^{\min(x_{j,j'},1)} e\left({w}x\right)\; dx\right).
\end{align*}
If $w=0$, the expression above equals
\begin{align*}
    \E_{j\in[ck]}e\left(\frac{{r+s}}{k}p(j)\right)\E_{j'\in[0, |ckl|-1]}e\left(b_{j,j'}\right)(e(s) + \min(x_{j,j'},1)(1-e(s))),
\end{align*}
otherwise we obtain
\begin{multline*}
    \E_{j\in[ck]}e\left(\frac{{r+s}}{k}p(j)\right)\E_{j'\in[0, |ckl|-1]}e\left({j'w}+b_{j,j'}\right)\\
    \frac{1}{2\pi i  {w}}\brac{e(s)\Bigbrac{e({w})-1}+(1-e(s))\Bigbrac{e({w\min(x_{j,j'},1)})-1}}.
\end{multline*}
This completes the proof.
\end{proof}

\subsection{Structure of $E_3$}

Combining the content of the previous lemmas, we move on to prove Proposition \ref{P: E_3}, i.e. we describe the set $E_3$ consisting of those points $(t,s)\in[0,1)^2$ for which $A(t,s) =0$ while $B(t,s)\neq 0$. 

     \begin{proof}[Proof of Proposition \ref{P: E_3}]
     Our objective can be restated as follows: we want to show that $E_3$ is countable, and that it is contained in $\{\brac{\beta \gamma, \gamma}: \gamma\in\R\}$ mod $\Z^2$.
    
     
     Let $(t,s)\in E_3$. By Lemma \ref{L: Q-independence} and Corollary \ref{C: a = n}, we have 
     \begin{align}\label{E: admissible (t,s)}
     t-\beta s = q + \beta r\quad \textrm{for\; some}\quad q,r\in\Z\quad \textrm{satisfying}\quad r+s\notin\Z\setminus\{0\}.
     \end{align}
     {In particular, the value $c$ appearing in various lemmas throughout Section \ref{SS: Q-dependent} is $1$.} Then
     \begin{align*}
         (t,s) = (\beta(r+s), r+s) + (q,-r),
     \end{align*}
     from which we deduce that $(t,s)$ lies in
     \begin{align}\label{E: hyperplane}
        \{(\beta \gamma, \gamma):\; \gamma\in\R\}\mod \Z^2.
     \end{align}
     It remains to show that $E_3$ is countable.
     To obtain this last claim, we split the proof into the following four cases whose exponential sums we analyzed separately in the previous section, and each of which requires a slightly different treatment now:
     \begin{enumerate}
    \item\label{i: a 1} $a(x)$ does not equal
    \begin{align}\label{E: special form of a_2}
        \frac{p(x)-ku}{k\beta + l} \quad \textrm{for\; any}\;\; (k, l)\in\Z^2\setminus{\{(0,0)\}},\; p\in\Z[x]{+\R}; 
    \end{align}
    \item\label{i: a 2} $a(x)$ equals \eqref{E: special form of a_2} with $k\neq 0, l=0$;
    \item\label{i: a 3} $a(x)$ equals \eqref{E: special form of a_2} with $k=0, l\neq 0$;
    \item\label{i: a 4} $a(x)$ equals \eqref{E: special form of a_2} with $k, l\neq 0$.
\end{enumerate}
     
     \smallskip
     \textbf{Case \eqref{i: a 1}.}
     \smallskip

    Combining the conclusions of Lemma \ref{L: a, beta a eq} and Corollary \ref{C: a = n}, the set $E_3$ consists of those pairs $(t,s)\in[0,1)^2$ satisfying \eqref{E: admissible (t,s)} for which additionally $q-t\in\mathbb{Z}\backslash\{0\}$.
    Hence $t\in\Z\setminus\{q\}$ and $s  = \frac{1}{\beta}(t-q)-r\in \frac{1}{\beta}\Z\setminus\{0\} + \Z$, and so the collection of pairs $(t,s)\in E_3$ forms a countable subset of \eqref{E: hyperplane}.
    


    \smallskip
    \textbf{Case \eqref{i: a 2}.}
    \smallskip

    In this case, Lemma \ref{L: beta a  = p + g} says that on top of \eqref{E: admissible (t,s)}, each $(t,s)\in E_3$ satisfies one of the two conditions:
       $q-t\in\Z\setminus\{0\}$ or $\E\limits_{i\in[{k}]}e\brac{\frac{r p(i)}{k}+ \rem{\frac{p(i)}{k}}s}=0$. If $q-t\in\Z\setminus\{0\}$, then  we have countably many choices of $(t,s)$ inside \eqref{E: hyperplane} just like in the previous case.
        If instead $\E_{i\in[k]}e\brac{\frac{r p(i)}{k}+ \rem{\frac{p(i)}{k}}s}=0$, then $r\in\Z$ and rudimentary manipulations imply that
        \begin{align}\label{E: reduced exp sum}
            \E_{i\in[k]} e\brac{(r+s)\rem{\frac{p(i)}{k}}} = 0.    
        \end{align}
         Consider therefore the remaining set
        \begin{align*}
            \tilde{E}_3 = \left\{(t,s)\in E_3:\; t = q+ \beta (r+s),\; q,r\in\Z,\; r+s\notin\Z\setminus\{0\},\; \eqref{E: reduced exp sum}\; \textrm{holds}\right\}.
        \end{align*}
        The crucial observation is that the exponential sum \eqref{E: reduced exp sum} can be realized as a polynomial of bounded degree. Indeed, for each $j\in\{0,\ldots, k-1\}$, let
        $$c_j =\frac{1}{k}\cdot \abs{\rem{i\in[k]:\; \rem{\frac{p(i)}{k}} = \frac{j+\{p(0)\}}{k}}}$$ and $Q(z) = \sum\limits_{j=0}^{k-1} c_j z^j$ be the rational polynomial satisfying
        \begin{align*}
            \E_{i\in[k]} e\brac{\alpha\rem{\frac{p(i)}{k}}} = Q(e(\alpha/k)){e(\alpha \{p(0)\}/k)}
        \end{align*}
        for every $\alpha\in\T$. By the fundamental theorem of algebra, there exist at most $k-1$ values $\alpha\in\T$ for which $Q(e(\alpha/k)) = 0$. In our case, $\alpha = r+ s$ and $t = q +\beta (r+s)$, and so we deduce that $\tilde{E}_3$ is countable. Combining both subcases, we deduce that $E_3$ is a countable subset of \eqref{E: hyperplane}.

\smallskip
\textbf{Case \eqref{i: a 3}.}
\smallskip
         
         In this case, Lemma \ref{L: a = p + g} guarantees that $(t,s)\in E_3$ satisfies 
         \begin{align}\label{E: reduced exp sum 3}
             \E_{i\in[{l}]}e\brac{(q-t)\rem{\frac{p(i)}{l}}} =0
         \end{align}
         on top of \eqref{E: admissible (t,s)}. The countability of $E_3$ then follows from a similar argument as in the previous case.

        \smallskip
        \textbf{Case \eqref{i: a 4}}
        \smallskip

        While in the previous two cases we deduced the countability of $E_3$ from the fundamental theorem of algebra, this time we will rely on basic properties of holomorphic functions.
        Let
        \begin{align*}
            w:= {\frac{k(q-t)-l(r+s)}{|kl|} = \frac{-(k\beta +l)(r+s)}{|kl|},}
        \end{align*}
        and recall that every $(t,s)\in E_3$ must satisfy one of the two conditions from Lemma~\ref{L: (k beta + l) a  = p + g} on top of \eqref{E: admissible (t,s)}.
        If $w=0$, then necessarily $s=-r\in\Z$, which is precluded by \eqref{E: admissible (t,s)}.
        Therefore $w\neq 0$, in which case the second condition from Lemma \ref{L: (k beta + l) a  = p + g} holds, i.e. 
\begin{multline}\label{E: tricky exp sum}
    \E_{j\in[k]}e\left(\frac{{r+s}}{k}p(j)\right)\E_{j'\in[0, |kl|-1]}e\left(-\frac{j'}{|kl|}(k\beta + l)(r+s)+b_{j,j'}\right)\\
    \brac{e(s)\Bigbrac{e({w})-1}+(1-e(s))\Bigbrac{e({w\min(x_{j,j'},1)})-1}}
\end{multline}    
is 0
for 
\begin{align*}
    x_{j,j'} &=|k|\left\{\frac{p(j)-\sgn(k)j'}{k}\right\}\\
    \textrm{and}\quad b_{j,j'} &= \floor{\frac{\sgn(k)(j' +1_{k<0})}{l}}(q+\beta(r+s)) -\floor{\frac{p(j)-\sgn(k)j'}{k}}s. 
\end{align*}

For each fixed $q,r\in\mathbb{Z}$, the average \eqref{E: tricky exp sum} is a finite linear combination of functions of the form $s\mapsto e(v s)$ for $v\in\R$. 
Each such function is holomorphic, and therefore the average in \eqref{E: tricky exp sum} is also holomorphic as a function of $s$. Hence, one of the two things may happen:
\begin{enumerate}
    \item \eqref{E: tricky exp sum} is identically 0 for all $s\in\C$;
    \item \eqref{E: tricky exp sum} has at most countably many zeros.
\end{enumerate}
If we are able to exclude the first case, then the second case immediately implies the countability of $E_3$.

\ 

We show that the former option is not possible, i.e. the average in \eqref{E: tricky exp sum} cannot vanish for all $s\in\C$, by identifying the value $s=-|kl|-r$ for which \eqref{E: tricky exp sum} is nonzero.
For this value $s$, we can simplify various expressions appearing in \eqref{E: tricky exp sum} to
\begin{align*}
        \frac{r+s}{k}p(j) &\equiv \frac{r+s}{k}p(0)\!\!\!\mod 1,\\
        -\frac{j'}{|kl|}(k\beta + l)(r+s) &\equiv \frac{\sgn(k)j'}{l}|kl|\beta\!\!\!\mod 1,\\
        b_{j,j'}&\equiv -\floor{\frac{\sgn(k)(j' +1_{k<0})}{l}}|kl|\beta\!\!\!\mod 1\\
        w &= k\beta + l\notin\Z.
\end{align*}
With these simplifications, \eqref{E: tricky exp sum} reduces to
\begin{align*}
    \E_{j'\in [0,|kl|-1]}
    e\brac{\Bigbrac{\frac{\sgn(k)j'}{l}-\floor{\frac{\sgn(k)(j' +1_{k<0})}{l}}}|kl|\beta} = 0.
\end{align*}
The left-hand side then equals  $e\Bigbrac{-\frac{\sgn(k)1_{k<0}}{l}|kl|\beta}$ times  
\begin{align*}
    \E_{j'\in [0,|kl|-1]}
    e\brac{\Big\{\frac{\sgn(k)(j'+1_{k<0})}{l}\Big\}|kl|\beta} &= \E_{j'\in [0,l-1]}
    e\brac{\Big\{\frac{j'}{l}\Big\}|kl|\beta}\\ &= \E_{j'\in [0,l-1]}e(j'|k|\beta) =\frac{e(|kl|\beta)-1}{l(e(|k|\beta)-1)},
\end{align*}
which is nonzero since $\beta\notin\Q$.
 \end{proof}

\subsection{Proof of Proposition \ref{P: beta irrational}}\label{SS: seminorm comparison base case}
We now bring all the pieces together in order to prove Proposition \ref{P: beta irrational}.

For a non-empty subset $H\subseteq[0,1)^2$, set
$$\CK(H) := \langle f\in L^2(\mu):\; T_1 f = e(t)f,\; T_2f = e(s)f\;\; \textrm{for\; some}\;\; (t,s)\in H\rangle,$$
and set $\CK(\emptyset):=\{0\}$. Let $P_H f$ be the orthogonal projection of $f$ onto $\CK(H)$. If $H=\{(t,s)\}$, then we also write
\begin{align*}
    P_{(t,s)}f =P_{\{(t,s)\}}f= \lim_{N\to\infty}\E_{n_1, n_2\in[N]}T_1^{n_1}T_2^{n_2}f\cdot e(-tn_1 - sn_2)
\end{align*}
for simplicity, and note that the single limit over $([N]^2)_N$ could be replaced by an iterated limit. If $H\cap H' = \emptyset$, then clearly $\langle P_H f, P_{H'} f\rangle = 0$ and $P_{H\cup H'} f = P_H f + P_{H'}f$.  It follows in particular that if $H\subseteq H'$, then
\begin{align}
    \nnorm{P_{H'} f}_{\be_1-\beta \be_2}^2 = \nnorm{P_{H} f}_{\be_1-\beta \be_2}^2 + \nnorm{P_{H'\setminus H} f}_{\be_1-\beta \be_2}^2\geq \nnorm{P_{H} f}_{\be_1-\beta \be_2}^2;
\end{align}
to see this, we write $P_{H'}f = P_H f +P_{H'\setminus H} f,$ expand the left-hand side as a Gowers-Cauchy-Schwarz inner product, and then use the orthogonality of $P_H f$, $P_{H'\setminus H} f$ to cancel the cross terms.

The next lemma will be crucial in completing the proof of Proposition \ref{P: beta irrational}. For eigenfunctions with eigenvalues of the form $(\beta \gamma, \gamma)$ mod $\Z^2$ with $\gamma\neq 0$, it allows us to replace the single-parameter action of $(T_1^n T_2^{-\floor{\beta n}})_n$ by the double-parameter action of $(T_1^{\floor{\frac{n}{\beta \gamma}}} T_2^{\floor{-\frac{m}{\gamma}}})_{n,m}$.

\begin{lemma}\label{L: scaling}
Let $\gamma\in\R$ be nonzero and $f\in L^2(\mu)$. Then there exists $c_{\beta, \gamma}\in\C$ such that  
\begin{align}\label{E: 2-parameter action}
\lim_{N\to\infty}\E_{n\in[N]} T_1^n T_2^{-\floor{\beta n}} P_{(\beta \gamma, \gamma)} f = c_{\beta, \gamma} \cdot \lim_{N\to\infty}\E_{n, m\in[N]} T_1^{\floor{\frac{n}{\beta \gamma}}} T_2^{\floor{-\frac{m}{\gamma}}} P_{(\beta \gamma, \gamma)} f.
\end{align}
\end{lemma}

One could alternatively try to compare the left-hand side of \eqref{E: 2-parameter action} with $$\lim\limits_{N\to\infty}\E\limits_{n, m\in[N]} T_1^{\floor{\xi n}} T_2^{\floor{-\xi \beta m}} P_{(\beta \gamma, \gamma)} f$$ for other choices of $\xi\in\R\backslash\{0\}$. However, 
this average vanishes unless $\xi\in\frac{1}{\beta\gamma}\Z$.


\begin{proof}
Let 
\begin{align*}
F &:= \lim_{N\to\infty}\E_{n\in[N]} T_1^n T_2^{-\floor{\beta n}} P_{(\beta \gamma, \gamma)} f = P_{(\beta \gamma, \gamma)} f \cdot \lim_{N\to\infty}\E_{n\in[N]} e(\beta\gamma n - \floor{\beta n}\gamma)\\
&= P_{(\beta \gamma, \gamma)} f \cdot \lim_{N\to\infty}\E_{n\in[N]} e(\rem{\beta n}\gamma) = P_{(\beta \gamma, \gamma)} f \cdot \int_0^1 e(\gamma y)\, dy
\end{align*}
and 
\begin{align*}
G &:= \lim_{N\to\infty}\E_{n, m\in[N]} T_1^{\floor{\frac{n}{\beta \gamma}}} T_2^{\floor{-\frac{m}{\gamma}}} P_{(\beta \gamma, \gamma)} f\\
&= P_{(\beta \gamma, \gamma)} f\cdot \lim_{N\to\infty}\E_{n, m\in[N]} e\brac{\floor{\frac{n}{\beta \gamma}}\beta \gamma + \floor{-\frac{m}{\gamma}}\gamma}\\
&= P_{(\beta \gamma, \gamma)} f\cdot \lim_{N\to\infty}\E_{n, m\in[N]} e\brac{-\rem{\frac{n}{\beta \gamma}}\beta \gamma - \rem{-\frac{m}{\gamma}}\gamma}.
\end{align*}

If $\gamma\in \Z\setminus\{0\}$, then $F = 0$, and so the claim follows trivially with $c_{\beta, \gamma} = 0$. 
Otherwise $\int_0^1 e(\gamma y)\, dy\neq 0$, and the goal is to show that
\begin{align*}
    C_{\beta, \gamma} := \lim_{N\to\infty}\E_{n, m\in[N]} e\brac{-\rem{\frac{n}{\beta \gamma}}\beta \gamma - \rem{-\frac{m}{\gamma}}\gamma}
\end{align*}
is nonzero as well so that
\begin{align*}
    c_{\beta, \gamma} := C_{\beta, \gamma}\inv\cdot\int_0^1 e(\gamma y)\, dy
\end{align*}
is well-defined.

We split the proof into several cases.
If $\gamma \notin(\frac{1}{\beta} \Q)\cup \Q$, then both $\beta \gamma$ and $\gamma$ are irrational, and so 
\begin{align*}
	C_{\beta, \gamma} = \int_0^1 e(-\beta \gamma x)\, dx\cdot \int_0^1 e(-\gamma y)\, dy.
\end{align*}
Both integrals are nonzero since neither $\gamma$ nor $\beta\gamma$ is an integer.

If $\gamma\in\Q\setminus\Z$, write $\gamma = -\frac{p}{q}$ for coprime integers $p,q$. Then 
\begin{align*}
C_{\beta, \gamma} = \int_0^1 e(-\beta \gamma x)\, dx \cdot \E_{m\in[p]} e\brac{\rem{\frac{qm}{p}}\frac{p}{q}}.
\end{align*}
The integral is nonzero since $\beta\gamma\notin\Z$. To see that the exponential sum is nonzero, let $q^*$ be the multiplicative inverse of $q$ mod $p$. Then the map $m\mapsto q^* m$ is a bijection on $\Z/p\Z$, and so 
\begin{align*}
	\E_{m\in[p]} e\brac{\rem{\frac{qm}{p}}\frac{p}{q}} &= \E_{m\in[p]} e\brac{\rem{\frac{qq^*m}{p}}\frac{p}{q}} = \E_{m\in[p]} e\brac{\frac{m}{q}}\neq 0 
\end{align*}
due to the coprimality of $p,q$.

Lastly, suppose that $\gamma\in\frac{1}{\beta}\Q$. Setting $\gamma = \frac{p}{q\beta}$ for coprime integers $p,q$, we deduce that
\begin{align*}
	C_{\beta, \gamma} = \int_0^1 e(-\gamma y)\, dy \cdot  \E_{n\in[p]} e\brac{-\rem{\frac{qn}{p}}\frac{p}{q}}.
\end{align*}
The integral is nonzero since $\gamma\notin\Z$, and the exponential sum is nonzero by the same argument as before.
\end{proof}
We conclude this section with the proof of Proposition \ref{P: beta irrational}.
\begin{proof}[Proof of Proposition \ref{P: beta irrational}]
If $a$ is constant, then the action can only be ergodic if the system is conjugate to a 1-point system, in which case the result follows trivially. Suppose therefore that $a$ is not constant.

    By the arguments below the statement of Proposition \ref{P: beta irrational}, we have
    \begin{align*}
        \nnorm{{f}}_{\be_1-\beta \be_2}^2 = \int_{E_3\cup\{(0,0)\}} |B(t,s)|^2\, d\sigma_{{f}}(t,s).
    \end{align*}
By Proposition \ref{P: E_3}, every element of the support $\tilde{E}_3:=E_3\cup\{(0,0)\}$ of the preceding integral takes the form $(\beta \gamma, \gamma)$ mod $\Z^2$.
Using this and orthogonality of projections on different eigenspaces, we get
\begin{align*}
    \nnorm{f}_{\be_1-\beta \be_2}^2 = \nnorm{P_{\tilde{E}_3}f}_{\be_1-\beta \be_2}^2 = \sum_{(\beta \gamma, \gamma)\in \Tilde{E}_3}\nnorm{P_{(\beta\gamma,\gamma)}f}_{\be_1-\beta \be_2}^2.
\end{align*}
By Lemma \ref{L: scaling}, for any nonzero $\gamma\in\R$,
  there exists $c_{\beta, \gamma}\geq 0$ such that 
\begin{align*}
\nnorm{P_{(\beta\gamma,\gamma)}f}_{\be_1-\beta \be_2}&=\lim_{N\to\infty}\norm{\E_{n\in[N]} T_1^n T_2^{-\floor{\beta n}} P_{(\beta \gamma, \gamma)} f}_{L^2(\mu)}\\
&= c_{\beta, \gamma} \cdot \lim_{N\to\infty}\norm{\E_{n, m\in[N]} T_1^{\floor{\frac{n}{\beta \gamma}}} T_2^{\floor{-\frac{m}{\gamma}}} P_{(\beta \gamma, \gamma)} f}_{L^2(\mu)}\\
&= c_{\beta, \gamma} \cdot \nnorm{P_{(\beta\gamma,\gamma)}f}_{\langle (1/\beta \gamma)\be_1, (-1/\gamma)\be_2\rangle}\ll_{\beta,\gamma} \nnorm{P_{(\beta\gamma,\gamma)}f}_{\langle \be_1, \be_2\rangle}^+
\end{align*}
by Lemma \ref{L: basic properties of seminorms}. 
A usual computation gives
\begin{align*}
    \nnorm{f}_{\langle \be_1, \be_2\rangle}^+ = \nnorm{P_{[0,1)^2}f}_{\langle \be_1, \be_2\rangle}^+
\end{align*}
(where $P_{[0,1)^2}f$ is the projection on the Kronecker factor of the joint action of $T_1, T_2$),
and so combining the computations so far, we get
\begin{align*}
    \nnorm{P_{(\beta\gamma,\gamma)}f}_{\be_1-\beta \be_2}\ll_{\beta,\gamma} \nnorm{P_{(\beta\gamma,\gamma)}f}_{\langle \be_1, \be_2\rangle}^+ \leq  \nnorm{P_{[0,1)^2}f}_{\langle \be_1, \be_2\rangle}^+ =  \nnorm{f}_{\langle \be_1, \be_2\rangle}^+
\end{align*}
for any $\gamma\in\R$ (including $\gamma = 0$, which is trivial). Hence if $\nnorm{f}_{\langle \be_1, \be_2\rangle}^+ = 0$, then $\nnorm{f}_{\be_1-\beta \be_2} = 0$ as well. 
\end{proof}
A glimpse at the proof shows that we cannot make the dependence between $\nnorm{f}_{\langle \be_1, \be_2\rangle}^+$ and $\nnorm{f}_{\be_1-\beta \be_2}$ quantitative since the constants $c_{\beta, \gamma}$ can be arbitrarily close to 0 for $(\beta\gamma,\gamma)\in \tilde{E}_3$.

\section{Proofs of main theorems}\label{S: proofs of main theorems}
In this section, we finally prove the theorems from the introduction. We start with a discussion on the ergodicity for slowly growing sequences.  In what follows, we say that a sequence of transformations $(T_n)_n$ is \textit{strongly/weakly ergodic} on $(X, \CX, \mu)$ if for any $f\in L^\infty(\mu)$, we have
\begin{align*}
    \lim_{N\to\infty}\E_{n\in[N]}T_n f= \int f\; d\mu
\end{align*}
in the strongly/weakly $L^{2}(\mu)$-sense (thus, the notion of strong ergodicity is the same as ergodicity). 

\begin{lemma}[Ergodicity of slowly growing sequences]\label{L: slowly growing}
    Let $a\in\CH$ with $a\ll \log$, and let $(X, \CX, \mu, T)$ be a system.  
    \begin{enumerate}
        \item Then $(T^{\floor{a(n)}})_n$ is strongly ergodic for $(X, \CX, \mu)$ if and only if $(X, \CX, \mu, T)$ is conjugate to a one-point system, or equivalently $\mu$ is the Dirac measure $\delta_x$ for~some~$x~\in~X$;
        \item If in addition $a\succ 1$, then $(T^{\floor{a(n)}})_n$ is weakly ergodic for $(X, \CX, \mu)$ if and only if $T$ is mixing.
    \end{enumerate}
\end{lemma}
\begin{proof}
Part (i). The converse direction is trivial, so we only prove the forward direction.

Suppose first that $a$ converges to a constant. Since $a$ is eventually monotone, we can find $c\in\N$ such that $\norm{T^c f - \int f\; d\mu}_{L^2(\mu)} = 0$ for all $f\in L^\infty(\mu)$. This implies that $f$ is constant $\mu$-a.e.; since $f$ was arbitrary, we deduce that the measure must be concentrated on a single point.

Suppose now that $1\prec a\ll\log$. First, we show that $a\inv(x)\gg (1+c)^x$ for some $c>0$. If this is not the case, then $\lim\limits_{x\to\infty}\frac{\log a\inv(x)}{x} = 0$. Taking the limit along $a(x)$ rather than $x$, we get $\lim\limits_{x\to\infty} \frac{\log x}{a(x)}=0$, contradicting our assumption $a\ll \log$.

Then, by Lemma \ref{L: dyadic intervals}, we get 
\begin{align*}
    \lim_{N\to\infty}\norm{\E_{n\in[a\inv(N), a\inv(N+1))}T^{\floor{a(n)}}f-\int f\; d\mu}_{L^2(\mu)} = 0
\end{align*}
for all $f\in L^\infty(\mu)$. Noticing that $\floor{a(n)} = N$ for $n\in [a\inv(N), a\inv(N+1))$, we can get rid of the average, obtaining
\begin{align*}
     \lim_{N\to\infty}\norm{T^N f-\int f\; d\mu}_{L^2(\mu)} = \norm{f-\int f\; d\mu}_{L^2(\mu)} = 0.
\end{align*}
Once again, the last equality implies that $f$ is constant $\mu$-a.e, and hence the measure must be concentrated on a single point.

\

Part (ii). We thank Nikos Frantzikinakis for showing this argument to us. Suppose first that $(T^{\floor{a(n)}})_n$ is weakly ergodic. Arguing as in the proof of Lemma \ref{L: slowly growing}, we deduce that
\begin{align*}
    \lim_{N\to\infty}\int f_1\cdot \E_{n\in[a\inv(N), a\inv(N+1))}T^{\floor{a(n)}}f_2\; d\mu = \int f_1\; d\mu\cdot \int f_2\; d\mu
\end{align*}
for all $f_1,f_2\in L^\infty(\mu)$. By noticing that $\floor{a(n)} = N$ for $n\in [a\inv(N), a\inv(N+1))$, we infer that
\begin{align*}
    \lim_{N\to\infty}\int f_1\cdot T^N f_2\; d\mu = \int f_1\; d\mu\cdot \int f_2\; d\mu,
\end{align*}
hence $T$ is mixing. The converse implication is immediate. 
\end{proof}

We start with establishing the seminorm control provided by Theorem \ref{T: seminorm control}.

\begin{proof}[Proof of Theorem \ref{T: seminorm control}]
Suppose that $a_{1},\dots,a_{\ell}\in\CH$ satisfy the difference ergodicity condition on a system $(X, \CX, \mu, T_1, \ldots, T_\ell)$, and that $T_1, \ldots, T_\ell$ are all ergodic. We will show that they are good for seminorm control. If $a_j\ll \log$ for some $j\in[\ell]$, then Lemma \ref{L: slowly growing} implies that the measure is concentrated on a single point, in which case seminorm control follows immediately. We assume therefore for the rest of the proof that $a_1, \ldots, a_\ell\succ \log$.


By \cite[Lemma A.3]{R22}, we can find an ordered Hardy family $\CQ = (q_1, \ldots, q_m)\subseteq\CH$ such that $a_{1},\dots,a_{\ell}\in\extspan_\R\CQ$.
By Theorem \ref{T: new estimates}, there exists a positive integer $s=O_{\ell, \CQ}(1)$ and vectors $\bv_{1},\dots,\bv_{s}$ coming from the set of the leading coefficients of 
        \begin{align*}
            \{a_1 \be_1\}\cup\{a_i \be_i - a_j \be_j:\; i\neq j,\; a_i, a_j\; \textrm{are\; dependent}\}
        \end{align*}
        (taken with respect to $\CQ$) such that for all $f_{1},\dots,f_{\ell}\in L^{\infty}(\mu)$, we have
\begin{align}\label{goal0}
                  \lim_{N\to\infty} 
         \norm{\E_{n\in [N]} 
         \prod_{j\in[\ell]}  T_{j}^{\sfloor{a_{j}(n)}} f_{j}}_{L^2(\mu)}=0
    \end{align}
    whenever $\nnorm{f_1}_{\bv_{1},\dots,\bv_{s}}=0$. Note that if $a_i, a_j$ are dependent, then the leading coefficient of $a_i \be_i - a_j \be_j$ equals $\beta_i \be_i - \beta_j \be_j$ for nonzero $\beta_i, \beta_j\in\R$ for which $\beta_j a_i - \beta_i a_j\ll\log$. Applying Proposition \ref{P: seminorm comparison} iteratively, we can replace each such vector in our seminorm by groups $\Z^\ell$ and $\langle \be_i, \be_j\rangle$. Combined with Lemma \ref{L: basic properties of seminorms}, the ergodicity of $T_i, T_j$ allows us to replace each such group by $\be_1$, and so we can find an integer $s'=O_{\ell, \CQ}(1)$ such that (\ref{goal0}) holds whenever $\nnorm{f_1}_{\be_1^{\times s'}}=0$. By symmetry, we can get seminorm control of (\ref{goal0}) in terms of respective bounded-degree Host-Kra seminorms of other functions. 
Therefore, $a_{1},\dots,a_{\ell}$ are good for seminorm control for the system and we are done.
\end{proof}

We are now ready to prove Theorem \ref{T: main theorem}, the main result of this paper.

\begin{proof}[Proof of Theorem \ref{T: main theorem}]
Suppose that $a_{1},\dots,a_{\ell}\in\CH$ satisfy the difference and product ergodicity conditions on a system $(X, \CX, \mu, T_1, \ldots, T_\ell)$; we will show that they are jointly ergodic. Again, we can assume that $a_1, \ldots, a_\ell\succ \log$ since otherwise the result is trivial by Lemma \ref{L: slowly growing}.

Using Theorem \ref{T: joint ergodicity criteria}, it suffices to establish that $a_1, \ldots, a_\ell$ are good for  equidistribution and seminorm control. By \cite[Corollary 12.5]{DKKST24}, the condition of being good for equidistribution follows from the product ergodicity condition. Moreover, the product ergodicity condition implies that $T_1, \ldots, T_\ell$ are all ergodic, which jointly with the difference ergodicity condition and Theorem \ref{T: seminorm control} give seminorm control.
\end{proof}

Next, we establish Theorem \ref{T: easy direction}.
\begin{proof}[Proof of Theorem \ref{T: easy direction}]
    Suppose now that $a_{1},\dots,a_{\ell}\in\CH$ are jointly ergodic for a system $(X, \CX, \mu, T_1, \ldots, T_\ell)$. 
    Once again, we can assume that  $a_1, \ldots, a_\ell\succ \log$ for otherwise the result follows trivially from Lemma \ref{L: slowly growing}. Then the product ergodicity condition follows from \cite[Lemmas 12.6 and 12.13]{DKKST24}.

    To get the difference ergodicity condition, we need to additionally assume that for every $i\neq j$ and $c_i,c_j\in\R$, the sequence $c_i a_i - c_j a_j$ is either an almost rational polynomial or it stays logarithmically away from rational polynomials. We will show that under these conditions, the exponential sum
    \begin{align}\label{E: convergence of exp sum}
        \E_{n\in[N]}e(t\floor{a_1(n)}+s\floor{a_2(n)})
    \end{align}
    converges, and hence by the spectral theorem, $\E_{n\in[N]}T_i^{\floor{a_i(n)}}T_j^{-\floor{a_j(n)}}f$ converges in $L^2(\mu)$ for any $f\in L^2(\mu)$. Therefore the $L^2(\mu)$ limit must equal the weak limit, which equals $\int f\; d\mu$ by joint ergodicity because
    \begin{align*}
        \lim_{N\to\infty}\int g\cdot \E_{n\in[N]}T_i^{\floor{a_i(n)}}T_j^{-\floor{a_j(n)}}f\; d\mu &= \lim_{N\to\infty}\int \E_{n\in[N]} T_j^{\floor{a_j(n)}}g\cdot T_i^{\floor{a_i(n)}}f\; d\mu\\
        &= \int g\; d\mu \cdot \int f\; d\mu.
    \end{align*}

    To show that \eqref{E: convergence of exp sum} converges, we expand it as
    \begin{align*}
        \E_{n\in[N]}e(t a_1(n)+s a_2(n) - t\{a_1(n)\}-s\{a_2(n)\}).
    \end{align*}
    Since the function $x\mapsto e(-t\{x\})$ is Riemann integrable for any $t\in\R$, we can approximate it arbitrarily well by trigonometric polynomials, and hence it suffices to show the convergence of 
    \begin{align*}
        \E_{n\in[N]}e(t' a_1(n)+s' a_2(n))
    \end{align*}
    for all $t'\in t+\Z,\; s'\in s+\Z$. By assumption, either $t' a_1(n)+s' a_2(n)$ equals $p(n)+c + o(1)$ for some $p\in\Q[x]$ and $c\in\R$ as $n\to\infty$, in which case the exponential sum obviously converges, or it stays logarithmically away from rational polynomials, in which case it converges by Theorem \ref{T: Boshernitzan}.    
\end{proof}

In order to prove Theorem \ref{T: counterexample}, we make use of the following seminorm estimate.

\begin{proposition}\label{est12}
    Let $(X,\mathcal{X},\mu,T_{1},T_{2})$ be a system and $f_{1},f_{2}\in L^{\infty}(\mu)$. Assume that $T_{2}\vert_{\CI(T^{-1}_{1}T_{2})}$ is mixing. We have that 
    $$\lim_{N\to\infty}\Bigl\Vert\E_{n\in[N]} T_{1}^n f_1\cdot T_{2}^{n+\floor{\log_{2}n}}f_2\Bigr\Vert_{L^2(\mu)}=0$$
    if $\nnorm{f_{j}}_{T_{j},T_{j}}=0$ for any $j=1,2$.
\end{proposition}
\begin{proof}
        Dyadically decomposing the sum, it suffices to show that 
        \begin{equation*}
        \frac{1}{N}\brac{\sum_{k=0}^{\floor{\log_2 N}-1} \sum_{n\in[2^k,2^{k+1})}T_{1}^n f_1\cdot T_{2}^{n+k}f_2 + \sum_{n\in[2^\floor{\log_2 N}, N]}T_{1}^n f_1\cdot T_{2}^{n+\floor{\log_2 N}}f_2}\to 0
    \end{equation*}in $L^2(\mu)$.
By the triangle inequality and the bound $N-2^\floor{\log_2(N)}\leq N$, the $L^2(\mu)$ norm of the previous expression is smaller than \begin{equation*}
    \frac{1}{N}\sum_{k=0}^{\floor{\log_2 N}-1}2^k \norm{\E_{n\in[2^k,2^{k+1})} T_{1}^n f_1\cdot T_{2}^{n+k}f_2}_{L^2(\mu)} + \norm{\E_{n\in[2^\floor{\log_2 N}, N]}T_{1}^n f_1\cdot T_{2}^{n+\floor{\log_2 N}}f_2}_{L^2(\mu)}.
\end{equation*}
    By applying the Cauchy-Schwarz inequality, we infer that the square of the first expression above is bounded by
    \begin{multline*}
        \frac{1}{N^2}\left(\sum_{k=0}^{\floor{\log_2 N}-1}2^k \right)\left(\sum_{k=0}^{\floor{\log_2 N}-1} 2^k\norm{\E_{n\in[2^k,2^{k+1})} T_{1}^n f_1\cdot T_{2}^{n+k}f_2}_{L^2(\mu)}^2\right)\\
        \ll \frac{1}{N}\sum_{k=0}^{\floor{\log_2 N}-1} 2^k\norm{\E_{n\in[2^k,2^{k+1})} T_{1}^n f_1\cdot T_{2}^{n+k}f_2}_{L^2(\mu)}^2.
    \end{multline*}
    Let $S:=T_{1}^{-1}T_{2}$.
    By the van der Corput inequality, this is at most 
    \begin{multline*}
       \ll\frac{1}{N}\sum_{k=0}^{\floor{\log_2 N}-1} 2^k\E_{h\in[H]}\left| \E_{n\in[2^k,2^{k+1})}\int(\overline{f_1}\cdot T_1^h f_1)\cdot S^n T_2^k(\overline{f_2}\cdot T_2^{h} f_2) \, d\mu\right|+ {\frac{1}{H}+\frac{H\log_2 N}{N}}.
    \end{multline*}

Note that for each fixed $h$ and $\e>0$, if $k$ is sufficiently large depending only on $h,\e$ and $f_{2}, T_{2}, S$, then 
$$\Bigl\Vert \E_{n\in[2^k,2^{k+1})}\Bigbrac{S^{n}(\overline{f_2}\cdot T_{2}^{h} f_2)-\E(\overline{f_2}\cdot T_{2}^{h} f_2\vert \CI(S))}\Bigr\Vert_{2}<\e$$
by the mean ergodic theorem for F\o lner sequences. So the previous expression can be bounded by 
\begin{multline*}
 \ll\E_{h\in[H]} \left(\frac{1}{N}\sum_{k=0}^{\floor{\log_2 N}-1}2^k \left|\int T_{2}^{-k}(\overline{f_1}\cdot T_{1}^h f_1)\cdot \E(\overline{f_2}\cdot T_{2}^{h} f_2\vert \CI(S)) \, d\mu\right|\right)
         \\+{\frac{1}{H}+\frac{H\log_2 N}{N}}+o_{H;N\to\infty}(1),
    \end{multline*}
    Since $T_{2}\vert_{\CI(S)}$ is mixing by assumption, for any fixed $h\in [H]$ we have
    \begin{multline*}
       \lim_{N\to\infty}\frac{1}{N}\sum_{k=0}^{\floor{\log_2 N}-1}2^k\abs{\int T_{2}^{-k}(\overline{f_1}\cdot T_{1}^h f_1)\cdot \E(\overline{f_2}\cdot T_{2}^{h} f_2\vert \CI(S)) \, d\mu}\\
       =\abs{\int \overline{f_1}\cdot T_{1}^h f_1\, d\mu \cdot\int \overline{f_2}\cdot T_{2}^h f_2\, d\mu},
    \end{multline*} and hence
    \begin{equation*}
         \abs{\frac{1}{N}\sum_{k=0}^{\floor{\log_2 N}-1} \sum_{n\in[2^k,2^{k+1})}T_{1}^n f_1\cdot T_{2}^{n+k}f_2}^2\ll\E_{h\in[H]}\left|\int \overline{f_1}\cdot T_{1}^h f_1\, d\mu \cdot\int \overline{f_2}\cdot T_{2}^h f_2\, d\mu\right|+\frac{1}{H}.
    \end{equation*}
    The square of this is bounded by
    \begin{equation*}
         \ll\E_{h\in[H]}\left|\int \overline{f_j}\cdot T_{j}^h f_j\, d\mu\right|^{2}+\frac{1}{H}
    \end{equation*}
    for any $j=1,2$.
       Taking $H\to\infty$, we conclude that the last quantity vanishes whenever $\nnorm{f_{j}}_{T_{j},T_{j}}$. An identical proof allows us to control $$\limsup\limits_{N\to\infty}\norm{\E_{n\in[2^\floor{\log_2 N}, N]}T_{1}^n f_1\cdot T_{2}^{n+\floor{\log_2 N}}f_2}_{L^2(\mu)}$$ by degree-2 Host-Kra seminorms.    
\end{proof}

\begin{proof}[Proof of Theorem \ref{T: counterexample}]
    By Lemma \ref{L: slowly growing}, $(T^{\floor{\log_2 n}})_n$ is ergodic if and only if $T$ is conjugate to a one-point system, so it remains to show that $(T^n, T^{n+\floor{\log_2 n}})_n$ is jointly ergodic if and only if $T$ is mixing.  Suppose first that $(T^n, T^{n+\floor{\log_2 n}})_n$ is jointly ergodic. Then in particular
    \begin{align*}
        \lim_{N\to\infty}\int \E_{n\in[N]}T^n f_1\cdot T^{n+\floor{\log_2 n}}f_2\; d\mu = \lim_{N\to\infty}\int \E_{n\in[N]} f_1\cdot T^{\floor{\log_2 n}}f_2\; d\mu,
    \end{align*}
    hence $(T^{\floor{\log_2 n}})_n$ is weakly ergodic. It follows from Lemma \ref{L: slowly growing} that $T$ is mixing. 

    Now we assume that $T$ is mixing. Then Proposition \ref{est12} implies that $(T^n, T^{n+\floor{\log_2 n}})_n$ is good for seminorm estimates. Since $T$ has no nontrivial eigenvalues, $(T^n, T^{n+\floor{\log_2 n}})_n$ is also good for equidistribution. So, by Theorem \ref{T: joint ergodicity criteria}, we have that $(T^n, T^{n+\floor{\log_2 n}})_n$ is jointly ergodic.
\end{proof}

\section{Joint ergodicity for pathological Hardy sequences}\label{section: repaired conjecture}

Given Corollary \ref{C:np}, it is natural to ask what can be said for sequences $a_{1},\dots,a_{\ell}\in\mathcal{H}$ for which there exist $i\neq j$, $c_i, c_j\in\R$, and $p\in\Q[x]$ satisfying $1\prec c_i a_i - c_j a_j-p\ll \log$. For convenience, we say that such a family of sequences is \emph{pathological}. Unlike the non-pathological case, the $L^{2}(\mu)$-limit 
$$\lim_{N\to\infty}\E_{n\in[N]}\prod_{j=1}^\ell T_j^{\floor{a_{j}(n)}}f_j$$
may not even exist (for example, by Lemma \ref{L: slowly growing}, $\E\limits_{n\in[N]}T^{\floor{\log_{2}(n)}}f$ converges in $L^2(\mu)$ if and only if $T^{n}f$ converges in $L^2(\mu)$, which happens precisely when $(X, \CX, \mu, T)$ is conjugate to a one-point system).

We propose here a couple of ways to amend the joint ergodicity problem to account for pathological sequences.

\subsection{The weak joint ergodicity problem}
We start with some definitions. 


\begin{definition}
      Let $a_1, \ldots, a_\ell:\N\to\R$ be sequences and $(X, \CX, \mu, T_1, \ldots, T_\ell)$ be a system. We say that $a_1, \ldots, a_\ell$ are 
      \begin{enumerate}
      \item \textit{strongly/weakly jointly ergodic} (or \emph{SJ/WJ}) for $(X, \CX, \mu, T_1, \ldots, T_\ell)$  if for all $f_1, \ldots, f_\ell\in L^\infty(\mu)$, we have that
\begin{align}\label{E: joint ergodicity2}
        \lim_{N\to\infty}\E_{n\in[N]}\prod_{j=1}^\ell T_j^{\floor{a_{j}(n)}}f_j=\prod_{j=1}^\ell \int f_j\, d\mu
    \end{align}
    in the strong/weak sense.
    \item   \emph{strongly/weakly difference ergodic} (or \emph{SD/WD}) if $(T_i^{\floor{a_i(n)}}T_j^{-\floor{a_j(n)}})_n$ is strongly/ weakly ergodic on $(X, \CX, \mu)$ for all distinct $i, j\in[\ell]$.
     \item   \emph{strongly/weakly product ergodic} (or \emph{SP/WP}) if $(T_1^{\floor{a_1(n)}}\times \cdots \times T_\ell^{\floor{a_\ell(n)}})_n$ is strongly/weakly ergodic on the product system $(X^\ell, \CX^{\otimes \ell}, \mu^\ell)$.
      \end{enumerate}
\end{definition}

We observe that all the strong notions above correspond to the notions of ergodicity, joint ergodicity, as well as difference and product ergodicity conditions defined in the introduction.

Let $Y$ denote $D,J$ or $P$.
It is clear that $SY$ implies $WY$. Conversely, if the relevant limit exists in the strong sense, then $SY$ is equivalent to $WY$.

We summarize some basic properties:

\begin{proposition}\label{all}
Let $a_{1},\dots,a_{\ell}\in\mathcal{H}$.
    \begin{enumerate}
    \item For any system, $SD+SP\Rightarrow SJ$;
    \item For any system,  $SJ\Rightarrow SP$;
    \item For any system,  $WJ\Rightarrow WD$;
    \item For some systems,  $SJ\not\Rightarrow SD$.
    \end{enumerate}
\end{proposition}
\begin{proof}
Part (i) follows from Theorem \ref{T: main theorem}. Part (ii) follows from Theorem \ref{T: easy direction}. Part (iv) follows from Theorem \ref{T: counterexample}. To see Part (iii), take $f_{i}=f, f_{j}=g$ and all other functions to be 1. Then $WJ$ implies that 
    $$\int f\; d\mu\cdot\int g\; d\mu=\lim_{N\to\infty}\E_{n\in[N]}\int T_{i}^{a_{i}(n)}f\cdot T_{j}^{a_{j}(n)}g\; d\mu=\lim_{N\to\infty}\E_{n\in[N]}\int T_{i}^{a_{i}(n)}T_{j}^{-a_{j}(n)}f\cdot g\; d\mu,$$
    which implies $WD$.
\end{proof}

Since the jointly ergodic conjecture ($SD+SP\Longleftrightarrow SJ$) fails, it is natural to ask the following relevant questions.

\begin{question}\label{c2} 
    Let $a_{1},\dots,a_{\ell}\in\mathcal{H}$. Is it true that $WD+WP\Leftrightarrow WJ$ for any system?
\end{question}

\begin{question}\label{c3}  
    Let $a_{1},\dots,a_{\ell}\in\mathcal{H}$. Is it true that $WD+SP\Leftrightarrow SJ$ for any system?
\end{question}

We remark that Question \ref{c2} can be very difficult, as the positive answer to it would imply that 2-mixing implies a property reminiscent of 3-mixing (see the following proposition).

\begin{proposition}
     Let $(X,\mathcal{X},\mu,T)$ be a system. The pair $(T^{\floor{\log_{2}n}},T^{2\floor{\log_{2}n}})_{n}$ is 
    \begin{enumerate}
        \item $WJ$ if and only if $T$ has the following property: for all $f_0, f_1, f_2\in L^\infty(\mu)$, we have $$\lim_{n\to\infty}\int f_0\cdot T^n f_1\cdot T^{2n}f_2\; d\mu = \int f_0\; d\mu \cdot \int f_1\; d\mu \cdot \int f_2\; d\mu;$$
        \item $WD$ {and $WP$} if and only if $T$ is (2-)mixing;
        \item $SJ, SD$ or $SP$ if and only if $T$ is conjugate to a one-point system.
    \end{enumerate}
\end{proposition}
\begin{proof}
 The proof of Part (i) is similar to that of Part (ii) of Lemma \ref{L: slowly growing}, and so we leave it to the interested readers.
    Parts (ii) and (iii) follow from Lemma \ref{L: slowly growing}.
\end{proof}

To make Question \ref{c3} more plausible, we address it for a tuple for which the (strong) joint ergodicity problem fails by Theorem \ref{T: counterexample}. 

\begin{proposition}\label{last}
    For any system $(X,\mathcal{X},\mu,T_{1},T_{2})$ and for the pair $(T_{1}^{n},T_{2}^{n+\floor{\log_{2}n}})_{n}$, we have that $WD+SP\Leftrightarrow SJ$.
\end{proposition}


\begin{proof}
    It follows from Proposition \ref{all} that $SJ\Rightarrow WD+SP$. So now we show the opposite direction. Note that $SP$ implies that $(T_{1}^{n},T_{2}^{n+\floor{\log_{2}n}})_{n}$ is good for equidistribution. By Theorem \ref{T: joint ergodicity criteria}, it suffices to show that $(T_{1}^{n},T_{2}^{n+\floor{\log_{2}n}})_{n}$ is good for seminorm control.
    
     Since $(T_{1}^{n},T_{2}^{n+\floor{\log_{2}n}})_{n}$ is WD, for any $\CI(T_1T_2\inv)$-measurable $f_0, f_1$ we have $$\int f_0 \cdot \E_{n\in[N]} T_2^{\floor{\log n}} f_1\to\int f_0\; d\mu \int f_1\; d\mu.$$ 
This means that $(T_{2}^{\floor{\log_{2}n}})_{n}$ is weakly ergodic for the system $(X,\mathcal{I}(T_1T_2\inv),\mu,T_{2})$. By Lemma \ref{L: slowly growing}, $T_{2}\vert_{\mathcal{I}(T_1T_2\inv)}$ is mixing, and by Proposition \ref{est12}, $(T_{1}^{n},T_{2}^{n+\floor{\log_{2}n}})_{n}$ is good for seminorm control.
\end{proof}

 \subsection{Changing the averaging scheme}
Another approach to revising the joint ergodicity classification problem for Hardy sequences (involving slowly growing functions) is to consider a different averaging scheme. This point of view was adopted in \cite{BMR20}, where some multiple recurrence results were obtained that cannot be established through the study of the corresponding (Ces\`{a}ro) ergodic averages. For instance, it can be easily proven that for any ergodic system $(X,\CX,\mu, T)$, the logarithmic averages $$\frac{1}{\log N}\sum_{n\leq N} \frac{T^{\floor{\log n}}f}{n}$$ 
converge to $\int f\; d\mu$ for all bounded measurable functions $f$. The main input in this result is the spectral theorem and the fact that $a\log n$ is equidistributed with respect to logarithmic averaging for any $a\in (0,1)$. In this sense, logarithmic averaging is more natural for sequences involving $\log$, and so we can ``repair'' Theorem \ref{T: counterexample} by considering logarithmic averages. We call a sequence $a:\N\to \Z$ \emph{ergodic with respect to logarithmic averages} on a system $(X,\CX, \mu,T)$  if \begin{equation*}
    \lim_{N\to\infty} \norm{\frac{1}{\log N}\sum_{n\in [N]} \frac{T^{a(n)}f}{n}-\int f\; d\mu }_{L^2(\mu)}=0.
\end{equation*}We define joint ergodicity of sequences with respect to logarithmic averages similarly to Definition \ref{D: joint ergodicity} by replacing Ces\`{a}ro averages with logarithmic ones.

\begin{conjecture}\label{Conj: log averaged counterexample}
    For any system $(X,\CX, \mu,T_1,T_2)$, the tuple $(T_1^n,T_2^{n+\floor{\log n}})_{n}$ is jointly ergodic for logarithmic averages if and only if $(T_2^{n+\floor{\log n}}T_1^{-n})_{n}$ and $(T_1^{n}\times T_2^{n+\floor{\log n}})_{n}$ are ergodic for logarithmic averages on $(X, \CX, \mu)$ and $(X\times X, \CX\otimes \CX, \mu\times\mu)$ respectively.
\end{conjecture}
The conjecture is true whenever $T_1=T_2$. Indeed, the difference condition is equivalent to $T$ being ergodic, and, in this case, we can easily see that the product condition holds as well. 
Additionally, the pair we study is good for seminorm control for a single transformation due to the results in \cite{BMR20}. Furthermore, it can be shown that they are good for equidistribution (for logarithmic averages). Using a variant of Theorem \ref{T: joint ergodicity criteria} for logarithmic averages \cite[Theorem 2.19]{FrKu22a}, we conclude that $(T^n, T^{n+\floor{\log n}})_{n}$ is jointly ergodic for all ergodic systems. This argument makes Conjecture \ref{Conj: log averaged counterexample} plausible.

We remark, though, that logarithmic averaging is not enough to establish joint ergodicity characterization when we deal with even slower functions, like $\log\log x$. In this example, one would need to consider an even weaker averaging scheme, namely, double logarithmic averaging. For this reason, we have to modify this line of reasoning.

We formalize the prior discussion.
Let $W(n)$ be a non-decreasing sequence of real numbers satisfying $W(n)\to\infty$ and denote $w(n)=W(n+1)-W(n)$. We define the $W$-averaging operators of a sequence $a_n$ on a normed space by the formula\begin{equation*}
   A^W_{N}\left(\left(a(n)\right)_{n}\right) :=\frac{1}{W(N)}\sum_{n=1}^N w(n)a_n.
\end{equation*}This is well-defined for $N$ sufficiently large.

Let $a_1,\ldots,a_\ell:\N\to \R$ be sequences and let $(X,\CX,\mu,T_1,\ldots, T_\ell)$ be a system.
We call $(T_1^{a_1(n)},\ldots,T_\ell^{a_\ell(n)})$  {\em $W$-jointly ergodic} for $(X,\CX,\mu,T_1,\ldots, T_\ell)$ if for all $f_1,\ldots, f_\ell\in L^{\infty}(\mu)$, we have 
\begin{equation*}
    \lim_{N\to\infty} \norm{A^W_{N}(T_1^{a(n)}f_1\ldots T_k^{a_k(n)}f_k)-\int f_1\; d\mu\cdots\int f_\ell\; d\mu}_{L^2(\mu)}=0.
\end{equation*}
Moreover, we say that a sequence of transformations $(T_n)_n$ is {\em $W$-ergodic} on $(X,\CX,\mu)$ if for all $f\in L^{\infty}(\mu)$, we have
\begin{equation*}
    \lim_{N\to\infty} \norm{A^W_{N}(T_n f)-\int f\; d\mu }_{L^2(\mu)}=0.
\end{equation*}
This brings us to the following possible modification of Problem \ref{Pr: joint ergodicity problem}.

\begin{conjecture}\label{C: weighted conjecture}
    Let $a_1,\ldots, a_\ell\in \mathcal{H}$ with $a_j\succ 1$ for every $j\in[\ell]$. Then there exists a non-decreasing sequence $(W(n))_n$ with $1\prec W(x)\ll x$
    such that for any system $(X,\CX, \mu,T_1,\ldots, T_\ell)$, the following are equivalent:\\
\begin{enumerate}
    \item $(T_1^{\floor{a_1(n)}},\ldots, T_\ell^{\floor{a_\ell(n)}})_n$ is $W$-jointly ergodic on $(X,\CX,\mu)$;\\
    \item the sequence $(T_1^{\floor{a_1(n)}}\times\dots\times T_\ell^{\floor{a_\ell(n)}})_n$ is $W$-ergodic on $(X^\ell, \CX^{\otimes \ell},\mu^{\ell})$ while the sequence $(T_i^{\floor{a_i(n)}}T_j^{-\floor{a_j(n)}})_n$ is $W$-ergodic on $(X,\CX, \mu)$ for all distinct $i, j\in[\ell]$.
    \end{enumerate}
\end{conjecture}
    A major missing ingredient needed to resolve Conjecture \ref{C: weighted conjecture} would be seminorm estimates for ergodic averages along Hardy sequences with suitable weights. To obtain such estimates, one should try to merge the argument of the authors for usual (unweighted) ergodic averages along Hardy sequences together with the method of Bergelson, Moreira and Richter \cite{BMR20}.

{One potential shortcoming of Conjecture \ref{C: weighted conjecture} vs. the original problem has to do with the choice of a suitable averaging scheme. While such schemes can always be found for Hardy sequences of polynomial growth, it is not clear what would be the ``correct'' averaging scheme if we go beyond Hardy sequences and consider e.g. sequences with oscillations, such as $x\sin{x}$.}


\medskip

\subsection{More general sequences of polynomial growth} While Corollary \ref{C:np} asserts that Problem \ref{Pr: joint ergodicity problem} admits an affirmative answer for all non-pathological Hardy sequences of polynomial growth, we believe that this result should extend to a more general class of iterates. In this section, we conjecture a natural generalization of Corollary \ref{C:np}, which will involve tempered functions defined below.

\begin{definition}[Tempered functions]
    Let $i\in\mathbb{N}_{0}$ and $x_0\geq 0$. A real-valued function $a\in C^{i+1}(x_0, \infty)$ is called a \emph{tempered function of degree $i$} if the following hold:
\begin{enumerate}
\item $a^{(i+1)}(x)$ tends monotonically to $0$ as $x\to\infty;$
\item  $\lim\limits_{x\to\infty}x|a^{(i+1)}(x)|=\infty.$
\end{enumerate} 
\end{definition}

A significant hurdle arising in the study of tempered functions is that in contrast to Hardy field functions, ratios of tempered functions may not converge. In order to avoid various problematic cases, one typically restricts attention to a well-behaved subclass of tempered functions. On letting 
\begin{align*}
    \mathcal{R} &:=\Big\{a\in C^\infty(\mathbb{R}_+):\;\lim_{x\to\infty}\frac{xa^{(i+1)}(x)}{a^{(i)}(x)}\in \mathbb{R}\;\;\text{for all}\;\;i\in\mathbb{N}_{0}\Big\},\\
    \mathcal{T}_i&:=\Big\{a\in\mathcal{R}:\;\exists\;i<\alpha\leq i+1,\;\lim_{x\to\infty}\frac{xa'(x)}{t(x)}=\alpha,\;\lim_{x\to\infty}a^{(i+1)}(x)=0\Big\},
\end{align*}
the aforementioned well-behaved class is given by $\mathcal{T}:=\bigcup_{i=0}^\infty \mathcal{T}_i$. (See \cite{BeHK09, Kouts21} for a detailed discussion of tempered functions.)

It is known that every function $a\in \mathcal{T}_i$ is a tempered function of degree $i$ and satisfies the growth condition $x^i\log x\prec a(x)\prec x^{i+1}$ (see \cite{BeHK09}).

In \cite{DKS23}, the first, second and fourth authors showed that Problem~\ref{Pr: joint ergodicity problem} holds whenever $$a_1=\ldots=a_\ell=a=c_1h+c_2t$$ for some $(c_1,c_2)\in \mathbb{R}^2\setminus\{(0,0)\}$ as well as a Hardy function $h\in \CH$ and a tempered function $t\in \mathcal{T}$ satisfying some growth conditions. 
For example, it holds for the function 
\[x^{\pi}/ \log x+x^{1/2}(2+\cos\sqrt{\log x})\] which is of polynomial growth but is not Hardy.

Although dealing with tempered functions is outside the scope of this article, we expect Problem~\ref{Pr: joint ergodicity problem} to hold in the following general case.

\begin{conjecture}
Let $a_1, \ldots, a_\ell:\N\to\R$ be sequences of the form 
\[a_j=c_{1,j} h_j+c_{2,j}t_j\] for some $(c_{1,j},c_{2,j})\in \mathbb{R}^2\setminus\{(0,0)\},$ $h_j\in \CH,$ and $t_j\in \mathcal{T}$ such that for all distinct $i, j\in[\ell]$ and $c_i, c_j\in\R$, the function $c_i a_i - c_j a_j$ is either an almost rational polynomial or it stays logarithmically away from $\Q[x]$. Then the sequences $a_1, \ldots, a_\ell$ are jointly ergodic for a system $(X, \CX, \mu, T_1, \ldots, T_\ell)$ if and only  if they satisfy the product and difference ergodicity conditions on the system. 
\end{conjecture}

  \bibliography{library}
\bibliographystyle{plain}
\end{document}